\documentclass[10pt, a4paper]{amsart}

\usepackage{amssymb,amsmath,graphics,verbatim, bm}
\usepackage{latexsym}
\usepackage{eucal}

\usepackage[dvips]{graphicx}
\usepackage{pstricks,pst-node,pst-text,pst-3d,epsfig,pifont,pst-plot}

\makeatletter
\@addtoreset{equation}{section}\makeatother

\newtheorem{theo}{Theorem}[section]
\newtheorem{lem}[theo]{Lemma}
\newtheorem{prop}[theo]{Proposition}
\newtheorem{cor}[theo]{Corollary}

\theoremstyle{remark} \newtheorem{remark}[theo]{Remark}
\newtheorem{defn}[theo]{Definition}

\newcommand{\rr}{\mathbb{R}}
\newcommand{\nn}{\mathbb{N}}
\newcommand{\cc}{\mathbb{C}}

\newcommand{\zz}{\mathbb{Z}}
\newcommand\D{\mathcal{D}}

\newcommand{\la}{\lambda}
\newcommand{\eps}{\epsilon}

\newcommand{\pl}{\partial}
\newcommand{\x}{\times}

\newcommand{\supp}{\operatorname{supp}}

\newcommand{\demi}{\frac{1}{2}}

\newcommand{\indic}{\operatorname{1\negthinspace l}}

\newcommand{\zf}{\mathrm{zf}}
\newcommand{\bfo}{\mathrm{bf}_0}
\newcommand{\rbo}{\mathrm{rb}_0}
\newcommand{\lbo}{\mathrm{lb}_0}
\newcommand{\lb}{\mathrm{lb}}
\newcommand{\rb}{\mathrm{rb}}

\newcommand{\bfc}{\mathrm{bf}}
\newcommand{\sca}{\mathrm{sc}}

\newcommand\Id{\operatorname{Id}}

\newcommand\RR{\mathbb{R}}
\newcommand\CC{\mathbb{C}}
\newcommand\NN{\mathbb{N}}

\newcommand\MMksc{M^2_{k, \sca}}

\newcommand\Tkbstar{{}^{k, b} T^*}
\newcommand\MMkb{M^2_{k, b}}

\newcommand{\SC}{\ensuremath{\mathrm{sc}}}
\newcommand{\SF}{\ensuremath{\mathrm{s}\Phi}}
\newcommand{\Tsfstar}[1][\mbox{}]{{^{\SF}T^*_{#1}}}

\newcommand{\Tscstar}[1][\mbox{}]{{^{\SC}T^*_{#1}}}
\newcommand\Ndiag{\Nscstar Z}
\newcommand\Lsharp{L^\#}

\def\sfOh{{}^{\SF} \Omega^{1/2}}
\def\sfO{{}^{\SF} \Omega}

\renewcommand\Re{\operatorname{Re}}

\newcommand\CIdot{\dot C^\infty}

\newcommand\Lbf{L^{\bfc}}
\newcommand\MMb{M^2_b}
\newcommand\WF{\operatorname{WF}}
\newcommand{\Nscstar}[1][\mbox{}]{{^{\SC}N^*_{#1}}}

\newcommand\diagb{\text{diag}_b}
\newcommand\ddiagb{\partial \text{diag}_b}

\newcommand\sigmat{\tilde \sigma}
\newcommand\mf{\operatorname{mf}}
\newcommand\ybar{w_1}
\newcommand\ybbar{\overline{w}}
\newcommand\wbar{\overline{w}}
\newcommand\mubbar{\overline{\kappa}}
\newcommand\yybar{\overline{y}}
\newcommand\mmubar{\overline{\mu}}
\newcommand\zbar{\overline{z}}
\newcommand\zetabar{\overline{\zeta}}
\newcommand\oly{{\overline{y}}}
\newcommand\Mbar{K}
\newcommand{\Nsfstar}[1][\mbox{}]{{^{\SF}N^*_{#1}}}
\newcommand\CI{C^\infty}
\newcommand\Omegakb{\Omega_{k,b}}
\newcommand\Omegakbh{\Omega_{k,b}^{1/2}}

\newcommand\mcA{\mathcal{B}}
\newcommand\mcB{\mathcal{B}}
\newcommand\mcE{\mathcal{E}}
\newcommand\half{{\frac1{2}}}

\def\dbyd#1#2{\frac{\partial #1}{\partial #2}}

\newcommand\comp{\operatorname{comp}}
\newcommand\loc{\operatorname{loc}}
\newcommand\chievl{\chi_\lambda^{\mathrm{ev}}(\sqrt{\Delta_N})}
\newcommand\chiev{\chi_\lambda^{\mathrm{ev}}}

\newcommand\Lbar{\overline{\Lambda'}}
\newcommand\Llbbar{\overline{\Lambda_{\lb}}}
\newcommand\Lrbbar{\overline{\Lambda_{\rb}'}}
\newcommand\omegab{\bm{\omega}}
\newcommand\muh{\hat\mu_0}
\newcommand\Lambdabf{\Lambda^{\bfc}}
\newcommand\Gbf{G^{\bfc}}
\newcommand\bL{\mathbf{L}}
\newcommand\bX{\mathbf{X}}
\newcommand\SR{\mathrm{SR}}
\newcommand\bH{\mathbf{H}}

\begin{document}
\title[Restriction and spectral multiplier theorems]{Restriction and spectral multiplier theorems on asymptotically conic manifolds}
\author{Colin Guillarmou}
\address{DMA, U.M.R. 8553 CNRS\\
Ecole Normale Sup\'erieure\\
45 rue d'Ulm\\ 
F 75230 Paris cedex 05 \\France}
\email{cguillar@dma.ens.fr}
\author{Andrew Hassell}
\address{Department of Mathematics, Australian National University \\ Canberra ACT 0200 \\ AUSTRALIA}
\email{Andrew.Hassell@anu.edu.au}
\author{Adam Sikora}
\address{Department of Mathematics, Australian National University \\ Canberra ACT 0200 \\ AUSTRALIA, and \ Department of Mathematics, Macquarie University \\ NSW 2109 \\ AUSTRALIA}
\email{sikora@mq.edu.au}
\thanks{This research was supported by Australian Research Council Discovery grants DP0771826  and DP1095448 (A.H., A.S.) and a Future Fellowship (A.H.). C.G. is partially supported by ANR grant ANR-09-JCJC-0099-01 and by the PICS-CNRS Progress in Geometric
Analysis and Applications, and thanks the math dept of  ANU for 
its hospitality.}

\begin{abstract} 
The classical Stein-Tomas restriction theorem is equivalent to the statement that the spectral measure $dE(\lambda)$ of the square root of the Laplacian on $\RR^n$ is bounded from $L^p(\RR^n)$ to $L^{p'}(\RR^n)$ for $1 \leq p \leq 2(n+1)/(n+3)$, where $p'$ is the conjugate exponent to $p$, with operator norm scaling as $\lambda^{n(1/p - 1/p') - 1}$. We prove a geometric generalization in which the Laplacian on $\RR^n$ is replaced by the Laplacian, plus suitable potential, on a nontrapping asymptotically conic manifold, which is the first time such a result has been proven in the variable coefficient setting. It is closely related to, but stronger than, Sogge's discrete $L^2$ restriction theorem, which is an $O(\lambda^{n(1/p - 1/p') - 1})$ estimate on the $L^p \to L^{p'}$ operator norm of the spectral projection for a spectral window of fixed length. From this, we deduce spectral multiplier estimates for these operators, including Bochner-Riesz summability results, which are sharp for $p$ in the range above. 

The paper divides naturally into two parts. In the first part, we show at an abstract level that restriction estimates imply spectral multiplier estimates, and are implied by certain pointwise bounds on the Schwartz kernel of $\lambda$-derivatives of the spectral measure. In the second part, we prove such pointwise estimates for the spectral measure of the square root of Laplace-type operators on asymptotically conic manifolds. These are valid for all $\lambda > 0$ if the asymptotically conic manifold is nontrapping, and for small $\lambda$ in general. We also observe that Sogge's estimate on spectral projections is valid for any complete manifold with $C^\infty$ bounded geometry, and in particular for asymptotically conic manifolds (trapping or not), while by contrast, the operator norm on $dE(\lambda)$ may blow up exponentially as $\lambda \to \infty$ when trapping is present. This justifies the statement that the estimate on $dE(\lambda)$ is strictly stronger than Sogge's estimate. 
\end{abstract}
\maketitle
\tableofcontents

\section{Introduction}
The aim of this article is to prove some $L^p$ multiplier properties for the Laplacian, and a Stein-Tomas-type restriction theorem 
for its spectral measure, on a class of Riemannian manifolds which include metric perturbations of Euclidean space. 
One of the first natural questions in harmonic analysis is to understand the $L^p$ boundedness of Fourier 
multipliers $M$ on $\rr^n$, defined by
\[ M(f)(x)=\frac{1}{(2\pi)^n}\int_{\rr^n}e^{ix.\xi}m(\xi)\hat{f}(\xi)d\xi\]
where $m$ is a measurable function. Notice that for radial multipliers $m(\xi)=F(|\xi|)$, 
this amounts to study the $L^p$ boundedness of $F(\sqrt{\Delta})$ where $\Delta$ is the non-negative Laplacian. 
Of course, for $p=2$,  the necessary and sufficient condition on $m$ for $M$ to be bounded on $L^2$ is that $m\in L^\infty(\rr^n)$, but the case $p\not=2$ is much more difficult. The first results in this direction were given by Mikhlin \cite{Mik}: $M$ acts boundedly on $L^p(\rr^n))$ 
for all $1<p<\infty$ if
\[m\in C^\infty(\rr^n\setminus \{0\}), \, |\xi|^k |\nabla^k m(\xi)|\in L^\infty,\, \forall k,\, 0 \leq k \leq n/2+1,\]
and sharpened by H\"ormander \cite{Hor}, \cite[Th. 7.9.5]{Hor1}: let $\psi\in C_0^\infty(\demi,2)$ be not identically zero, 
then $M$ acts boundedly on $L^p(\rr^n))$ for all $1<p<\infty$ if
\[\sup_{t > 0}||m(t\cdot) \, \psi||_{H^s(\rr^n)}<\infty, \quad \frac{n}{2} < s \in \NN.
\]

More generally, let $\bL$ be a self-adjoint operator acting on $L^2$ of some measure space.  Using the spectral theorem, 
`spectral multipliers'  $F(\bL)$ can be defined for any bounded Borel  function $F$, and   act continuously on $L^2$. A question which 
has attracted a lot of attention during the last thirty years is to find some necessary conditions on the function 
$F$  to ensure that the operator $F(\bL)$ extends as a bounded operator for 
some range of $L^p$ spaces for $p \neq 2$. Probably the most natural and concrete examples are functions of the 
Laplacian on complete Riemannian manifolds, or functions 
of  Schr\"odinger operators with real potential $\Delta +V$, but  these problems are 
also studied for abstract self-adjoint operators. Some particular families of functions $F$ are also investigated in the 
theory of spectral multipliers: some of the most important examples include oscillatory integrals 
$e^{i(t\bL)^\alpha}(\Id+(t\bL)^\alpha)^{-\beta}$ and Bochner-Riesz means \eqref{br}. 
The subject of Bochner-Riesz means and spectral multipliers is so
broad that it is impossible to provide a comprehensive bibliography here, 
so we refer the  reader to the following  papers where further literature can be found
\cite{An, CS, ClSt, CSi, MM, MuSt, Sogge, SeeS, Sog, Tay, Tha}.  


The theory of Fourier multipliers and Bochner-Riesz analysis in this setting is  related to the so-called \emph{sphere restriction problem} for the Fourier transform:
find the pairs $(p, q)$  for which the \emph{sphere restriction operator} $\SR(\lambda)$, defined by 
\[\SR(\lambda) f(\omega):=\hat{f}(\la\omega), \, \, \omega\in S^{n-1}, \la >0,\] 
acts boundedly from $L^p(\rr^n)$ to $L^q(S^{n-1}))$.
See for example \cite{fef, fef2}. 
Of course, the dependence in $\la$  is trivial here since $\SR(\lambda) f=\la^{-n}\SR(1)(f(\la^{-1}\, \cdot))$ but 
this parameter $\la$ will be important later on.
There is a long list of results on this problem, but the first ones for general dimensions are due to Stein  and Tomas. The theorem of Tomas \cite{To}, improved by Stein \cite[Chapter IX, Section 2]{Stein} for the endpoint  $p=2\frac{n+1}{n+3}$ is the following:  $\SR(1)$ maps $L^p(\rr^n)$ boundedly to $L^q(S^{n-1}))$ if $p\leq 2\frac{n+1}{n+3}$ and $q\leq \frac{n-1}{n+1}\frac{p}{p-1}$ (notice that $q=2$ when $p$ reaches the endpoint). On the other hand, a necessary condition (based on the Knapp example) for boundedness is only given by  $p< 2\frac{n}{n+1}$ and this leads to the conjecture that 
$p< 2\frac{n}{n+1}$ and $q\leq \frac{n-1}{n+1}\frac{p}{p-1}$ is a necessary and sufficient condition. 
In fact, this has been shown by Zygmund \cite{Zy} in dimension $2$, improving a result of Fefferman \cite{fef} (by obtaining the endpoint estimate), but the conjecture is still open for $n>2$.  For more references and new results in this direction, we refer the interested reader to the survey by Tao \cite{Tao} on the subject. 

Like the $L^p$ multiplier problem,  the sphere restriction problem   has a corresponding natural generalization to 
certain types of manifolds (at least if we think of Fourier transform as a spectral diagonalisation for the Laplacian), and 
in particular those which have similar structure at infinity as  Euclidean space.
On $\rr^n$, the Schwartz kernel of the spectral measure $dE_{\sqrt{\Delta}}(\la)$ of $\sqrt{\Delta}$ is given by 
\[dE_{\sqrt{\Delta}}(\la;z,z')=\frac{\la^{n-1}}{(2\pi)^n}\int_{S^{n-1}}e^{i(z-z').\la\omega}d\omega, \quad z, z' \in \RR^n, \] 
therefore $dE_{\sqrt{\Delta}}(\la)=\frac{\la^{n-1}}{(2\pi)^n}\SR(\lambda)^*\SR(\lambda)$ and the restriction theorem for $q=2$ is equivalent to finding the largest $p<2$  such that $dE_{\sqrt{\Delta}}$ maps $L^p$ to $L^{p'}$. There is a natural class of Riemannian manifolds, called
\emph{scattering manifolds} or \emph{asymptotically conic manifolds}, for which  the spectral measure of the Laplacian
admits an analogous factorization. 
 Such manifolds, introduced by Melrose \cite{scatmet}, are by definition the interior $M^\circ$ of a compact manifold with boundary $M$, such that the metric $g$ is smooth on $M^\circ$ and has the form 
\begin{equation}
g=\frac{dx^2}{x^4}+\frac{h(x)}{x^2} 
\label{scatteringmetric}\end{equation}
in a collar neighbourhood near $\pl M$, where $x$ is a smooth boundary defining function for $M$ and $h(x)$ a smooth one-parameter family of metrics on $\pl M$; the function $r:=1/x$ near $x=0$ can be thought of as a radial coordinate near infinity and the metric is asymptotic to the exact conic metric $((0,\infty)_r\x \pl M, dr^2+r^2h(0))$ as $r \to \infty$. 
Associated to the Laplacian on such a manifold is the family of Poisson operators $P(\lambda)$ defined for $\lambda > 0$. These form a sort of distorted Fourier transform for the Laplacian: they map $L^2(\partial M)$ into the null space of $\Delta_g - \lambda^2$ and satisfy 
$dE_{\sqrt{\Delta_g}}(\la)=(2\pi)^{-1} P(\lambda) P(\lambda)^*$ \cite{HV1}. Thus $(\lambda/2\pi)^{-(n-1)/2} P(\lambda)^*$ is an analogue of the restriction operator in this setting.  The corresponding restriction problem is therefore to study the 
$L^p(M)\to L^q(\pl M)$ boundedness of $P(\lambda)^*$, and its norm in terms of the frequency $\la$ (the dependence of $P(\lambda)$ in $\la$ is no longer a scaling as it is for $\rr^n$).\\
   
The aim of the present work is to address these multiplier and restriction problems in the geometric setting of asymptotically conic manifolds. In fact, we shall first show, in an abstract setting, that restriction-type estimates on the spectral measure of an operator imply spectral multiplier results for that operator. Then we will prove such restriction estimates for a class of operators which are 0-th order perturbations of the Laplacian on asymptotically conic manifolds. 
In particular, our results cover the following settings:
\begin{itemize}
\item Schr\"odinger operators, i.e. $\Delta + V$ on $\RR^n$, where $V$ smooth and decaying sufficiently at infinity;

\item The Laplacian with respect to  metric perturbations of the flat metric on $\RR^n$, again decaying sufficiently at infinity;


\item The Laplacian on asymptotically conic manifolds. 
\end{itemize}

Our first main result is that restriction estimates imply spectral multiplier estimates:

\begin{theo}\label{bori}
Let $\bL$ be a non-negative self adjoint operator on $L^2(X,d\mu)$ where $(X,d,\mu)$ is a metric measure space
 such that the volume of balls satisfy the uniform bound $C_2>\mu(B(x,\rho))/\rho^n>C_1$ for some $C_2>C_1>0$. 
Suppose that the operator $\cos(t\sqrt{\bL})$ satisfies finite speed propagation property \eqref{fsp}, that the spectrum of $\bL$  
is absolutely continuous and that there exists $1\le p <2$ such that the spectral measure of $\bL$ satisfies 
\begin{equation}
\|d E_{\sqrt{\bL}}(\lambda)\|_{p \to p' }\le C\lambda^{n(1/p-1/p')-1},
\label{sp-meas-est}\end{equation}
where $p'$ is the exponent conjugate to $p$.  
Let $s > n(1/p - 1/2)$ be a Sobolev exponent. Then there exists $C$ depending only on $n, p$, $s$, and the constant in \eqref{re} such that, 
for every even $F \in H^s(\RR)$ supported in $[-1, 1]$, 
$F(\sqrt{\bL})$ maps $L^p(X) \to L^p(X)$, and 
\begin{equation}
\sup_{\alpha>0} \big\|F(\alpha\sqrt{\bL})\big\|_{p\to p} \leq  C\|F\|_{H^s}  .
\label{multip}\end{equation}
\end{theo}

\begin{remark} 
As noted above, the hypothesis \eqref{sp-meas-est} is valid on Euclidean space $\RR^n$ and for $1 \leq p \leq 2(n+1)/(n+3)$. In this case, the range of the Sobolev exponent above, $s > n(1/p - 1/2)$ is known to be \emph{sharp}; see \cite[Section IX.2]{Stein}. 
\end{remark}

In the second part of the paper, we prove \eqref{sp-meas-est} for the spectral measure of the Laplacian $\Delta_g$, plus a suitable potential,  on asymptotically conic manifolds.

\begin{theo}\label{main3} Let $(M,g)$ be an asymptotically conic manifold of dimension $n \geq 3$, and let $x$ be a smooth boundary defining function of 
$\pl M$. Let $\bH := \Delta_g+V$ be a Schr\"odinger operator 
on $M$, with $V\in x^3C^\infty(M)$, and assume that $\bH$ is a positive operator and that $0$ is neither an eigenvalue nor a resonance. Then:
\begin{itemize}
\item[(A)] 
  For any $\lambda_0 > 0$ there exists a constant $C>0$ such that the spectral measure $dE(\lambda)$
for $\sqrt{\bH}$ satisfies
\begin{equation}
\| dE_{\sqrt{\bH}}(\lambda) \|_{L^p(M) \to L^{p'}(M)} \leq C \lambda^{n(1/p - 1/p') - 1}
\label{restr}\end{equation}
for $1 \leq p \leq 2(n+1)/(n+3)$ and $0 < \lambda \leq \lambda_0$.\\
\item[(B)]  If $(M,g)$ is nontrapping, then there exists $C>0$ such that 
\eqref{restr} holds for all $\lambda > 0$. \\
\item[(C)]   If $(M,g)$ is trapping  and has asymptotically \textbf{Euclidean} ends,
there exists $\chi\in C_0^\infty(M^\circ)$ and $C>0$ such that 
\begin{equation}\label{cutoffest}
\|(1-\chi) dE_{\sqrt{\bH}}(\lambda) (1-\chi)\|_{L^p(M) \to L^{p'}(M)} \leq C \lambda^{n(1/p - 1/p') - 1}, \quad \forall \lambda\geq 0 
\end{equation}
for $1 < p \leq 2(n+1)/(n+3)$. 
However, \eqref{restr} need not hold for all $\lambda > 0$: there exist  (trapping) asymptotically Euclidean manifolds $(M,g)$,  sequences $\lambda_n \to \infty$ and $C, c > 0$  such that 
\begin{equation}
\| dE_{\sqrt{\Delta_g}}(\lambda_n) \|_{L^p(M) \to L^{p'}(M)} \geq C e^{c \lambda_n}. 
\label{restr-counterex}\end{equation}
\item[(D)]  On the other hand, the spectral projection estimate
\begin{equation}\label{integrated}
\big\| \indic_{[\lambda, \lambda + 1]}(\sqrt{\Delta_g}) \big\|_{L^p(M) \to L^{p'}(M)} \leq C \lambda^{n(1/p - 1/p') - 1}, \quad\forall \lambda \geq 1,
\end{equation}
holds for $1 \leq p \leq 2(n+1)/(n+3)$ for all asymptotically conic manifolds, trapping or not, and indeed for the much larger class of complete manifolds with $\CI$ bounded geometry. 
\end{itemize}
\end{theo}

Since $\bH$ in Theorem~\ref{main3} also satisfies the finite speed of propagation property \eqref{fsp}, we deduce from the two theorems above

\begin{cor}\label{maincor}
Let $\bL = \bH$, where $\bH$ is as in Theorem~\ref{main3}, and assume that $(M,g)$ in Theorem~\ref{main3} is nontrapping.   Then 
$\bL$ satisfies \eqref{multip}, where $F$ and $s$ are as in Theorem~\ref{bori} and $p \in [1, 2(n+1)/(n+3)]$. 
\end{cor}

\begin{figure}\label{figure}
\vspace{-6cm} 
\begin{pspicture}(-1.5,14)(12,5)
	\psset{xunit=9cm, yunit=2.7cm}
	\newgray{gray0}{.9}
	\newgray{gray1}{.7}
	\newgray{gray2}{.7}
	\newgray{gray3}{.6}
	\pspolygon[linestyle=none,fillstyle=solid,fillcolor=gray2](.5,0)(.2,.7)(0,1.5)(0,2.5)(1,2.5)(1,1.5)(.8,.7)
	\pspolygon[linestyle=none,fillstyle=solid,fillcolor=gray0](.5,0)(.8,.7)(.625,0)
\pspolygon[linestyle=none,fillstyle=solid,fillcolor=gray0](.5,0)(.2,.7)(.375,0)
	\pspolygon[linestyle=none,fillstyle=solid,fillcolor=gray0](.5,0)(.8,.7)(.625,0)
	\pspolygon[linestyle=none,fillstyle=solid,fillcolor=gray3](.5,0)(0,2)(0,2.8)(1,2.8)(1,2)
	\psline[linewidth=.5pt]{->}(0,0)(1.2,0)
	\uput[d](1.2,0){\Large$\frac{1}{p}$}
	\psline[linewidth=.5pt]{->}(0,0)(0,3)
	\uput[l](0,2.5){\Large $s$}
	\psline[linewidth=2.5pt](.375,0)(0,1.5)
	\psline[linewidth=1.5pt](.5,0)(0,2)
	\psline[linewidth=1.5pt](.5,0)(1,2)
\psline[linewidth=1pt](.5,0)(.8,.7)
\psline[linewidth=1pt](.5,0)(.2,.7)	
	\psline[linewidth=2.5pt](.625,0)(1,1.5)
	\psline[linestyle=dotted](.8,0)(.8,.7)
	\psline[linestyle=dotted](.2,0)(.2,.7)
	\psline[linewidth=.5pt](1,0)(1,1.5)
	\uput[l](0,1.95){\Large$\frac{n+1}{2}$}
	\uput[l](0.52,1.75){\Large A}
	\uput[l](0.18,1.15){\Large B}
	\uput[l](0.88,1.15){\Large B}
	\uput[l](0.18,0.45){\Large C}
	\uput[l](0.88,0.45){\Large C}
	\uput[l](0.65,0.15){\Large D}
	\uput[l](0.42,0.15){\Large D}
	\uput[l](0,1.5){\Large$\frac{n}{2}$}
	    \uput[l](-0.05,0){\Large$\frac{1}{2}$}
 	\uput[d](.5,0){\Large$\frac{1}{2}$}
	  \uput[d](0,0){\Large$0$}
	  \uput[d](1,0){\Large$1$}
	\uput[d](.625,0){\Large$\frac{n+1}{2n}$}
	\uput[d](.375,0){\Large$\frac{n-1}{2n}$}
	\uput[d](.8,0){\Large$\frac{n+3}{2(n+1)}$}
	\uput[d](.2,0){\Large$\frac{n-1}{2(n+1)}$}

\end{pspicture}
\vspace{6cm}
\caption{This figure shows where statement \eqref{multip} has been established on nontrapping asymptotically conic manifolds, for different values of $s$ and $p$. In region A this was previously known \cite{DOS}; see also Proposition~\ref{propsl}. In the present paper we establish \eqref{multip} also for region B (previously this was known only in the classical case of flat Euclidean space and the flat Laplacian). In region C it is known to be false, while region D is still unknown. For comparison with the Bochner-Riesz multiplier $F_\delta(\lambda) = (1 - \lambda^2)_+^\delta$ observe that $F_\delta$ is in $H^s$ for $s > \delta + 1/2$. 
For $F = F_\delta$,  part of region D is known for flat Euclidean space \cite{Lee}, 
and the celebrated Bochner-Riesz conjecture is that, for flat Euclidean space, \eqref{multip} is true for $F = F_\delta$ in the whole of region D.} \end{figure}
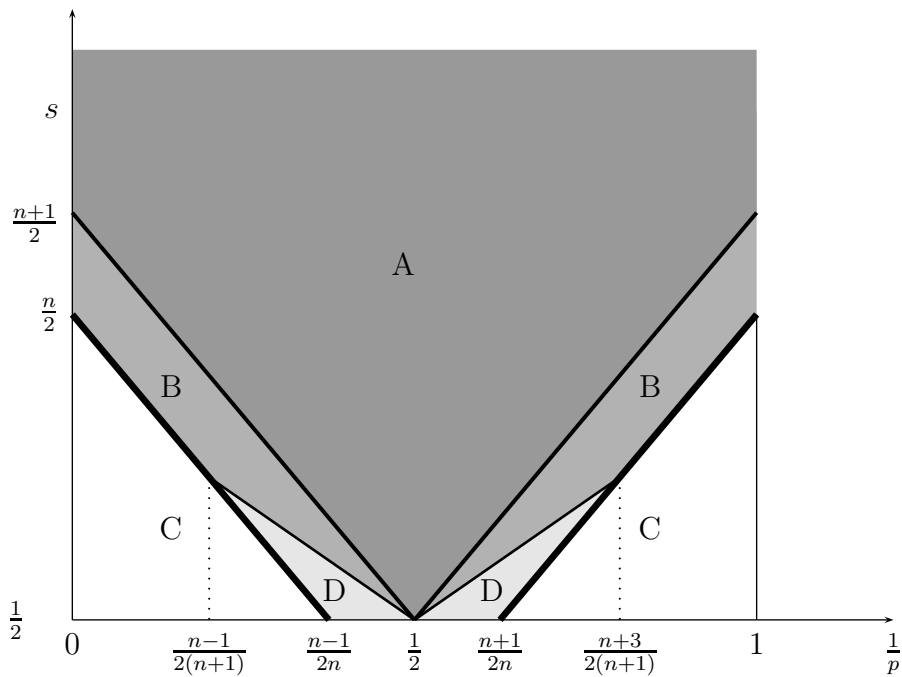

\begin{remark}
As far as we are aware, the restriction estimates for the spectral measure in 
Theorem \ref{main3} were previously known only for $\bH$ being the Laplacian in the Euclidean space $\rr^n$. As for the spectral multiplier result of Corollary~\ref{maincor}, this was previously known for $s > n(1/p - 1/2) + 1/2$ \cite{DOS}. 
Thus, for $p \in [1, 2(n+1)/(n+3)]$, we gain half a derivative over the best results previously known. The region in the $(1/p, s)$-plane in which we improve previous results is illustrated in Figure~\ref{figure}. 
The lower threshold of $n(1/p - 1/2)$ for the Sobolev exponent $s$ in Corollary~\ref{maincor} is known to be sharp in Euclidean space, and it is not hard to see that it is sharp for any asymptotically conic manifold. 
\end{remark}

\begin{remark} There are not many examples of sharp spectral multiplier results in the literature. Those known to the authors are as follows. 
The sharp multiplier theorem \eqref{multip} for $p=2\frac{n+1}{n+3}$ (the other $p$ are obtained by interpolation) was proved for the Laplacian on any compact manifold by Seeger-Sogge \cite{SeeS}. In fact, they only needed the integrated estimate \eqref{integrated} 
to obtain the multiplier theorem in that setting. 
In the setting of the twisted Laplacian operator
$ \Delta _x +\Delta _y +\frac14(\left\| x \right\|^2 + \left\| y \right\|^2)-i\sum_{j = 1}^n {(x_j\partial_{y_j}-y_j \partial _{x_j } )} $ the sharp multiplier theorem \eqref{multip} was proved by Stempak-Zienkiewicz \cite{StZi}. However, in this setting  the 
required form of restriction estimates has a form different from both \eqref{restr} and \eqref{integrated}, see \cite{KoRi}.
The last case of a sharp multiplier theorem known to the authors, although with a slightly different range of $p$, is for the
 harmonic oscillator and is described in \cite{Ka, KoTa, Tha}.
 \end{remark}

\begin{remark}
 Notice that the multiplier theorem of the type \eqref{multip} does not hold
for manifolds with exponentially volume growth  (like negatively curved complete manifolds); a \emph{necessary} condition on the multiplier $F$ in that case is typically a holomorphic extension of $F$ into a strip. See for instance the work of Clerc-Stein or Anker \cite{ClSt,An} for the case of non-compact symmetric spaces, or Taylor \cite{Tay} in the case of manifolds with bounded geometry, 
where sufficient conditions are also given.  
 \end{remark}

\begin{remark}
Theorems \ref{main3} and Theorem \ref{bori} imply Bochner-Riesz summability for a range of exponents similar to those proved for the Euclidean Laplacian in \cite[Proposition, p. 390]{Stein} and \cite[Theorem~2.3.1]{Sogge} and for compact manifolds by Christ-Sogge \cite{CS}, Sogge \cite{Sog}. See Corollary \ref{riesz} below.
\end{remark}


\begin{remark}
Probably the non-trapping condition is not necessary to obtain the estimate \eqref{restr} for all $\la>0$; it seems likely that asymptotically conic manifolds with a hyperbolic trapped set of sufficiently small dimension will also satisfy \eqref{restr}, by analogy with \cite{BGH}. However,  manifolds with elliptic trapping
will typically have sequences of $\lambda$ for which the norm on the left hand side of \eqref{restr} grows super-polynomially; see Section~\ref{trappingex}. 
\end{remark}

\begin{remark} 
The spatially cut-off estimate \eqref{cutoffest} can be compared to the 
non trapping $L^2$ estimate proved by Cardoso-Vodev \cite{CV} 
\[||(1-\chi)(L-\la^2+i0)^{-1}(1-\chi)||_{L_\alpha^2\to L_{-\alpha}^2}=O(\la^{-1}) , \quad \forall \la>1,\,\, \forall \alpha>\demi\] 
where $L^2_\alpha:=\langle r \rangle^{-\alpha}L^2(M)$. As a matter of fact, we use this estimate to prove \eqref{cutoffest}.
\end{remark}

The heuristics one can extract from Theorem \ref{main3}  Êand the last two remarks can be summarized as follows:
\begin{itemize}
\item  the sharp restriction estimate on $dE(\la)$ at bounded and low frequencies $\la$
only depends on the geometry near infinity;
\item  the high frequencies restriction estimate on $dE(\la)$ also depends strongly 
on global dynamical properties (trapping/non-trapping);
\item  the integrated estimate \eqref{integrated} for all frequencies $\la>1$ only 
depends on having uniform  local geometry.
\end{itemize}  

The proof of Theorem \ref{bori}, given in Section~\ref{reisme},  is based on a principle common to the proofs of most of Fourier and 
spectral multiplier theorems. 
The rough idea is that one can control the $L^p$ to $L^p$ norm of operators with singular integral kernels 
by estimating the $L^p$ to $L^q$ norm of the 
operator for some $q>p$ (usually $q=2$)  and showing that a large part of 
the corresponding kernel is concentrated near the diagonal. 
See for example \cite{fef, fef2, Sog, SeeS}. For calculations starting from $L^1 \to L^2$ estimates 
 this principle can be equivalently stated in terms of 
weighted $L^2$ norms of the kernel; see  \cite{Hor, MM, CSi}. 
Our implementation of this principle in the proof of Theorem \ref{bori} is based on finite speed propagation of
the wave equation, following \cite{CGT, CSi, Sik} for example. 
In the proof, we decompose  the operator $F(\alpha \sqrt{L})$ as a sum over $\ell\in\nn$ of multipliers $F_\ell(\alpha \sqrt{L})$ satisfying some finite speed propagation properties with $F_\ell$ Schwartz. The $L^p\to L^p$ norms for $F_\ell(\alpha \sqrt{L})$ are controlled by $C(\alpha 2^\ell)^{n(1/p-1/2)}$
times the  $L^p\to L^2$ norms and then the $TT^*$ argument reduces the problem to the bound of the $L^p\to L^{p'}$ norms of 
$|F_\ell|^2(\alpha\sqrt{L})$, which can be obtained using the restriction estimate of the spectral measure.\\

The proof of Theorem \ref{main3} proceeds in two steps. In the first step we suppose that we have an abstract operator $\bL$ whose spectral measure can be factorized as $dE_{\sqrt{\bL}}(\lambda) = (2\pi)^{-1} P(\lambda) P(\lambda)^*$ (cf. the discussion below \eqref{scatteringmetric}), where the initial space of $P(\lambda)$ is a Hilbert space. We then prove the following
result in Section~\ref{sec:keire}:

\begin{prop}\label{QQ1} Let $(X, d, \mu)$ and $\bL$ be as in Theorem~\ref{bori}, and assume in addition that $dE_{\sqrt{\bL}}(\lambda) = (2\pi)^{-1} P(\lambda) P(\lambda)^*$ as described above. Also assume that for each $\lambda$ we have an operator partition of unity on $L^2(X)$,
\begin{equation}
\Id = \sum_{i=1}^{N(\lambda)} Q_i(\lambda),
\label{opou}\end{equation}
where the $Q_i$ are uniformly bounded as operators on $L^2(X)$
and $N(\lambda)$ is uniformly bounded. We assume that for $1 \leq i \leq N(\lambda)$, and some nonnegative function $w(z,z')$ on $X \times X$, the estimate 
\begin{equation} 
\Big| \big( Q_i(\lambda)   dE_{\sqrt{\bH}}^{(j)}(\lambda) Q_i(\lambda) \big) (z,z') \Big| \leq C \lambda^{n-1-j} \big( 1 + \lambda w(z,z') \big)^{-(n-1)/2 + j} 
\label{spec-meas-j-1}\end{equation}
holds for $j = 0$ and for $j = n/2 -1$ and $j = n/2$ if $n$ is even, or for $j = n/2 - 3/2$ and $j = n/2 + 1/2$ if $n$ is odd. 
Here $dE_{\sqrt{\bL}}^{(j)}(\lambda)$ means $(d/d\lambda)^j dE_{\sqrt{\bL}}(\lambda)$, and $C$ is independent of $\lambda$ and $i$. 
Then restriction estimates \eqref{re} hold for all $p$ in the range $[1, 2(n+1)/(n+3)]$. Moreover, if the estimates above hold only for $0 < \lambda \leq \lambda_0$, then low energy restriction estimates \eqref{lere} hold for the same range of $p$. 
\end{prop}

The key point here is that we only need to consider operators 
$Q_i(\lambda) dE_{\sqrt{\bL}}^{(j)}(\lambda) Q_k(\lambda)$ for $i = k$, which effectively means that we only need to analyze the kernel of $dE_{\sqrt{\bL}}^{(j)}(\lambda)$ close to the diagonal. 
The proof of this 
is based on the complex interpolation idea of Stein \cite{St} and appears in Section~\ref{sec:keire}. 

The second step is to prove estimates \eqref{spec-meas-j-1} in the case where $\bL$ is the Laplacian or a Schr\"odinger operator on an asymptotically conic manifold. We show 

\begin{theo}\label{pointwisesmestimates} Let $(M,g)$ and $\bH$ be as in Theorem~\ref{main3}. Then there exists an operator partition of unity, \eqref{opou}, where the $Q_i$ are uniformly bounded as operators on $L^2(X)$
and  $N(\lambda)$ is uniformly bounded, such that the estimates 
\eqref{spec-meas-j-1}
hold for all integers $j \geq 0$ and for $0 < \lambda \leq \lambda_0$, with $w(z, z')$  the Riemannian distance between points $z, z' \in M^\circ$. Moreover, if $(M,g)$ is nontrapping, then 
estimates \eqref{spec-meas-j-1} hold for all $0 < \lambda < \infty$. 
\end{theo}

In the free Euclidean setting, this estimate is obvious (with the trivial partition of unity) by using the explicit formula of the spectral measure, but in our general setting it turns out to be quite involved and we really need to choose the partition of unity carefully. We use some results of 
 \cite{HV2} on the resolvent of $L$ on the spectrum, the high-energy (semi-classical) version of this  \cite{HW} and the low energy estimates of our previous work \cite{GHS}. These three articles on which we build our estimates describe the Schwartz kernel of the spectral measure as a Legendrian distribution (Fourier Integral Operator, in a sense) 
on a desingularized version of the compactification of the space $M\x M$, and this was done in a sort of uniform way with respect to the spectral parameter $\la$. The operators $Q_i$ in the partition of unity will be pseudodifferential operators of a particular sort; see Section~\ref{(A)} for the estimate \eqref{spec-meas-j-1} for small $\lambda$, and  Section~\ref{(B)} for the same estimate for large $\lambda$. By our discussion above, this establishes parts (A) and (B) of  Theorem~\ref{main3}. Part (C) of  Theorem~\ref{main3} is proved in Section~\ref{sec:slr} and part (D) is proved in Section~\ref{sec:spe}. 
\\

\textbf{Acknowledgements.} C.G. was supported by ANR grant  ANR-09-JCJC-0099-01 and is grateful to the ANU math department for its hospitality. A.H. was partially supported by ARC Discovery grants DP0771826 and DP1095448. A. S. was partially supported by ARC grant DP110102488. C.G. and A.H. thank MSRI for its hospitality during the 2008 workshop `Elliptic and Hyperbolic operators on singular spaces', where some of this research was carried out.  
We also thank S. Zelditch, J. Wunsch and G. Vodev for useful discussions.


\part{Abstract self-adjoint operators}

\section{Restriction estimates imply spectral multiplier estimates}\label{reisme}
Let $\bL$ be an abstract positive self-adjoint operator on $L^2(X)$, where $X$ is a metric measure space with metric $d$ and measure $\mu$. We make the following assumptions about $\bL$ and $(X, d, \mu)$:
\begin{itemize}
\item The space $X$ is separable and has dimension $n$ in the sense of the volume growth of balls: that is, there exist constants $0 < c_1 < c_2 < \infty$ such that
\begin{equation}
c_1 \rho^n \leq \mu(B(x, \rho)) \leq  c_2 \rho^n
\label{vol-balls}\end{equation}
for every $x \in X$ and $\rho > 0$;
\item $\cos(t \sqrt{\bL})$ satisfies finite speed propagation in the sense that
\begin{equation}
\supp \, \cos(t \sqrt{\bL}) \subset \mathcal{D}_{t} := \{ (z_1, z_2) \subset X \times X \mid d(z_1, z_2) \leq t \}.
\label{fsp}\end{equation}
The meaning of this statement is that $\langle f_1, \cos(t \sqrt{\bL}) f_2 \rangle = 0$ whenever $\supp f_1 \in B(z_1, \rho_1)$, $\supp f_2 \in B(z_2, \rho_2)$ and $|t|  + \rho_1 + \rho_2 \leq d(z_1, z_2)$. 
\item $\bL$ satisfies restriction estimates, which come in a strong and a weak form. We say that $\bL$ satisfies \emph{$L^p$ to $L^{p'}$  restriction estimates for all energies} if the spectral measure $dE_{\sqrt{\bL}}(\lambda)$ maps $L^p(X)$ to $L^{p'}(X)$ for some $p$ satisfying $1 \leq p < 2$ and all $\lambda >0$, with an operator norm estimate
\begin{equation}
\big\| dE_{\sqrt{\bL}}(\lambda) \big\|_{L^p(X) \to L^{p'}(X)} \leq C \lambda^{n(1/p - 1/p') - 1}, \text{ for all } \lambda > 0.
\label{re}\end{equation}
We also consider a weaker form of these estimates: we say that
$\bL$ satisfies \emph{low energy $L^p$  to $L^{p'}$  restriction estimates}
if $dE_{\sqrt{\bL}}(\lambda)$ maps $L^p(X)$ to $L^{p'}(X)$ for some $p$ satisfying $1 \leq p < 2$ and all $\lambda \in (0, \lambda_0]$, with an operator norm estimate
 \begin{equation}
\big\| dE_{\sqrt{\bL}}(\lambda) \big\|_{L^p(X) \to L^{p'}(X)} \leq C \lambda^{n(1/p - 1/p') - 1}, \ 0 < \lambda \leq \lambda_0
\label{lere}\end{equation}
for some $C$, together with weaker estimates for $\lambda \geq \lambda_0$:
\begin{equation}
\big\| E_{\sqrt{\bL}}[0 , \lambda] \big\|_{L^p(X) \to L^{p'}(X)} \leq C \lambda^{n(1/p - 1/p')}, \quad \lambda \geq \lambda_0
\label{re-integrated}\end{equation}
with a uniform $C$. (Here $E_{\sqrt{\bL}}[0 , \lambda]$ is the same as $\indic_{[0, \lambda]}(\sqrt{\bL})$.)
\end{itemize}

\begin{remark} The assumptions (with restriction estimates for all energies) are satisfied by taking $X = \RR^n$ with the standard metric and measure, and $\bL$ to be the (positive) Laplacian on $\RR^n$ (with domain $H^2(\RR^n)$). As we shall see, the assumptions are also satisfied for asymptotically conic manifolds, with the low energy restriction estimates holding unconditionally, and restriction estimates for all energies satisfied if the manifold is nontrapping.
\end{remark}  
\begin{remark} Clearly, \eqref{re-integrated} follows from \eqref{re} by integrating over the interval $[0, \lambda]$. However, in Remark~\ref{metricbottle} we give an example where we have
by Proposition~\ref{sp-proj-est}
  $$\| E_{\sqrt{\bL}}[\lambda , \lambda + 1] \|_{L^p(X) \to L^{p'}(X)} \leq C \lambda^{n(1/p - 1/p') - 1}, \quad \lambda \geq \lambda_0,$$ 
  (which implies \eqref{re-integrated}),  but the pointwise estimate on the $L^p \to L^{p'}$ operator norm of $dE(\lambda)$ grows exponentially for a subsequence of $\lambda$ tending to infinity.
\end{remark}

\begin{remark} The   spectral projection estimate \eqref{re-integrated}  is implied by a heat kernel bound
\begin{equation}
\big\| e^{-t\bL} \big\|_{L^p \to L^{p'}} \leq C t^{-n(1/p - 1/p')/2}, \quad t \leq \frac1{\lambda_0} .
\label{hkb}\end{equation}
This follows from short-time Gaussian bounds for the heat kernel, 
which hold for the Laplacian on any complete Riemannian manifold with bounded curvature and injectivity radius bounded below \cite[Theorem 4]{CLY}. 
This is proved by writing 
\begin{multline*}
  E_{\sqrt{\bL}}[0, \lambda]  =  E_{\sqrt{\bL}}[0, \lambda] \, e^{\bL/\lambda^2}  e^{-\bL/\lambda^2} \\ \implies \big\|  E_{\sqrt{\bL}}[0, \lambda] \big\|_{p \to 2}  \leq \big\|  E_{\sqrt{\bL}}[0, \lambda] e^{\bL/\lambda^2} \big\|_{2 \to 2} \big\| e^{-\bL/\lambda^2} \big\|_{p \to 2}.
\end{multline*}
Conversely, \eqref{re-integrated}  implies the heat kernel bound  \eqref{hkb}, which can be seen by writing $e^{-t\bL}$ as in integral over the spectral measure, and then integrating by parts. 
\end{remark}

\subsection{The main result}

The following theorem is the main result of this section.

\begin{theo}\label{thm:sm} Suppose that $(X, d, \mu)$ and $\bL$ satisfy \eqref{vol-balls} and \eqref{fsp}, and that $\bL$ satisfies  $L^p$ to $L^{p'}$ restriction estimates for all energies, \eqref{re}, for some $p$ with $1 \leq p < 2$. 
Let $s > n(1/p - 1/2)$ be a Sobolev exponent. Then there exists $C$ depending only on $n, p$, $s$, and the constant in \eqref{re} such that, 
for every even $F \in H^s(\RR)$ supported in $[-1, 1]$, 
$F(\sqrt{\bL})$ maps $L^p(X) \to L^p(X)$, and 
\begin{equation}
\sup_{\alpha>0} \big\|F(\alpha\sqrt{\bL})\big\|_{p\to p} \leq  C\|F\|_{H^s}  .
\label{sme}\end{equation}
If $\bL$ only satisfies the weaker estimates \eqref{lere}, \eqref{re-integrated}, i.e. low energy $L^p$ to $L^{p'}$  restriction estimates, then for all $F$ as above, we have 
\begin{equation}
\sup_{\alpha \geq 4/\lambda_0}\big\|F(\alpha\sqrt{\bL})\big\|_{p\to p} \leq  C\|F\|_{H^s}
\label{lesme}\end{equation}
where $C$ depends on $n$, $p$, $s$, $\lambda_0$, and the constants in \eqref{lere} and \eqref{re-integrated}. \end{theo}

\begin{remark} Notice that if $p > 2n/(n+1)$ then $s = 1/2$ satisfies $s> n(1/p - 1/2)$. However,  $H^{1/2}$ functions need not be bounded, and such functions cannot be $L^p$ multipliers even for $p = 2$, and a fortiori for $p \neq 2$. We deduce that, under the assumptions of  Theorem~\ref{thm:sm}, \eqref{re} or even \eqref{lere} is impossible for $p > 2n/(n+1)$. 
\end{remark}

In preparation for the proof of Theorem~\ref{thm:sm}, we have (following \cite{CGT})

\begin{lem}\label{step}
Assume that $\bL$ satisfies \eqref{fsp} and that
$F$ is an even bounded Borel function with  Fourier transform
 $\hat F$ satisfying
$\mbox{\rm supp}\; \hat{F} \subset [-\rho,\rho]$.
Then 
$$
\mbox{\rm supp}\; K_{F(\sqrt \bL)} \subset \D_\rho.
$$
\end{lem}
\begin{proof}
If $F$ is an even function, then by the Fourier inversion formula,
$$
F(\sqrt \bL) =\frac{1}{2\pi}\int_{-\infty}^{+\infty}
  \hat{F}(t) \cos(t\sqrt \bL) \;dt.
$$
But supp $\hat{F} \subset [-\rho,\rho]$
and Lemma \ref{step} follows from (\ref{fsp}). 
\end{proof}


The next lemma is a crucial tool in using restriction type results, i.e. $L^p \to L^q$ continuity of spectral projectors,   to obtain spectral
multiplier type bounds, i.e.  
$L^p \to L^p$ estimates. 

\begin{lem}\label{w}
Suppose that $(x, d, \mu)$ satisfies \eqref{vol-balls} and $S$ is an bounded linear operator from $L^p(X) \to L^q(X)$  such that 
$$
\supp S  \subset \D_{\rho}
$$
for some $\rho>0$. Then for any any $1\le p <q \le \infty $ there exists a constant $C=C_{p,q}$ such that
$$
\|S\|_{{p}\to {p}} \le C 
\rho^{n(1/p -1/q)}
\|S\|_{p\to q}.
$$
\end{lem}
\begin{proof}
We fix $\rho>0$. Then first we choose a sequence $x_n \in M$ such that
$d(x_i,x_j)> \rho/10$ for $i\neq j$ and $\sup_{x\in X}\inf_i d(x,x_i)
\le \rho/10$. Such sequence exists because $M$ is separable.
Second we define $\widetilde{B_i}$ by the formula
\begin{equation}\label{jadro}
\widetilde{B_i}=\bar{B}\left(x_i,\frac{\rho}{10}\right)-\left(\cup_{j<i}\bar{B}\left(x_j,\frac{\rho}{10}\right)\right),
\end{equation}
where $\bar{B}\left(x, \rho\right)=\{y\in M \colon d(z,z') \le \rho\}$. 
Third we put $\chi_i=\chi_{\widetilde B_i}$, where $\chi_{\widetilde B_i}$ is the characteristic function 
of set ${\widetilde B_i}$.  Fourth we define the operator
$M_{\chi_i}$ by the formula 
$M_{\chi_i}g=\chi_i g$.

Note that for $i\neq j$ 
 $B(x_i, \frac{\rho}{20}) \cap B(x_i, \frac{\rho}{20})=\emptyset$. Hence 
$$
K=\sup_i\#\{j;\;d(x_i,x_j)\le  2\rho\} \le
  \sup_x \frac{|\bar{B}(x, 2\rho)|}{\left|B\left(x, \frac{\rho}{20}\right)\right|}< \frac{40^nc_2}{c_1}< \infty.
$$
It is not difficult to see that 
$$
\D_{\rho}  \subset \cup_{\{i,j;\, d(x_i,x_j)<
 2 \rho\}} \widetilde{B}_i\times \widetilde{B}_j \subset \D_{4 \rho}
$$
so
$$
Sf =\sum_{{d}(x_i,x_j)< 2\rho} M_{\chi_i}S M_{\chi_j}f.
$$
Hence by H\"older inequality 
\begin{eqnarray*}
\|S f\|_{p}^p=\big\|\sum_{{d}(x_i,x_j)< 2\rho} M_{\chi_i}S M_{\chi_j}f\big\|_{L^p}^p=\sum_i \big\|\sum_{j;\,{d}(x_i,x_j)< 2\rho} M_{\chi_i}SM_{\chi_j}f\big\|_{p}^p   \\ \le 
\sum_i|\widetilde{B}_i|^{p(1/p-1/q)}\big\|\sum_{j;\,{d}(x_i,x_j)< 2\rho} M_{\chi_i}SM_{\chi_j}f\big\|_{q}^p
\\ \le C \rho^{np(1/p-/q)}  \sum_i\big\|\sum_{j;\,{d}(x_i,x_j)< 2\rho} M_{\chi_i}SM_{\chi_j}f\big\|_{q}^p
\\ \le CK^{p-1}  \rho^{np(1/p-1/q)}  \sum_i  \sum_{j;\,{d}(x_i,x_j)< 2\rho}   \big\|M_{\chi_i}SM_{\chi_j}f\big\|_{q}^p
\\ \le CK^{p}  \rho^{np(1/p-1/q)}  \sum_j\big\|SM_{\chi_j}f\big\|_{q}^p
\\ \le CK^{p} \rho^{np(1/p-1/q)}\|S\|_{p\to q}^p  \sum_j\big\|M_{\chi_j}f\big\|_{p}^p
\\=  CK^{p} \rho^{np(1/p-1/q)}\|S\|_{p\to q}^p  \|f\|_{p}^p
\end{eqnarray*}
This finishes the proof of Lemma~\ref{w}.
\end{proof}

\begin{proof}[Proof of Theorem~\ref{thm:sm}]
We first assume that $\bL$ satisfies $L^p$ to $L^{p'}$ restriction estimates for all energies. 
We take $\eta \in C_c^\infty(-4,4)$ even and such that 
$$
\sum_{n\in \zz} \eta\big( \frac{t}{2^l} \big)=1 \quad \text{for all } t \neq 0. 
$$
Then we set $\phi(t)=\sum_{l\le 0} \eta(2^{-l} t)$, 
$$
F_0(\lambda)=\frac{1}{2\pi}\int_{-\infty}^{+\infty}
 \phi(t) \hat{F}(t) \cos(t\lambda) \;dt
$$ 
and
\begin{equation}\label{defF_l}
F_l(\lambda) =\frac{1}{2\pi}\int_{-\infty}^{+\infty}
 \eta\big( \frac{t}{2^l} \big) \hat{F}(t) \cos(t\lambda) \;dt.
\end{equation}
Note that by virtue of the Fourier inversion formula 
$$
F(\lambda)=\sum_{l \ge 0}F_l(\lambda)
$$
and by Lemma \ref{step}
$$
{\rm supp}\,F_l(\alpha  \sqrt \bL) \subset \D_{2^l \alpha }.
$$
Now by Lemma \ref{w},
\begin{equation}\label{osz1}
\big\|F(\alpha  \sqrt \bL)\big\|_{p\to p } \le \sum_{l \ge 0}\big\|F_l( \alpha \sqrt {\bL})\big\|_{p\to p }
\le  \sum_{l \ge 0}   (2^l \alpha )^{n(1/p-1/2)}\big\|F_l(\alpha  \sqrt {\bL})\big\|_{p\to 2 }.
\end{equation}
Unfortunately, $F_l$ is no longer compactly supported. To remedy this  we choose a function $\psi \in C_c^\infty(-4, 4)$ such that $\psi(\lambda)=1$ for $\lambda \in (-2,2)$ 
and note that
$$
\big\|F_l(\alpha \sqrt {\bL})\big\|_{p\to 2 }\le \big\|(\psi F_l)(\alpha \sqrt {\bL})\big\|_{p\to 2 }+\big\|((1-\psi)F_l)(\alpha \sqrt {\bL})\big\|_{p\to 2 }.
$$
To estimate the norm $\|\psi F_l(\alpha \sqrt {\bL})\|_{p\to 2 }$ we use our restriction estimates (\ref{re}).
Using a $T^*T$ argument and the fact that $\supp \psi \subset [-4,4]$, we note that 
  \begin{eqnarray}
\big\|\psi F_l(\alpha \sqrt \bL)\big\|_{p\to 2}^2 = \big\||\psi F_l|^2( \alpha \sqrt \bL)\big\|_{p\to p' }\le \int_{0}^{4/\alpha } |\psi F_l(\alpha \lambda)|^2
\big\|dE_{\sqrt{\bL}}(\lambda)\big\|_{p \to p'}\, d\lambda \nonumber \\ \le \frac{C}{\alpha} \int_{0}^{4} |\psi F_l(\lambda)|^2
\big\|  dE_{\sqrt{\bL}}(\la/\alpha ) \big\|_{p \to p'} \,   d\lambda.\label{osz2}
\end{eqnarray}
It follows from the above calculation and \eqref{re} that 
\begin{equation}
\alpha^{n(1/p-1/2)}\big\|\psi F_l(\alpha \sqrt {\bL})\big\|_{p\to 2 }\le C\|\psi F_l\|_{2}, 
\label{psiFl-L2est}\end{equation}
for all $\alpha >0$. 
As a consequence, we obtain 
$$
\sum_{l \ge 0}   2^{ln(1/p-1/2)}\alpha ^{n(1/p-1/2)}\big\|\psi F_l(\alpha  \sqrt {\bL})\big\|_{p\to 2 }
\le \sum_{l \ge 0}   2^{ln(1/p-1/2)}\|\psi F_l\|_{2}
$$
for all $\alpha >0$.  
Now let us recall that by definition of Besov space
$$
 \sum_{l \ge 0}   2^{ln(1/p-1/2)}\|\psi F_l\|_{2}\le  \sum_{l \ge 0}   2^{ln(1/p-1/2)}\| F_l\|_{2}
 =\|F\|_{B_{1,2}^{n(1/p-1/2)}}.
$$
See, e.g., \cite[Chap.~I and~II]{Triebel}
for more details. We also recall that if $s>s'$ then $H_s \subset B_{1,2}^{s'}$ and 
$\|F\|_{B_{1,2}^{n(1/p-1/2)}}\le C_s \|F\|_{H^s}$ for all $s>  n(1/p-1/2)$, see again \cite{Triebel}. Therefore, we have shown that 
\begin{equation}
\sum_{l \ge 0}   2^{ln(1/p-1/2)}\alpha ^{n(1/p-1/2)}\big\|\psi F_l(\alpha  \sqrt {\bL})\big\|_{p\to 2 } \leq C \| F \|_{H^s}.
\label{psipart}\end{equation}

Next we obtain bounds for  the part of estimate \eqref{osz1} corresponding 
to the term $\|(1-\psi)F_l(\alpha \sqrt {\bL})\|_{p\to 2 }$. 
This only requires the spectral projection estimates \eqref{re-integrated}. We write 
\begin{equation*}\begin{gathered}
|(1 - \psi) F_l |^2(\alpha \sqrt \bL) = \int_0^\infty \big|(1-\psi)(\alpha \lambda) F_l(\alpha \lambda) \big|^2 dE_{\sqrt{\bL}}(\lambda) \\
= -\int_0^\infty \Big( \frac{d}{d\lambda}  \big| (1-\psi)(\alpha \lambda) F_l(\alpha \lambda) \big|^2 \Big) E_{\sqrt{\bL}}(\lambda) \, d\lambda \\
= -\int_0^\infty \Big( \frac{d}{d\lambda}  \big| (1-\psi)(\lambda) F_l( \lambda) \big|^2 \Big) E_{\sqrt{\bL}}(\lambda/\alpha) \, d\lambda. 
\end{gathered}\end{equation*}
Hence, using \eqref{re-integrated}, 
\begin{equation}
\big\|(1 - \psi) F_l(\alpha \sqrt \bL)\big\|_{p\to 2}^2  \le C \int_{0}^{\infty} \Big( \frac{d}{d\lambda}|(1-\psi)(\lambda) F_l(\lambda)|^2 \Big) \lambda^{n(1/p - 1/p')} \, d\lambda.
\label{zzz}\end{equation}
We write 
$$
F_l(\lambda) = \frac1{2\pi} \int e^{it(\lambda - \lambda')} \eta \big( \frac{t}{2^l} \big) F(\lambda') \, d\lambda' \, dt,
$$
use the identity $$e^{it(\lambda - \lambda')} = i^{-N}(\lambda - \lambda')^{-N} (d/dt)^N e^{it(\lambda - \lambda')},$$ 
and integrate by parts $N$ times. Note that if $\lambda \in \supp 1 - \psi$ and $\lambda' \in \supp F$ then $\lambda \geq 2$ and $\lambda' \leq 1$, and hence $\lambda - \lambda' \geq \lambda/2$.  It follows that 
$$
\big| ((1 - \psi)F_l)(\lambda) \big| \leq C \lambda^{-N} 2^{-l(N-1)} \| F \|_2
$$
with $C$ independent of $N$. Similarly,
$$
\big| \frac{d}{d\lambda} ((1 - \psi)F_l)(\lambda) \big| \leq C \lambda^{-N} 2^{-l(N-2)} \| F \|_2.
$$
Using this in \eqref{zzz} with $N$ sufficiently large, we obtain
$$
(2^l \alpha)^{n(1/p - 1/2)} \big\|((1 - \psi) F_l)(\alpha \sqrt \bL) \big\|_{p\to 2}
\leq C 2^{-l} \| F \|_2.
$$
Therefore, we have 
\begin{equation}
\sum_l  (2^l \alpha)^{n(1/p - 1/2)} \big\|((1 - \psi) F_l)(\alpha \sqrt \bL) \big\|_{p\to 2}
\leq C \| F \|_2 \leq C \| F \|_{H^s}. 
\label{1-psipart}\end{equation}
Equations \eqref{osz1},  \eqref{psipart} and \eqref{1-psipart} prove \eqref{sme}.

The proof in the case  that $\bL$ satisfies low-energy restriction estimates \eqref{lere} and \eqref{re-integrated} proceeds the same way, except that we require the condition $\alpha \leq 4/\lambda_0$ at the step \eqref{osz2}
in order that we can use the pointwise estimate \eqref{lere} on the spectral measure in this integral. 
\end{proof}

\begin{remark}
Note that if we only assume that \eqref{re-integrated} holds for all $\lambda >0$ then we still have

\begin{eqnarray*}
\alpha^{n(1/p-1/2)}\big\|\psi F_l(\alpha \sqrt {\bL})\big\|_{p\to 2 }\le \alpha^{n(1/p-1/2)}
\big\|\psi F_l(\alpha \sqrt {\bL})e^{\alpha^2\bL}\big\|_{2\to 2 }\big\|e^{-\alpha^2\bL}\big\|_{p\to 2 }\\
\le C\|\psi F_l\|_{\infty},
\end{eqnarray*}
Now the above estimate  is just a version of \eqref{psiFl-L2est} with  norm $\|\psi F_l\|_{2}$
replaced by $\|\psi F_l\|_{\infty}$. Next if we replace
Besov space $B_{1,2}^{n(1/p-1/2)}$ by $B_{1,\infty}^{n(1/p-1/2)}$ then  we can still follow the proof of Theorem~\ref{thm:sm}.
Recall also  that if $s>s'$ then $W^s_\infty \subset B_{1,\infty}^{s'}$ and
$\|F\|_{B_{1,\infty}^{n(1/p-1/2)}}\le C_s \|F\|_{W^s_\infty}$ for all $s>  n(1/p-1/2)$,
where $\|F\|_{W^s_\infty}=\|(I-d^2/dx^2)^{s/2}F\|_\infty$; see again \cite{Triebel}.
This implies that \eqref{psipart} holds with the norm $\| F \|_{H^s}$ replaced by
the norm $\|F\|_{W^s_\infty}$. 
As the rest of the proof of Theorem~\ref{thm:sm} does not require \eqref{re}, the above argument proves the following
proposition.

\begin{prop}\label{propsl} Suppose that $(X, d, \mu)$ and $\bL$ satisfy \eqref{vol-balls} and \eqref{fsp}, and that $\bL$ satisfies
 \eqref{re-integrated} for all $\lambda >0$.
Let $s > n|1/p - 1/2|$ be a Sobolev exponent. Then there exists $C$ depending only on $n, p$, $s$, and the constant in \eqref{re-integrated} such that,
for every even $F \in W^s_\infty (\RR)$ supported in $[-1, 1]$,
$F(\sqrt{\bL})$ maps $L^p(X) \to L^p(X)$, and
\begin{equation}
\sup_{\alpha>0} \big\|F(\alpha\sqrt{\bL})\big\|_{p\to p} \leq  C\|F\|_{W^s_\infty} .
\label{winftysm}\end{equation}
 \end{prop}
 
 Note also that if  $s>s'$ then    $\|F\|_{W^{s'}_\infty} \le C \| F \|_{H^{s+1/2}}$.
 That is, the multiplier result with exponent one half bigger then the optimal exponent does not require \eqref{re} and holds just under
 assumption \eqref{re-integrated}, which is equivalent with the standard heat kernel bounds \eqref{hkb} (for all $t$).
 For $p=1$ Proposition~\ref{propsl} was proved in \cite{CS} and can be alternatively proved using \cite[Theorem 3.5] {CS}
 and interpolation, see also \cite[Theorem 3.1]{DOS}.
 
 From this point of view, the key point about Theorem~\ref{thm:sm} is the gain of half a derivative over the more elementary \eqref{winftysm}. 
\end{remark}

\subsection{Bochner-Riesz summability}

We use Theorem~\ref{thm:sm} to discuss boundedness of Bochner-Riesz means of the operator $\bL$. Bochner-Riesz summability is technically speaking a slight weakening of Theorem~\ref{thm:sm} but is very close, and it allows us to compare our results with results described in \cite{Stein} and \cite{Sogge}. Let us recall that Bochner-Riesz means of order $\delta$ are defined by the formula 
\begin{equation}
(1 - \bL/\lambda^2)_+^\delta, \quad \lambda > 0. 
\label{br}\end{equation}
For $\delta = 0$, this is the spectral projector $E_{\sqrt{\bL}}([0, \lambda])$, while for $\delta > 0$ we think of \eqref{br} as a smoothed version of this spectral projector; the larger $\delta$, the more smoothing. 
Bochner Riesz summability describes the range of $\delta$ for which the above operators are bounded on $L^p$ uniformly in $\lambda.$ 

\begin{cor}\label{riesz}
Suppose that $(X, d, \mu)$ is as above, and that restriction estimates \eqref{re} for $1 \leq p \leq \frac{2(n+1)}{n+3}$ and finite speed propagation property \eqref{fsp} hold for operator $\bL$. 
Then for all $p\in [1,\frac{2(n+1)}{n+3}]\cup  [\frac{2(n+1)}{n-1},\infty]$ and all $\delta>n|1/p-1/2|-1/2$, we have
\begin{equation}
\| (1 - \bL/\lambda^2)_+^\delta \|_{p\to p} \le C \text{ for all } \lambda > 0.
\label{brs}\end{equation}
For all $p\in (\frac{2(n+1)}{n+3},\frac{2(n+1)}{n-1})$ 
the above estimates hold if $\delta>\frac{n-1}{2}|1/p-1/2|$. 
 \end{cor}
\begin{proof}
Note that $(1-\lambda^2)_+^\delta \in H^s$ if and only if $\delta>s-1/2 $.  Now for $p<\frac{2(n+1)}{n+3}$ Corollary~\ref{riesz} follows from Theorem~\ref{thm:sm}. For $\frac{2(n+1)}{n+3} < p < 2$ Corollary~\ref{riesz} follows from interpolating between \eqref{brs} with $p = \frac{2(n+1)}{n+3}$, and the trivial estimate for $p=2$. For $p > 2$ the results follow by duality. 
\end{proof}

\begin{remark} We noted in the proof above that Corollary~\ref{riesz} follows from Theorem~\ref{thm:sm}. In fact 
the Corollary~\ref{riesz} is slightly but essentially weaker than Theorem~\ref{thm:sm}. Indeed Corollary~\ref{riesz}
is equivalent with Theorem~\ref{thm:sm} in which the $H^s$ norm is replaced by the Sobolev $W^{s+1/2}_1$ norm. Let us recall that $\|F\|_{W^s_1}=\|(-d/dx+I)^s F\|_1$. To prove it we note that
$$
F(\sqrt{\bL})=\int \chi_+^{\nu}(\lambda -\sqrt{\bL})F^{\nu}(\lambda)d\lambda,
$$
where $\chi_+$ is as in Section~\ref{sec:keire} and $F^{\mu}=F*\chi_+^{-\mu}$. Now $\| F^{s}\|_1 \le C \|F\|_{W^{s'}_1}$ for all $s<s'$ and so 
Bochner-Riesz summability of order $a$ implies Theorem~\ref{thm:sm} with the Sobolev norm $W^{s+1}_1$
for all $s>a$. Note that for compactly supported functions $F$ which we consider here the norm $W^{s+1/2}_1$
is essentially stronger the  $H^s=W^{s}_2$ norm. Note also vice versa, Theorem~\ref{thm:sm} with the Sobolev norm $W^{s+1}_1$ implies Bochner-Riesz summability of order $a$ for all $a>s$. 
\end{remark}

\subsection{Singular integrals}

Finally we will discuss a singular integral version of our spectral multiplier result. The following theorem is just reformulation of \cite[Theorem~3.5 ]{CSi}. We write $D_\kappa$ for the scaling operator $D_\kappa F(x) = F(\kappa x)$. 
   \begin{theo}\label{singint}
  Suppose that operator $\bL$ satisfies finite speed propagation property \eqref{fsp}, that $s>n/2$ and that
  \begin{equation}\label{spec}
 \| dE_{\sqrt{\bL}}(\lambda)\|_{1 \to \infty} \le \lambda^{n-1} \text{ for all } \lambda > 0.
  \end{equation}
 Next let $\eta$ be a smooth compactly supported non-zero function.   Then  for any Borel bounded function
   $F$ such that  $\sup\limits_{ \kappa > 0} \Vert \eta \, D_\kappa F
\Vert_{W^p_s}<\infty$
   the operator $F(A)$ is of weak type $(1,1)$ and is bounded on $L^q(X)$
   for all $1<q<\infty$. In addition
     \begin{equation}\label{hor:m1}
     \|F({A})\|_{L^1\to  L^{1,\infty}}
     \le  C_s  \Big(\sup_{ \kappa >   0} \Vert \eta
     \, D_\kappa F \Vert_{W^p_s}+|F(0)|\Big).
     \end{equation}
   \end{theo}

\begin{remark}
It is a standard observation  that up to equivalence the norm  $$ \sup_{\kappa >   0} \Vert \eta
     \, D_\kappa F \Vert_{W^p_s}$$  does not depend on the auxiliary function $\eta$ as long as $\eta$ is not identically equal zero. 
\end{remark}
     
  \begin{proof}
  Using $T^*T$ trick we note that  by \eqref{spec} one has 
  \begin{eqnarray*}
\|F(\sqrt \bL)\|_{1\to 2}^2 = \||F|^2(\sqrt \bL)\|_{1\to \infty }\le \int_{0}^\infty |F(\lambda)|^2
\|dE_{\sqrt{\bL}}(\lambda)\|_{1 \to \infty}d\lambda \\ \le C \int_{0}^\infty |F(\lambda)|^2
\lambda^{n-1}d\lambda.
\end{eqnarray*}
Hence if $\supp F \subset [0,R)$ then 
$$
\|F(\sqrt \bL)\|_{1\to 2}^2\le CR^n \| D_R F\|^2_{2};
$$
that is the estimates (3.22) of Theorem~3.5 of \cite{CSi} hold. Now Theorem~\ref{singint} follows from 
\cite[Theorem~3.5]{CSi}. 
  \end{proof}
 \begin{remark}  Theorem~\ref{singint} is a singular integral version of Theorem~\ref{thm:sm} for $p=1$.
  We expect that a similar extension to a singular integral version is possible for all $p$. That is if one 
  assumes that $s>n|1/2-1/p|$ then one can prove weak-type $(p,p)$  version of estimates  \eqref{hor:m1}.
  However the proof of such results seems to be more complex and not directly related to the rest of this paper, 
  so we will not pursue this idea further here. 
  \end{remark}  
  

\section{Kernel estimates imply restriction estimates}\label{sec:keire}

The goal of this section is to prove Proposition~\ref{QQ1}; that is, we show that  restriction estimates \eqref{re} or \eqref{lere} follow from certain pointwise estimates of $\lambda$-derivatives of the kernel of the spectral measure. To  the proof of this proposition, we first prove a simplified version in which the partition of unity does not appear. We work in the same abstract setting as the previous section. 

\begin{prop}\label{keire} Let $(X, d, \mu)$ be a metric measure space and  $\bL$ an abstract positive self-adjoint operator on $L^2(X, \mu)$. Assume that the spectral measure $dE_{\sqrt{\bL}}(\lambda)$ for $\sqrt{\bL}$ has 
a Schwartz kernel $dE_{\sqrt{\bL}}(\lambda)(z,z')$ that satisfies, for some nonnegative function $w$ on $X \times X$, the following estimate
\begin{equation} 
\Big| \big( \frac{d}{d\lambda} \big)^{j}  dE_{\sqrt{\bL}}(\lambda)  (z,z') \Big| \leq C \lambda^{n-1-j} \big( 1 + \lambda w(z,z') \big)^{-(n-1)/2 + j} 
\label{spec-meas-j-1-3}\end{equation}
holds for $j = 0$ and for $j = n/2 -1$ and $j = n/2$ if $n$ is even, or for $j = n/2 - 3/2$ and $j = n/2 + 1/2$ if $n$ is odd. 
Then \eqref{re} holds for all $p$ in the range $[1, 2(n+1)/(n+3)]$. Moreover, if the estimates above hold only for $0 < \lambda < \lambda_0$, then \eqref{lere} hold for the same range of $p$. 
\end{prop}

We prove this proposition via complex interpolation, 
embedding the derivatives of the spectral measure in an analytic family of operators --- following the original (unpublished) proof of Stein in the classical case. To do this we use the distributions $\chi_+^a$, defined by
$$
\chi_+^a=\frac{x_+^a}{\Gamma(a+1)},
$$
where $\Gamma$ is the Gamma function and 
$$
x_+^a=x^a \quad \mbox{if} \quad x \ge 0 \quad \quad \mbox{and}\quad  x_+^a=0 \quad  \mbox{if} \quad x < 0.
$$
The $x_+^a$ are clearly distributions for $\Re a > -1$, and we have 
for $\Re a > 0$,
\begin{equation}\label{chi}
\frac{d}{dx} x_+^a = a x_+^{a-1} \implies \frac{d}{dx} \chi_+^a=\chi_+^{a-1}
\end{equation}
which we use to extend the family of functions $ \chi_+^a$ to a family of distributions on $\rr$ defined for all $a\in \cc$; see \cite{Hor1} for details.  Since $\chi_+^0(x) = H(x)$ is the Heaviside function, it follows that 
\begin{equation}
\chi_+^{-k} =\delta_0^{(k-1)}, \quad  k=1,2,\ldots ,
\label{Heaviside}\end{equation}
and therefore 
$$
\chi_+^0(\lambda - \sqrt{\bL}) = E_{\sqrt{\bL}}((0, \lambda]), \text{ and } 
\chi_+^{-k} = \big( \frac{d}{d\lambda} \big)^{k-1} dE_{\sqrt{\bL}}(\lambda), \quad k \geq 1. 
$$

A standard computation shows that for all $w,z\in \cc$
\begin{equation}\label{aa2}
\chi_+^w * \chi_+^z=\chi_+^{w+z+1}
\end{equation}
where $\chi_+^w* \chi_+^z$ is the convolution of the distributions $\chi_+^w$ and $\chi_+^z$
see \cite[(3.4.10)]{Hor1}. We can use this relation to \emph{define} the operators $\chi_+^z(\lambda - \sqrt{\bL})$ for $\Re z < 0$, provided that the spectral measure of $\sqrt{\bL}$ satisfies estimates of the type in Proposition~\ref{keire}:

\begin{defn} Suppose that $X$, $\bL$ and $w$ are as in Proposition~\ref{keire}, and that $\bL$ satisfies the kernel estimate
\begin{equation}
\Big| \big( \frac{d}{d\lambda} \big)^{k} dE_{\sqrt{\bL}}(\lambda)(z,z') \Big| \leq C \lambda^{l} \big( 1 + \lambda w(z,z') \big)^{\beta} 
\label{wzz}\end{equation}
for some $k \geq 0$, $l \geq 0$ and $\beta$. Then, for $-(k+1) < \Re a < 0$  we define the operator $\chi_+^a(\lambda - \sqrt{\bL})$ to be that operator with kernel
\begin{multline}
\chi_+^{k+a} * \chi_+^{-(k+1)}(\lambda - \sqrt(\bL))(z,z') \\ = 
(-1)^k \int_0^{\lambda} \frac{\sigma^{k+a}}{\Gamma(k+a+1)} 
\big( \frac{d}{d\sigma} \big)^{k} dE_{\sqrt{\bL}}(\lambda - \sigma)(z,z') \, d\sigma.
\label{chiplusadefn}\end{multline}
\end{defn}

Notice that the integral converges, since $\Re (k+a) > -1$ and $l \geq 0$ in \eqref{wzz}. It is also independent of the choice of integer $k > -\Re a - 1$ (provided \eqref{wzz} holds), as we check by integrating by parts in $\sigma$ in the integral above, and using \eqref{chi}. Note that the kernel $\chi_+^a(\lambda - \sqrt{\bL})(z,z')$ is analytic in $a$, and as an integral operator maps $L^1_{\comp}(X)$ to $L^\infty_{\loc}(X)$. Therefore, for each fixed $\lambda > 0$, the family $\chi_+^a(\lambda - \sqrt{\bL})$ is an analytic family of operators in the sense of Stein \cite{St} in the parameter $a$, for $\Re a > -k$. 

In the proof of Proposition~\ref{keire} we will need the following 

\begin{lem}\label{infty}
Suppose that $k \in \nn$, that $-k < a<b<c$ and that $b=\theta a +(1-\theta)c$. Then there exists a constant  $C$ such that  for any $C^{k-1}$ function $f \colon \rr \to \cc$ with compact support,  one has
$$
\| \chi_+^{b+is}*f\|_\infty \le C(1+|s|)e^{\pi |s|/2} \| \chi_+^{a}*f\|^\theta_\infty  \| \chi_+^{c}*f\|^{1-\theta}_\infty
$$
for all $s\in \rr$. 
\end{lem}

\begin{remark} The convolution $\chi_+^a * f$, for $a > -k$ and $f \in C_c^{k-1}(\RR)$, may be defined to be $\chi_+^{a+k-1} * f^{(k-1)}$; this is independent of the choice of $k$. 
\end{remark}

\begin{proof}
Set, for $\zeta \in \CC$, 
$$
I_\zeta f=\chi_+^\zeta *f
$$
and consider the operator $I_{{b}+is}(\sigma I_c+I_a)^{-1}$, where number $\sigma \in \cc$ such that $|\sigma|=1$ 
will be specified later. By \eqref{aa2}
$$
I_{{b}+is}(\sigma I_c+ I_a)^{-1}=I_{{\beta}+is}(\sigma I_{-1}+I_\alpha)^{-1}=I_{{\beta}+is}
(\sigma I+ I_\alpha)^{-1},
$$
where $\beta=b-c-1 $ and $\alpha=a-c-1$. Note that $ \alpha <\beta<-1$. 
A standard calculation \cite[Example 7.1.17, p. 167 and (3.2.9) p. 72]{Hor1} shows that for $\Re \zeta \le -1$
$$
\widehat{\chi_+^{\zeta}}(\xi)= e^{-i\pi(\zeta+1)/2} (\xi - i0)^{-\zeta-1}. 
$$
It follows that $I_{{\beta}+is}(\sigma I+I_\alpha)^{-1}f=f*\eta_s $, where $\widehat{\eta_s}$ is the locally integrable function
$$
\widehat{\eta_s(\xi)}=\frac{-ie^{-i\pi ({{\beta}}+is)/2}\xi_+^{-({{\beta}}+is) -1} +ie^{{i\pi ({{\beta}}+is)/2}}\xi_-^{-({{\beta}}+is)-1}}{\sigma - ie^{-i\pi {{\alpha}}/2}\xi_+^{-{{\alpha}}-1} + ie^{i\pi {{\alpha}}/2}\xi_-^{-{{\alpha}}-1} }.
$$
Here   $\xi_+=\max(0,\xi)$ and  $\xi_-=-\min(0,\xi)$. 
Note that if $|\sigma|=1$ and $\sigma \notin \{ ie^{-i\pi {{\alpha}}/2}, -ie^{-i\pi {{\alpha}}/2} \}$ then 
$$
\left| \frac{d}{d\xi}\widehat{\eta_s(\xi)}\right| \le C (1+|s|)e^{\pi |s|/2}\min \left(|\xi|^{-\beta-2},|\xi|^{-\beta+\alpha-1}\right)
$$
and $-\beta+\alpha-1  <  -1<-\beta-2$. 
It follows from the above estimates that the function $\frac{d}{d\xi}\widehat{\eta_s}$ is in  $L^p(\rr)$ space 
for some $1 < p <2$ and is also in some weighted space $L^1((1+|x|)^\epsilon dx, \rr)$. By the Sobolev embedding  and Hausdorff-Young
theorems, the function $x \to x\eta_s(x)$ is in $L^{p'}(\rr)$ for the conjugate exponent $ p'<\infty$  and in $ C^{\epsilon'}(\rr)$ for some $\epsilon' >0$. Hence $\eta_s$  is in $L^1$ and we have
$$
\|{\eta_s}\|_1 \le C(1+|s|)e^{\pi |s|/2}.
$$
Hence the operator  $ I_{{b}+is}(\sigma I_c+ I_a)^{-1}=I_{{\beta}+is}
(\sigma I+ I_\alpha)^{-1} $ is bounded on $L^\infty(\rr)$ and 
\begin{eqnarray*}
\|I_{{b}+is}f\|_\infty \le C(1+|s|)e^{\pi |s|/2}\|\sigma I_cf+ I_{{a}}f\|_\infty
\\ \le C(1+|s|)e^{\pi |s|/2}   \left( \|I_cf\|_\infty+ \|I_{{a}}f\|_\infty\right).
\end{eqnarray*}
Now if we set $D_{\kappa} f(x)=f(\kappa x)$ then for all $\zeta\in \cc$ 
$$
I_\zeta D_{\kappa}f=\kappa ^{-\zeta-1}D_{\kappa}I_\zeta f
$$
so  
$$
\kappa ^{{-b}}\|I_{{b}+is}f\|_\infty=\kappa ^{{-b}}\|D_{\kappa} I_{{b}+is}f\|_\infty =\kappa \|I_{{b}+is}D_{\kappa}f\|_\infty.
$$
Hence 
\begin{eqnarray*}
\kappa ^{{-b}}\|I_{{b}+is}f\|_\infty=\kappa \|I_{{b}+is}D_{\kappa}f\|_\infty \\
\le C(1+|s|)e^{\pi |s|/2}\left(\kappa \|I_a (D_{\kappa} f)\|_\infty+ \kappa \| I_c(D_{\kappa} f)  \|_\infty\right)\\
= C(1+|s|)e^{\pi |s|/2}\left(\kappa ^{-a}\|I_a f\|_\infty+ \kappa ^{-c}\| I_cf  \|_\infty\right)
\end{eqnarray*}
Putting  $\kappa ^{a-c}=\|I_a f\|_\infty\|I_c f\|_\infty^{-1}$ in the above estimate yields Lemma  \ref{infty}. 
\end{proof}

\begin{proof}[Proof of Proposition~\ref{keire}]
To prove \eqref{re} in the range $1 \leq p \leq 2(n+1)/(n+3)$, it suffices by interpolation to establish the result for the endpoints $p = 1$ and $p = 2(n+1)/(n+3)$. The endpoint $p=1$ is precisely  \eqref{spec-meas-j-1} for $j=0$, so it remains to obtain the endpoint $p = 2(n+1)/(n+3)$. This we will obtain through complex interpolation, applied to the analytic (in the parameter $a$) family $\chi_+^a(\lambda - \sqrt{\bL})$ in the strip $-(n+1)/2 \leq \Re a \leq 0$. 

On the line $\Re a = 0$, we have the estimate
$$
\Big\| \chi^{is}(\lambda - \sqrt{\bL}) \Big\|_{L^2 \to L^2} \leq \Big| \frac1{\Gamma(1 + is)} \Big| = \sqrt{\frac{\sinh \pi s}{\pi s}} \leq C e^{\pi |s|/2}. 
$$
On the line $\Re a = -(n+1)/2$, we will prove an estimate of the form 
\begin{equation}
\Big\| \chi^{-(n+1)/2 + is}(\lambda - \sqrt{\bL}) \Big\|_{L^1 \to L^\infty} \leq C (1 + |s|) e^{\pi |s|/2} \lambda^{(n-1)/2} \text{ for all } s \in \RR.
\label{left-est}\end{equation}
Then, since we can write
$$
dE_{\sqrt{\bL}}(\lambda) = \chi_+^{-1}(\lambda - \sqrt{\bL})
$$
and 
\begin{align*}
-1 &= \big( \frac{n-1}{n+1} \big) \big( 0 \big) \,  + \big( \frac{2}{n+1} \big) \big( \frac{n+1}{2} \big),  \\
\frac{n+3}{2(n+1)} &= \big( \frac{n-1}{n+1} \big) \big( \frac1{2} \big) + \big( \frac{2}{n+1} \big) \big( 1 \big),
\end{align*}
we obtain \eqref{re} at $p = 2(n+1)/(n+3)$ by complex interpolation. 

It remains to prove \eqref{left-est}. 
Let $\eta \in C_c^{\infty}(\rr)$ be such a function that $0 \le \eta(x) \le 1$ for all $x\in \rr$ and 
$\eta(x)=1$ for $|x| \le 2$ and $\eta(x)=0 $ for  $|x| \ge 4$.
Set 
$$\begin{gathered}
F^{s,\Lambda}_{z,z'}(\lambda) ={\chi_+^{-\frac{3}{2}- is}}* \Big( \eta(\cdot/\Lambda){\chi_+^{-k}(\cdot - \sqrt{\bL})}(z,z') \Big)(\lambda), \quad n=2k \phantom{ + 1 .} \\
F^{s,\Lambda}_{z,z'}(\lambda) ={\chi_+^{-2- is}}* \Big(\eta(\cdot/\Lambda){\chi_+^{-k}(\cdot - \sqrt{\bL})}(z,z') \Big)(\lambda), \quad n=2k+1.
\end{gathered}$$
Note that  supp$\chi_+^z \subset [0,\infty)$ for all $z$, and $\bL \geq 0$. It follows that for $\lambda \le \Lambda$ and $n=2k$,
\begin{eqnarray*}
F^{s,\Lambda}_{z,z'}(\lambda) 
={\chi_+^{-\frac{3}{2}- is}}*{\chi_+^{-k}(\lambda - \sqrt{\bL})}(z,z')=\chi_+^{-\frac{n+1}{2}- is}(\lambda - \sqrt{\bL})(z,z')
\end{eqnarray*}
and for $\lambda \le \Lambda$, $n=2k+1$ 
\begin{eqnarray*}
F^{s,\Lambda}_{z,z'}(\lambda)
={\chi_+^{-{2}- is}}*{\chi_+^{-k}(\lambda - \sqrt{\bL})}(z,z')
=\chi_+^{-\frac{n+1}{2}- is}(\lambda - \sqrt{\bL})(z,z'),
\end{eqnarray*}
i.e. the cutoff function $\eta$ has no effect for $\lambda \leq \Lambda$. 
Hence 
\begin{eqnarray*}
 \| \chi_+^{-\frac{n+1}{2}- is}(\Lambda - \sqrt{\bL})\|_{1 \to \infty}\le  \sup_{z,z'} 
| F^{s,\Lambda}_{z,z'}(\Lambda) |.
\end{eqnarray*}
 We consider first the odd dimensional case $n=2k+1$. By Lemma~\ref{infty} and \eqref{Heaviside}
\begin{equation}\begin{aligned}
 \big| F^{s,\Lambda}_{z,z'}(\Lambda) \big| &\le \big\| F^{s,\Lambda}_{z,z'}  \big\|_\infty \\ 
 &\le  
C(1+|s|)e^{\pi|s|/2} \sup \Big| \chi_+^{-1}* \Big( \eta(\cdot/\Lambda){\chi_+^{-k}(\cdot - \sqrt{\bL})}(z,z') \Big) \Big|^{1/2}  \\
&\phantom{aaaaaaaaaaaaaaa}\times 
 \sup \Big|\chi_+^{-3}* \Big( \eta(\cdot/\Lambda){\chi_+^{-k}(\cdot - \sqrt{\bL})}(z,z') \Big) \Big|^{1/2}  \\ 
 &\le C(1+|s|)e^{\pi|s|/2} \sup_{\lambda>0}\Big|\eta(\lambda/\Lambda){\chi_+^{-k}(\lambda - \sqrt{\bL})}(z,z')\Big|^{1/2}  \\
&\phantom{aaaaaaaaaaaaaaa}\times 
 \sup_{\lambda>0}\Big|\frac{d^2}{d\lambda^2}\eta(\lambda/\Lambda){\chi_+^{-k}(\lambda - \sqrt{\bL})}(z,z')\Big|^{1/2} 
 \end{aligned}\label{wuja}\end{equation}
 where the presence of the $\eta$ cutoff is now crucial. 
It follows from \eqref{spec-meas-j-1-3} with $j = n/2 - 3/2$ and $j = n/2 + 1/2$, i.e. $j=k-1$ and $j=k+1$, that 
$$
\sup_{\lambda>0}\big|\eta(\lambda/\Lambda){\chi_+^{-k}(\lambda - \sqrt{\bL})}(z,z')\big|\le C \Lambda^{k+1} (1+\Lambda w(z,z'))^{-1}.
$$
(Here we used the fact that the function $\lambda^k (1 + \lambda w)^{\beta}$ is an increasing function of $\lambda$ provided $\lambda \geq 0$, $w \geq 0$, $k \geq 0$ and $k + \beta \geq 0$.) Similarly, 
\begin{eqnarray*}
 \sup_{\lambda>0}\big|\frac{d^2}{d\lambda^2}\eta(\lambda/\Lambda){\chi_+^{-k}(\lambda - \sqrt{\bL})}(z,z')\big| \le 
 \sup_{\lambda>0}\big|\eta(\lambda/\Lambda){\chi_+^{-k-2}(\lambda - \sqrt{\bL})}(z,z')\big|
 \\ +\frac{1}{\Lambda}\sup_{\lambda>0}\big|\eta'(\lambda/\Lambda){\chi_+^{-k-1}(\lambda - \sqrt{\bL})}(z,z')\big|
  +\frac{1}{\Lambda^2}\sup_{\lambda>0}\big|\eta'(\lambda/\Lambda){\chi_+^{-k}(\lambda - \sqrt{\bL})}(z,z')\big| \\
 \le C \Lambda^{k-1} (1+\Lambda w(z, z')).
\end{eqnarray*}
Our estimate \eqref{left-est} for $n=2k+1$ follows now from these two  estimates and  \eqref{wuja}.

If $n=2k$ is even, then by Lemma~\ref{infty} and \eqref{Heaviside}
\begin{equation}\begin{aligned}
 \big| F^{s,\Lambda}_{z,z'}(\Lambda) \big| &\le \big\| F^{s,\Lambda}_{z,z'}  \big\|_\infty \\ 
 &\le  
C(1+|s|)e^{\pi|s|/2} \sup   \Big| \chi_+^{-1}* \Big(\eta(\cdot/\Lambda){\chi_+^{-k}(\cdot - \sqrt{\bL})}(z,z') \Big) \Big|^{1/2} 
 \\
&\phantom{aaaaaaaaaaaaaaa}\times 
 \sup\Big|\chi_+^{-2}* \Big( \eta(\cdot/\Lambda){\chi_+^{-k}(\cdot - \sqrt{\bL})}(z,z') \Big) \Big|^{1/2} \\ 
 &C(1+|s|)e^{\pi|s|/2}  \sup_{\lambda>0}\big|\eta(\lambda/\Lambda){\chi_+^{-k}(\lambda - \sqrt{\bL})}(z,z')\big|^{1/2} \\
&\phantom{aaaaaaaaaaaaaaa}\times  
 \sup_{\lambda>0}\big|\frac{d}{d\lambda}\eta(\lambda/\Lambda){\chi_+^{-k}(\lambda - \sqrt{\bL})}(z,z')\big|^{1/2} 
\end{aligned}\end{equation}
and we follow the same argument as in the odd dimensional case to establish \eqref{left-est} for $n=2k$.  
\end{proof}

In some situations, including the case of Laplace-type operators on asymptotically conic manifolds discussed later in this paper, we can express the spectral measure $dE(\lambda)$ in the form $P(\lambda) P(\lambda)^*$, where the initial space of $P(\lambda)$ is an auxiliary Hilbert space $H$. In this case, we can use a $T T^*$ argument to show that the conclusions of Proposition~\ref{keire} follow from localized estimates on $dE(\lambda)$, that is, on kernel estimates on $Q_i dE(\lambda) Q_i$, with respect to a    operator partition of unity
$$
\Id = \sum_{i=1}^{N(\lambda)} Q_i(\lambda), \quad 1 \leq i \leq N(\lambda).
$$
Notice that we allow the partition of unity to depend on $\lambda$. However, we shall assume that  $N(\lambda)$ is uniformly bounded in $\lambda$. 

\begin{remark} Here we assume that $Q_i(\lambda) dE_{\sqrt{\bL}}^{(j)}(\lambda) Q_i(\lambda)$ can be defined somehow and has a Schwartz kernel; for example, we might know that there is some weight function $\omega$ on $X$ such that $dE_{\sqrt{\bL}}^{(j)}(\lambda)$ is a bounded map from  $\omega^{j+1} L^2(X)$ to $\omega^{-j-1} L^2(X)$, and that $Q_i(\lambda)$ maps $\omega^a L^2(X)$ boundedly to itself for any $a$. This is the case in our application to asymptotically conic manifolds, with $\omega = x$ (where $x$ is as in \eqref{scatteringmetric}).  
\end{remark}

\begin{proof}[Proof of Proposition~\ref{QQ1}]
Observe that Proposition~\ref{QQ1} reduces to Proposition~\ref{keire} in the case that the partition of unity $Q_i$ is trivial. 
We apply the argument in the proof of Proposition~\ref{keire} to the operators $Q_i(\lambda) dE(\lambda) Q_i(\lambda)$, i.e. we replace $dE_{\sqrt{\bL}}(\lambda)$ by $Q_i(\lambda) dE_{\sqrt{\bL}}(\lambda) Q_i(\lambda)^*$ in \eqref{chiplusadefn}. The conclusion is that 
\begin{equation*}
\big\| Q_i(\lambda) dE_{\sqrt{\bL}}(\lambda) Q_i(\lambda)^*  \big\|_{L^p(X) \to L^{p'}(X)} \leq C \lambda^{n(1/p - 1/p') - 1}, \text{ for all } \lambda > 0.
\end{equation*}
Using the fact that $dE_{\sqrt{\bL}}(\lambda) =  P(\lambda) P(\lambda)^*$ and the $T T^*$ trick, we deduce that 
\begin{equation*}
\big\| Q_i(\lambda) P(\lambda)   \big\|_{L^2(X) \to L^{p'}(X)} \leq C \lambda^{n(1/2 - 1/p') - 1/2}, \text{ for all } \lambda > 0.
\end{equation*}
Now we can sum over $i$, and find that 
\begin{equation*}
\big\| P(\lambda)   \big\|_{L^2(X) \to L^{p'}(X)} \leq C \lambda^{n(1/2 - 1/p') - 1/2}, \text{ for all } \lambda > 0.
\end{equation*}
Finally, we use $dE_{\sqrt{\bL}}(\lambda) =  P(\lambda) P(\lambda)^*$ and the $T T^*$ trick again to deduce that 
\begin{equation*}
\big\| dE_{\sqrt{\bL}}(\lambda)   \big\|_{L^p(X) \to L^{p'}(X)} \leq C \lambda^{n(1/p - 1/p') - 1}, \text{ for all } \lambda > 0,
\end{equation*}
yielding \eqref{re}. Moreover, if the estimates hold only for $0 < \lambda \leq \lambda_0$, then we obtain \eqref{lere} instead. 
\end{proof}

\begin{remark} We acknowledge and thank Jared Wunsch for suggesting to us that the $T T^*$ trick would be useful here. 
\end{remark}


\part{Schr\"odinger operators on asymptotically conic manifolds}

In this second part of the paper, we specialize to the case that $(X, d, \mu)$ is an asymptotically conic manifold $(M^\circ,g)$ with the Riemannian distance function $d$ and Riemannian measure $\mu$, and $\bL$ is a Schr\"odinger operator $\bH$ on $L^2(M^\circ,g)$, that is, an operator of the form $\bH = \Delta_g + V$, where $\Delta_g$ is the positive Laplacian associated to $g$ and $V \in \CI(M)$ is a potential function vanishing to third order at the boundary of the compactification $M$ of $M^\circ$. We assume that $\bH$ has no $L^2$-eigenvalues (which implies that it is positive as  an operator) and that zero is not a resonance. 

The goal in this part of the paper is to show that $\bH$ satisfies the low energy spectral measure estimates \eqref{lere}, and the full spectral measure estimates \eqref{re} provided that $(M^\circ,g)$ is nontrapping. To do this, we will establish the estimates 
\eqref{spec-meas-j-1} for a suitable partition of unity $Q_i(\lambda)$. In the case of low energy estimates, i.e. $\lambda \in (0, \lambda_0]$ for $\lambda_0 < \infty$, these $Q_i$ will be pseudodifferential operators, 
lying in the calculus of operators introduced in \cite{GH1}. Thus our first task is to determine the nature of the operator $Q_i dE(\lambda) Q_i$ for such $Q_i$, which is the subject of Section~\ref{ows}.  Before this, however, we recall some of the geometric preliminaries from \cite{GHS} and \cite{HW}.

\section{Geometric preliminaries}
The Schwartz kernel of the spectral measure was constructed in \cite{GHS} for low energies and in \cite{HW} for high energies on a compactification of the space $ [0, \lambda_0] \times (M^\circ)^2$, resp. $[0, h_0] \times (M^\circ)^2$, where we use $h = \lambda^{-1}$ in place of $\lambda$ for high energies. We  use the definitions and machinery from these papers extensively, and 
we do not review this material comprehensively here, since that would double the length of this paper. Nevertheless, we shall describe these compactifications, review some of their geometric properties, and define some coordinate systems that we shall use in the following sections. 

Recall from the introduction that $(M^\circ, g)$ is asymptotically conic if $M^\circ$ is the interior of a compact manifold $M$ with boundary, such that in a collar neighbourhood of the boundary, the metric $g$ takes the form $g = dx^2 / x^4 + h(x)/x^2$, where $x$ is a boundary defining function and $h(x)$ is a smooth family of metrics on the boundary $\partial M$. We use $y = (y_1, \dots, y_{n-1})$ for local coordinates on $\partial M$, so that $(x, y)$ furnish local coordinates on $M$ near  $\partial M$. Away from $\partial M$, we use $z = (z_1, \dots, z_n)$ to denote local coordinates. 

\subsection{The low energy space $\MMkb$.}\label{sec:MMkb}
In \cite{GH1} and \cite{GHS}, following unpublished work of Melrose-S\'a Barreto the low energy space $\MMkb$ is defined as follows: starting with $[0, \lambda_0] \times M^2$, we define submanifolds $C_3:=\{0\}\x\pl M\x\pl M$ 
and 
\[C_{2,L}:=\{0\}\x\pl M\x  M, \quad C_{2,R}:=\{0\}\x M\x\pl M,\quad C_{2,C}:=[0,1]\x\pl M\x\pl M.\]
The space $\MMkb$ is then defined as $[0, \lambda_0] \times M^2$ with the codimension 3 corner $C_3$ blown up, followed by the three codimension 2 corners $C_{2, *}$:
\[M_{k,b}^2:=\big[[0,1]\x M\x M; C_3, C_{2,R},C_{2,L},C_{2,C}\big]\]
The new boundary hypersurfaces created by these blowups are labelled $\bfo, \rbo, \lbo$ and $\bfc$, respectively, and the original boundary hypersurfaces $\{ 0 \} \times M^2$, $[0, \lambda] \times M \times \partial M$ and $[0, \lambda] \times \partial M \times  M$ are labelled $\zf, \rb, \lb$ respectively. We remark that $\zf$ is canonically diffeomorphic to the b-double space 
$$
\MMb = \big[ M^2; \partial M \times \partial M \big].
$$
Also, each section $\MMkb \cap \{ \lambda = \lambda_* \}$, for fixed $0 < \lambda_* < \lambda_0$ is canonically diffeomorphic to $\MMb$.

 We define functions $x$ and $y$ on $\MMkb$ by lifting from the left copy of $M$ (near $\partial M$), and $x', y'$ by lifting from the right copy of $M$; similarly $z, z'$ (away from $\partial M$). We also define $\rho = x/\lambda$, $\rho' = x'/\lambda$, and $\sigma = \rho/\rho' = x/x'$. 
 Near $\bfc$ and away from $\rb$, we use coordinates $y, y', \sigma, \rho', \lambda$, while near $\bfc$ and away from $\lb$, we use coordinates $y, y', \sigma^{-1}, \rho, \lambda$. We also use the notation $\rho_{\bullet}$, where $\bullet = \bfo, \lbo$, etc, to denote a generic boundary defining function for the boundary hypersurface $\bullet$. 

This space has a compressed cotangent bundle $\Tkbstar \MMkb$, defined in \cite[Section 2]{GHS}. A basis of sections of this space is given, in the region $\rho, \rho' \leq C$ (which includes a neighbourhood of $\bfc$), by
\begin{equation}
\frac{d\rho}{\rho^2}, \quad \frac{d\rho'}{{\rho'}^2}, \quad \frac{ dy_i}{\rho}, \quad \frac{dy'_i}{\rho'},   \quad \frac{d\lambda}{\lambda}
\label{Tkbstar-basis}\end{equation}
Therefore, any point in $\Tkbstar \MMkb$ lying over this region can be written 
\begin{equation}
\nu \frac{d\rho}{\rho^2}  +\nu' \frac{d\rho'}{{\rho'}^2}+ \mu_i \frac{dy_i}{\rho} + \mu'_i \frac{dy'_i}{\rho'} + T \frac{d\lambda}{\lambda}.
\label{T}\end{equation}
This defines local coordinates $(y, y', \sigma, \rho', \lambda, \mu, \mu', \nu, \nu', T)$ in $\Tkbstar \MMkb$, near $\bfc$ and away from $\rb$, where $(\mu, \mu', \nu, \nu', T)$ are linear coordinates on each fibre. 

The compressed density bundle $\Omegakb(\MMkb)$ is defined  to be that line bundle whose smooth nonzero sections are given by the wedge product of a basis of sections for $\Tkbstar(\MMkb)$. Using the coordinates above, we can write a smooth nonzero section $\omegab$ as follows: 
\begin{equation}
\omegab = \Big| \frac{d\rho d\rho' dy dy' d\lambda}{\rho^{n+1} {\rho'}^{n+1} \lambda} \Big| \sim \lambda^{2n} \Big| \frac{dg dg' d\lambda}{\lambda} \Big| \text{ in the region } \rho, \rho' \leq C. 
\end{equation}
For $\rho, \rho' \geq C$, we can take $\omegab = (x x')^n |dg dg' d\lambda/\lambda|$. Here 
 $dg$, resp. $dg'$ denotes the Riemannian density with respect to $g$, lifted to $\MMkb$ by the left, resp. right projection. 

The boundary of $\Tkbstar \MMkb$ lying over boundary hypersurface $\bullet$ is denoted $\Tkbstar_{\bullet} \MMkb$. 
The space $\Tkbstar_{\lb} \MMkb$ fibres over the space $\Nsfstar Z_{\lb} \times [0, \lambda_0]$, which\footnote{The spaces $Z_{\bullet}$ and $\Nsfstar Z_{\bullet}$ are defined in  \cite[Section 2]{GHS}.} is canonically isomorphic to $\Tscstar_{\partial M} M \times [0, \lambda]$ ($\Tscstar_{\partial M} M$ is defined in \cite{HV1}, \cite{HV2}). This fibration is given in local coordinates by
\begin{equation}
(y, y', \sigma, \lambda, \mu, \mu', \nu, \nu', T) \to 
(y, \mu, \nu, \lambda).
\label{lfib}\end{equation}
Similarly there is a natural fibration from $\Tkbstar_{\rb} \MMkb$ to  $\Nsfstar Z_{\rb}\times [0, \lambda_0]$, which takes the form 
\begin{equation}
(y, y', \sigma, \lambda, \mu, \mu', \nu, \nu', T) \to 
(y', \mu', \nu', \lambda).
\label{rfib}\end{equation}
We also note that there are natural maps $\pi_L, \pi_R$ mapping $\Nsfstar Z_{\bfc}\times [0, \lambda_0]$, which is naturally isomorphic to $\Tscstar_{\bfc} \MMb\times [0, \lambda_0]$ (see \cite{HV1} or \cite{HV2}), to  $\Tscstar_{\partial M} M\times [0, \lambda_0]$ which are induced by the projections  $T^* M^2 \to T^* M$ onto the left, respectively right factor.  In local coordinates, these are given by 
\begin{equation}
\pi_L(y, y', \sigma, \mu, \mu', \nu, \nu', \lambda) = (y, \mu, \nu, \lambda), \quad 
\pi_R(y, y', \sigma, \mu, \mu', \nu, \nu', \lambda) = (y', \mu', \nu', \lambda).
\label{piLpiR}\end{equation}
We use these maps in Section~\ref{ows}. 

The space $\Tkbstar_{\bfc} \MMkb$ is canonically diffeomorphic to 
$\Tsfstar_{\bfc} \MMb \times [0, \lambda_0]$, where $\Tsfstar_{\bfc} \MMb$ is the scattering-fibred cotangent bundle of $\MMb$ defined in \cite{HV1}. The space
$\Tsfstar_{\bfc} \MMb$  has a natural contact structure, and Legendre submanifolds with respect to this structure play an important role in encoding the oscillations of the spectral measure at the boundary of $\MMkb$. In fact, three Legendre submanifolds of $\Tsfstar_{\bfc} \MMb$ arise in the identification of the spectral  measure as a Legendre distribution (see \cite[Section 3]{GHS}), which we now briefly describe. One is denoted $\Nscstar \ddiagb$,  which in coordinates used in \eqref{T} is given by 
\begin{equation}
\Nscstar \ddiagb = \{ (y, y', \sigma, \mu, \mu', \nu, \nu') \mid y = y', \sigma = 1, \mu = -\mu', \nu = -\nu' \} ;
\label{Nscstar-ddiagb}\end{equation}
it is a sort of conormal bundle to the boundary of the diagonal $\ddiagb$, 
\begin{equation}
\ddiagb = \{ (y, y', \sigma) \mid y = y', \sigma = 1 \},
\label{ddiagb}\end{equation}
 in $\MMb$, and carries the `operator wavefront set' or `microlocal support' of scattering pseudodifferential operators. Another is the incoming/outgoing Legendrian submanifold $L^\sharp$, which in coordinates used in \eqref{T} is given by 
\begin{equation*}
\Nscstar \ddiagb = \{ (y, y', \sigma, \mu, \mu', \nu, \nu') \mid  \mu = \mu' = 0, \nu = \pm 1, \nu' = -\nu \} ;
\end{equation*}
it has two components (corresponding to the sign of $\nu$) and  describes oscillations that are purely radial, that is, purely incoming or outgoing. The third and most interesting Legendre submanifold is the propagating Legendrian, denoted $L^{\bfc}$. To describe it, let $G$ denote the characteristic variety of $\bH - \lambda^2$. Then $L^{\bfc}$ is given by the flowout from  $\Nscstar \ddiagb \cap G$ by the bicharacteristic flow of $\bH$. It connects the incoming and outgoing components of $L^\sharp$ and has a conic singularity at each. In terms of these Legendre submanifolds we have 
 
\begin{theo}\label{lesmLd} \cite[Theorem 3.10]{GHS}
The spectral measure 
$dE_{\sqrt{\bH}}(\lambda)$, for $0 < \lambda \leq \lambda_0$, is a conormal-Legendre distribution in the space 
$$
I^{m,p; r_{\lb}, r_{\rb}; \mcA}(\MMkb, (L^{\bfc}, L^{\sharp, \bfc}); \Omegakbh) \otimes \big| \frac{d\lambda}{\lambda} \big|^{-1/2}, 
$$
with $m = -1/2$, $p = (n-2)/2$, $r_{\lb} = r_{\rb} = (n-1)/2$, and where $\mcA$ is an index family with index sets at the faces $\bfo, \lbo, \rbo, \zf$ starting at order $-1$, $n/2 - 1$, $n/2 - 1$, $n-1$ respectively. 
\end{theo}

\subsection{The high energy space $\bX$}\label{sec:bX}
The high energy space $\bX$ is defined by $\bX = [0, h_0] \times \MMb$. The boundary hypersurfaces $[0, h_0] \times M \times \partial M$, $[0, h_0] \times \partial M \times  M$ and $\{ 0 \} \times \MMb$ are denoted $\rb, \lb$ and $\mf$ (`main face'), respectively, and the boundary hypersurface arising from $[0, h_0] \times \partial M \times \partial M$ is denoted $\bfc$. Notice that this space fits together with the low energy space: in the range $\lambda \in (C^{-1}, C)$ (where $\lambda = 1/h$), the spaces both have the form $(C^{-1}, C) \times \MMb$, and the labelling of boundary hypersurfaces is consistent. 
As before, we write $\sigma = x/x'$. 
We use coordinates $(y, y', \sigma, x', h)$ near $\bfc$ and away from $\rb$, and coordinates $(y, y', \sigma^{-1}, x, h)$ near $\bfc$ and away from $\lb$. Away from $\bfc, \lb, \rb$ we use coordinates $(z, z', h)$. 

The compressed cotangent bundle $\Tsfstar \bX$ is described in \cite{HW}. A basis of sections of this bundle is given in the region $x, x' \leq \epsilon$ by 
$$
\frac{dy_i}{xh}, \quad \frac{dy'_i}{x'h}, \quad d \big(\frac1{xh}\big), \quad  d \big(\frac1{x'h}\big), \quad  d \big(\frac1{h}\big).
$$
In terms of this basis, any point in $\Tsfstar \bX$ lying over this region can be written 
\begin{equation}
\mu \cdot \frac{dy}{xh} + \mu' \cdot \frac{dy'}{x'h} + \nu d \big(\frac1{xh}\big) + \nu'  d \big(\frac1{x'h}\big) + \tau d \big(\frac1{h}\big).
\label{coord-he-1}\end{equation}
This defines local coordinates $(y, y', \sigma, x', h, \mu, \mu', \nu, \nu', \tau)$, where $(\mu, \mu', \nu, \nu', \tau)$ are local coordinates on each fibre. In the region $x, x' \geq \epsilon$, a basis of sections is
$$
\frac{dz_i}{h}, \quad \frac{dz'_i}{h}, \quad d \big(\frac1{h}\big),
$$
and in terms of this basis, any point in $\Tsfstar \bX$ lying over this region can be written 
\begin{equation}
\zeta \cdot \frac{dz}{h} + \zeta' \cdot \frac{dz'}{h} + \tau d \big(\frac1{h}\big).
\label{coord-he-2}\end{equation}
This defines local coordinates $(z, z', h, \zeta, \zeta', \tau)$ on $\Tsfstar \bX$ over this region. 

This compressed density bundle $\sfO(\bX)$ is defined to be that line bundle whose smooth nonzero sections are given by a wedge product of a basis of sections for $\Tsfstar \bX$. We find that $|dg dg' dh/h^2| = |dg dg' d\lambda|$ is a smooth nonzero section of this bundle. 

We also note that there are natural maps from $\Tsfstar[\mf] \bX \to \Tscstar M$, which (abusing notation) we will also denote $\pi_L, \pi_R $, which are induced by the projections onto the left, respectively right factor $T^* M^2 \to T^* M$. In local coordinates, these are given by 
\begin{equation}
\pi_L(z,z', \zeta, \zeta', \tau) = (z, \zeta), \quad 
\pi_R(z,z', \zeta, \zeta', \tau) = (z', \zeta')
\label{piLpiRmf}\end{equation}
away from  the boundary hypersurface $\bfc$, or near $\bfc$ by 
\begin{equation}
\pi_L(x, y, x', y', \mu, \mu', \nu, \nu', \tau) = (x, y, \mu, \nu), \quad 
\pi_R(x, y, x', y', \mu, \mu', \nu, \nu', \tau) = (x',y', \mu', \nu').
\label{piLpiRmfbfc}\end{equation}

The space $\Tsfstar[\mf]\bX$ has a natural contact structure, as described in \cite{HW}. Legendre submanifolds with respect to this contact structure are important in describing the singularities of the spectral measure at high frequencies. We need to define two Legendre submanifolds $\Nsfstar \diagb$ and $L$ in order to describe the spectral measure at high energies as a Legendre distribution on $\bX$ (see \cite{HW}). 
The first of these, $\Nsfstar \diagb$, is associated to the diagonal submanifold $\diagb \subset \{ 0 \} \times \MMb$, defined using the coordinates above by 
\begin{equation}
\Nsfstar \diagb = \{ (z, z', h, \zeta, \zeta', \tau) \mid z= z', \zeta = -\zeta', h = 0, \tau = 0 \}
\label{Nsfstar-diagb}\end{equation}
away from $\bfc$, and 
\begin{multline}
\Nsfstar \diagb = \{ (y, y', \sigma, x', h, \mu, \mu', \nu, \nu', \tau) \mid y = y', \sigma = 1, h = 0, \\ \mu = -\mu', \nu = -\nu', \tau = 0 \}
\label{Nsfstar-diagb-bf}\end{multline}
near $\bfc$. The other, $L$, is obtained just as $L^{\bfc}$ was obtained from $\Nscstar \ddiagb$ in the previous subsection, namely as the flowout by the bicharacteristic flow of $\bH$ starting from the intersection of $\Nsfstar \diagb$ and the characteristic variety of $h^2\bH - 1$. Indeed, the submanifolds $L^{\bfc}$ and $\Nscstar \ddiagb$ are essentially the boundary hypersurfaces of $L$ and $\Nsfstar \diagb$ lying over $\bfc \cap \mf$. In terms of these Legendre submanifolds, we have

\begin{theo}\label{hesmLd} \cite[Corollary 1.2]{HW}
Suppose that $(M, g)$ is nontrapping. Then the spectral measure $dE_{\sqrt{\bH}}(\lambda)$ is a Legendre distribution 
 on $\bX$,  associated to an intersecting pair of Legendre submanifolds with conic points $(L, L^\sharp)$ where $L \subset \Tscstar_{\mf} \bX$ has a conic singularity at $L^\sharp \subset \Tscstar_{\mf \cap \bfc} \bX$:
 $$
dE_{\sqrt{\bH}}(\lambda) \in I^{m,p; r_{\bfc}, r_{\lb}, r_{\rb}}(\bX, (L, \Lsharp); \sfOh) \otimes |d\lambda|^{1/2}
$$
with $m = 1/2$, $p = (n-2)/2$, $r_{\bfc} = -1/2$, $r_{\lb} = r_{\rb} = (n-1)/2$. Here we use the order conventions in Remark~\ref{orderconventions}. 
\end{theo}

\begin{remark}\label{orderconventions} We use different order conventions from \cite{HW}, to agree with those used in \cite{GHS}. In terms of equation (4.15) of \cite{HW}, the order convention in the present paper corresponds to taking $N = 2n$ (not $2n+1$ as in \cite{HW}), i.e. the total space dimension, but not including the $\lambda$ dimension, and taking the fibre dimensions $f_{\bfc} = 0$ and $f_{\lb} = f_{\rb} = n$, i.e. again not including the $\lambda$ dimension. This has the effect that the orders in the present paper are $1/4$ larger at $\mf = \MMb \times \{ h = 0 \}$, and $1/4$ smaller at $\bfc, \lb$ and $\rb$, compared to \cite{HW}, and explains the discrepancies in the orders above compared to those given in Corollary 1.2 of \cite{HW}. (An  advantage of the ordering convention used here is that a semiclassical pseudodifferential operator of (semiclassical) order $m$, multiplied by $|dh/h^2|^{1/2} = |d\lambda|^{1/2}$ becomes a Legendre distribution  of the same order $m$ at 
 the conormal bundle of the diagonal in $\mf$.)
 \end{remark}


\section{Microlocal support}\label{ows}

Recall from the end of the Introduction our strategy for proving Theorem~\ref{main3}, involving estimates \eqref{spec-meas-j-1}. The elements $Q_i$ of our partition of unity will be chosen to be 
pseudodifferential operators  lying in the calculus of operators introduced in  \cite[Definition 2.7]{GH1}. In view of Theorem~\ref{lesmLd}, we need to understand what happens when a  conormal-Legendre distribution $F \in I^{m, r_{\lb}, r_{\rb}, \mcB}(\Lambda, \Omegakbh)$ is pre- and post-multiplied by  such operators. We shall use the notation $\Psi^{m}_k(M, \Omegakbh)$ to denote what in \cite{GH1} was written $\Psi^{m, \mcE}(M, \tilde \Omega_b^{1/2})$ where the index family $\mcE$ assigns the  $\CI$ index family at $\sca, \bfo$ and $\zf$ and the empty index family at all other boundary hypersurfaces. Such operators have kernels defined on the space $\MMksc$, defined in \cite{GH1}, 
that are conormal of order $m$ to the diagonal, uniformly to the boundary, smooth away from the diagonal, and rapidly vanishing at all boundary hypersurfaces not meeting the diagonal. As shown in \cite[Proposition 2.10]{GH1}, $\Psi^0(M, \Omegakbh)$ is an algebra. It follows, using H\"ormander's ``square root'' trick \cite[Section 18.1]{Hor3} that  such kernels act as uniformly bounded (in $\lambda$) operators on $L^2(M)$. 

In this section, we consider operators $Q$, $Q'$ such that
\begin{equation}\begin{gathered}
\text{ $Q, Q'$ are  of order $-\infty$, i.e. $Q, Q' \in \Psi^{-\infty}_k(M, \Omegakbh)$,} \\
\text{with compactly supported symbols;}
\end{gathered}\label{minusinfty}\end{equation}
\begin{equation}\begin{gathered}
\text{ $Q$ and $Q'$ have kernels supported close to the diagonal,}
\\ 
\text{ in particular in the region $\{ \sigma := x/x' \in [1/2, 2] \}$.}
\end{gathered}\label{kernel-support}\end{equation}

With these assumptions, the kernels of $Q, Q'$ are smooth (across the diagonal) on the space $\MMksc$. Viewed as distributions on $\MMkb$ (which has one fewer blowup than $\MMksc$) the kernels have a conic singularity at the boundary of the diagonal, $\ddiagb$. As shown in 
\cite[Section 5.1]{HV2}, this means that they are  Legendre distributions in $I^{0, \infty, \infty; (0, 0, \emptyset, \emptyset)}(\MMkb, \Nscstar \ddiagb; \Omegakbh)$, i.e. Legendre distributions of order $0$ associated to $\Nscstar \ddiagb$ (see \eqref{Nscstar-ddiagb}), with the $\CI$ index set $0$  at $\bfo$ and $\zf$, and vanishing in a neighbourhood of $\lb$, $\rb$,  $\lbo$ and $\rbo$ (which is of course an trivial consequence of \eqref{kernel-support}). 
 
 \begin{remark}
The composition $QF$ or $FQ'$ is always well-defined when $F$ is a Legendre distribution on $\MMkb$ and $Q, Q'$ are as above, since $F$ can  be regarded as a map from $x^a L^2(M)$ to $x^{-a} L^2(M)$ for sufficiently large $a \in \RR$, depending smoothly on $\lambda \in (0, \lambda_0)$, while pseudodifferential operators of order $0$ are bounded on $x^a L^2(M)$ (uniformly in $\lambda$) for any $a$. 
\end{remark} 
  
To state our results, we need to introduce some notation and define the notion of the microlocal support of $F$. Let $\Lambda \subset \Nsfstar Z_{\bfc} \equiv \Tscstar_{\bfc}\MMb$ be the Legendre submanifold associated to $F$. We always assume that $\Lambda$ is compact. Recall from \cite[Section 4]{HW} that $\Lambda$ determines two associated Legendre submanifolds $\Lambda_{\lb}$ and $\Lambda_{\rb}$ which are the bases of the fibrations on $\partial_{\lb} \Lambda$ and $\partial_{\rb} \Lambda$, respectively. These may be canonically identified with Legendre submanifolds of $\Tscstar M$. We also define $\Lambda'$ by negating the fibre coordinates corresponding to the right copy of $M$, i.e. 
\begin{equation}
q' = (y, y', x/x',\mu, \mu', \nu, \nu') \in \Lambda' \iff q = (y, y', x/x',\mu, -\mu', \nu, -\nu') \in \Lambda.
\label{qq'}\end{equation}
Similarly we define  $\Lambda_{\rb}'$ by negating the fibre coordinates:
$$ 
 q' = (y', \mu', \nu') \in \Lambda_{\rb}' \iff q = (y',  -\mu',-\nu') \in \Lambda_{\rb}.
$$

We also define $\Lbar$, $\Llbbar$, $\Lrbbar$ by 
\begin{equation}
\Lbar = \Lambda' \times [0, \lambda_0], \quad \Llbbar = \Lambda'_{\lb} \times [0, \lambda_0], \quad \Lrbbar = \Lambda'_{\rb} \times [0, \lambda_0].
\label{Lbar-defn}\end{equation}

To define the microlocal support, $\WF'(F)$, of $F$ we first recall from \cite{GHS} that 
 $F \in I^{m, r_{\lb}, r_{\rb}, \mcB}(\Lambda, \Omegakbh)$ means that $F$ can be decomposed $F = F_1 + F_2 + F_3 + F_4 + F_5 + F_6$, where 
\begin{itemize}
\item $F_1$ is supported near $\bfc$ and away from $\lb, \rb$;
\item $F_2$ is supported near $\bfc \cap \lb$;
\item $F_3$ is supported near $\bfc \cap \rb$;
\item $F_4$ is supported near $\lb$ and away from $\bfc$;
\item $F_5$ is supported near $\rb$ and away from $\bfc$;
\item $F_6$ vanishes rapidly at the boundary of $\MMb$;
\end{itemize}
and each $F_i$, $1 \leq i \leq 5$ has an oscillatory representation as follows:

$\bullet$ $F_1$ is a finite sum of terms of the form (up to rapidly vanishing terms which may be included in $F_6$)
\begin{equation}
\rho^{m-k/2+n/2} \int_{\RR^k} e^{i\Phi(y, y', x/x', v)/\rho} a(\lambda, \rho, y, y', \sigma, v) \, dv \ \omegab 
\label{u1}\end{equation}
where $\Phi$ locally parametrizes $\Lambda$, $\omegab$ is a nonzero section of the half-density bundle $\Omegakbh$,  compactly supported in $v$, and 
\begin{multline}
\text{$a$ is polyhomogeneous conormal in $\lambda$ with index set $\mcA_{\bfo}$ and} \\ \text{smooth in all other variables.}
\label{acond}\end{multline}

$\bullet$ $F_2$ is a finite sum of terms of the form (up to rapidly vanishing terms which may be included in $F_6$)
\begin{multline}
 \sigma^{r_{\lb}-k/2}  \rho'^{m-(k+k')/2+n/2} \int_{\RR^{k+k'}} e^{i\Phi_1(y, v)/ \rho} e^{i\Phi_2(y,y', \sigma, v, w)/ \rho'} \\ \times a(\lambda,  \rho', y, y', \sigma, v, w) \, dv \, dw \, \omegab
 \label{u2}\end{multline}
where $\Phi = \Phi_1 + \sigma \Phi_2$ locally parametrizes $\Lambda$ (in particular, $\Phi_1$ locally parametrizes $\Lambda_{\lb}$), and $a$ satisfies \eqref{acond} . 

$\bullet$ $F_3$ is a finite sum of terms of the form (up to rapidly vanishing terms which may be included in $F_6$)
\begin{equation}
 \rho^{m-(k+k')/2+n/2} \sigmat^{r_{\rb}-k/2} 
\!\! \int\limits_{\RR^{k+k'}} \!\!  e^{i\Phi_1'(y',v)/ \rho'} e^{i\Phi_2'(y, y', \sigmat, v,w)/ \rho} a(\lambda,  \rho, y, y', \sigmat, v, w) \, dv \, dw \, \omegab
\label{u3}\end{equation}
where $\sigmat =  \rho'/ \rho = \sigma^{-1}$ and  $\Phi = \Phi_1' + \sigmat \Phi_2'$ locally parametrizes $\Lambda$ (in particular, $\Phi_1'$ locally parametrizes $\Lambda_{\rb}$), and $a$ satisfies \eqref{acond}. 

$\bullet$ $F_4$ is a finite sum of terms of the form 
\begin{equation}
 \rho^{r_{\lb}-k/2} \int_{\RR^k} e^{i\Phi_1(y, v)/ \rho} a(\lambda,  \rho, y, z', v) \, dv \, \omegab
\label{u4}\end{equation}
where $\Phi$ parametrizes $\Lambda_{\lb}$ and $a$ satisfies \eqref{acond}. 

$\bullet$ $F_5$ is a finite sum of terms
\begin{equation}
( \rho')^{r_{\rb}-k/2} \int_{\RR^k} e^{i\Phi_1'(y', v')/ \rho'} a(\lambda,  \rho', y', z, v) \, dv \, \omegab
\label{u5}\end{equation}
where $\Phi'$ parametrizes $\Lambda_{\rb}$ and $a$ satisfies \eqref{acond}.

Then we define the microlocal support $\WF'(F)$ of $F$ to be a closed subset of $\Lbar \cup \Llbbar \cup \Lrbbar$ as follows: we say that $(q', \lambda) \in \Lbar$ is not in $\WF'(F)$ iff there is a neighbourhood of $(q, \lambda) \in \Lambda \times [0, \lambda_0] $ in which $F$ has order $\infty$. In terms of the oscillatory integral representation \eqref{u1}, say,
the condition that $F$ has order infinity at $(q, \lambda)$ is equivalent to $a$ vanishing rapidly in a neighbourhood of the point $(\lambda, 0, y, y', \sigma, v)$ which corresponds under \eqref{qq'} to $(q, \lambda)$ in the sense that    $d_{y, y', \sigma, \rho}(\Phi(y, y', x/x', v)/\rho) = q$ and $d_v \Phi(y, y', x/x', v) = 0$ (by nondegeneracy there is only one $v$ with this property). Similarly, in \eqref{u2} and \eqref{u3}. Likewise, we say that  $(q, \lambda) \in \Llbbar$ is not in $\WF'(F)$ iff there is a neighbourhood of the \emph{fibre} (see \eqref{lfib}) of  $(q, \lambda) \in \Lambda_{\lb} \times [0, \lambda_0] $ in which $F$ has order $\infty$, and
$(q', \lambda) \in \Lrbbar$ is not in $\WF'(F)$ iff there is a neighbourhood of the \emph{fibre} of  $(q, \lambda) \in \Lambda_{\rb} \times [0, \lambda_0] $ in which $F$ has order $\infty$. The fibre here is a copy of $M$. In terms of the oscillatory integral representation \eqref{u2},  
the condition that $F$ has order infinity in a neighbourhood of the fibre of $(q, \lambda) = (y, \mu, \nu, \lambda) \in \Llbbar$ is equivalent to $a$ vanishing rapidly in a neighbourhood of the point $(\lambda, \rho', y, y', 0, v, w)$ for all $(\rho', y', v, w)$ such that 
$d_{y, \rho}(\Phi_1/\rho) = q$ and $d_{v} \Phi_1 = 0$. Similarly, in \eqref{u4} the condition is that $a$ vanishes rapidly in a neighbourhood of the point $(\lambda, 0, y, z', v)$ for all $(z', v)$ such that $d_{y, \rho}(\Phi_1/\rho) = q$ and $d_{v} \Phi_1 = 0$.

These components of $\WF'(F)$ will be denoted $\WF'_{\bfc}(F)$, $\WF'_{\lb}(F)$ and $\WF'_{\rb}(F)$, respectively. 

Note that, if $F \in I^{m, r_{\lb}, r_{\rb}, \mcA}(\Lambda)$, then $F$ is rapidly decreasing at $\bfc$, $\lb$ and $\rb$  iff $\WF'(F)$ is empty. 
Also note that if $\WF'_{\lb}(F)$ is empty, then $\partial_{\lb} \Lambda \times [0, \lambda_0]$ is disjoint from $\WF'_{\bfc}(F)$, but the converse need not hold: if the kernel of $F$ is supported away from $\bfc$ then certainly $\WF'_{\bfc}(F)$ will be empty, but $\WF'_{\lb}(F)$ need not be. 

This definition makes sense also for pseudodifferential operators $Q$ of order $-\infty$ with compact operator wavefront set. In the case of a pseudodifferential operator, the Legendre submanifold is $\Nscstar \ddiagb$ defined in \eqref{Nscstar-ddiagb}, and   the components $\Lambda_{\lb} \cup \Lambda'_{\rb}$ are empty. Since $\Nscstar \ddiagb$ is canonically diffeomorphic to $\Tscstar_{\partial M} M$, 
we will always consider the microlocal support $\WF'(Q)$ of a pseudodifferential operator $Q$ of differential order $-\infty$ to be a subset of $\Tscstar_{\partial M} M \times [0, \lambda_0]$. 

\begin{lem}\label{QFQ'-microsupport}
Assume that $F \in I^{m,r_{\lb}, r_{\rb}; \mcA}(\MMb, \Lambda; \Omega)$ is associated to a compact Legendre submanifold $\Lambda$ and that $Q \in \Psi^{-\infty}_k(M; \Omegakbh)$ is of differential order $-\infty$, with compact operator wavefront set. Then $Q F$ is also a Legendre distribution in the space $I^{m,r_{\lb}, r_{\rb}; \mcA}(\MMb, \Lambda; \Omega)$ and we have
\begin{equation}\begin{gathered}
\WF'_{\lb}(QF) \subset \WF'(Q) \cap \WF'_{\lb}(F) \\
 \WF'_{\bfc}(QF) \subset  \pi_L^{-1} \WF'(Q)  \cap \WF'_{\bfc}(F)   \\ 
 \WF'_{\rb}(QF) \subset \WF'_{\rb}(F) 
\label{WFQF}\end{gathered}\end{equation}
where $\pi_L, \pi_R$ are as in \eqref{piLpiR}. 
Moreover, if $Q$ is microlocally equal to the identity on $\pi_L(\WF'_{\bfc}(F))$ and $\WF'_{\lb}(F)$, then $Q F  - F \in I^{\infty, \infty, r_{\rb}}(\MMb)$, i.e. vanishes to infinite order at $\lb$ and $\bfc$. 
\end{lem}

There is of course a corresponding theorem for composition in the other order, which is obtained by taking the adjoint of the lemma above. Combining the two we obtain

\begin{cor}
Suppose that $F$ and $Q, Q'$ are as above. Then 
\begin{equation}\begin{gathered}
\WF'_{\lb}(QFQ') \subset \WF'(Q) \cap \WF'_{\lb}(F) \\
 \WF'_{\bfc}(QFQ') \subset  \pi_L^{-1} \WF'(Q)  \cap \pi_R^{-1} \WF'(Q') \cap \WF'_{\bfc}(F)   \\ 
 \WF'_{\rb}(QFQ') \subset \WF'(Q') \cap  \WF'_{\rb}(F) .
\label{WFQFQ'}\end{gathered}\end{equation}
\end{cor}

\begin{proof}[Proof of lemma]
We decompose as above $F = F_1 + F_2 + F_3 + F_4 + F_5 + F_6$, and consider each piece $F_i$ separately. 


$\bullet$ $F_1$ term. Using the notation in \eqref{u1},  the composition $Q F_1$ takes the form
\begin{multline}
(2\pi)^{-n} \int_0^\infty \int e^{i\big((y-y'')\cdot \mu + (1-\rho/\rho'') \nu\big)/\rho} q(\lambda, \rho, y, \mu, \nu) \\
\times (\rho'')^{m-k/2+n/2} e^{i\Phi(y'', y', \rho'/\rho'', v)/\rho''} 
a(\lambda, \rho', y'', y', \rho'/\rho'', v) \, dv \, d\mu \, d\nu \, \frac{dy'' \, d\rho''}{{\rho''}^{n+1}} \, \omegab.
\end{multline}
Here the measure $\lambda^n dg''$ which arises from the combination of half-densities in $Q$ and $F$ is equal to $dy'' d\rho''/{\rho''}^{n+1}$ times a smooth nonzero factor, which has been absorbed into the $a$ term. 
Writing $\sigma'' =  \rho/ \rho''$,  this can be expressed
\begin{multline}
(2\pi)^{-n} \rho^{m-k/2-n+n/2} \int e^{i\big((y-y'')\cdot \mu + (1-\sigma'') \nu + \sigma'' \Phi(y'', y', \sigma'' /\sigma, v) \big)/\rho} \\
\times q(\lambda, \rho, y, \mu, \nu) (\sigma'')^{m-k/2+n/2 - n - 1}
a(\lambda,  \rho', y'', y', \sigma'' \sigma^{-1},  v) \, dv \, d\mu \, d\nu \, dy'' \, d\sigma'' \, \omegab.
\label{QF_1}\end{multline}

For $ \rho \geq \epsilon > 0$ the phase is not oscillating and this is polyhomogeneous conormal at $\bfo$ with the same index set $\mcA_{\bfo}$ as for $a$. For $ \rho$ small,  we perform stationary phase in the $(y'', \sigma'', \mu, \nu)$ variables. The phase has a nondegenerate stationary point where $y'' = y, \sigma'' = 1, \mu = d_y \Phi, \nu = \Phi + \sigma^{-1} d_\sigma \Phi$,  and we obtain an asymptotic expansion at $ \rho$ of the form
\begin{multline}
 \rho^{m-k/2+n/2} \int_{\RR^k} e^{i \Phi(y, y',  \sigma, v) / \rho} \tilde a(\lambda,  \rho, y, y', \sigma, v) \, dv \, \omegab , \\
\tilde a(\lambda, \rho, y, y', \sigma, v) = 
 \lambda^{-n/2} \sum_{j=0}^M  \rho^j \Big(  \frac{(\partial_{y''} \cdot \partial_\mu + \partial_{\sigma''} \partial_\nu)^j}{i^j j!}  q(\lambda,  \rho, y, \mu, \nu)  \\ \times  (\sigma'')^{m-k/2+n/2 - n - 1}  a(\lambda,  \rho', y'', y', \sigma'' /\sigma,v) \Big) \Bigg|_{y=y'', \sigma'' = 1, \mu = d_y \Phi, \nu =\Phi + \sigma d_\sigma \Phi} \!\!\!\!\!\!\!\!+ O( \rho^{M+1}) .
\label{comp-symbol-formula}\end{multline}
In particular, this is a Legendre distribution associated to $\Lambda$ of the same order, and with the same index family,  as $F$. Moreover, we see from the formula  \eqref{comp-symbol-formula} that the microlocal support  $\WF'_{\bfc}(QF_1)$ is contained in $\WF'_{\bfc}(F)$, as well as contained in $\pi_L^{-1} \WF'(Q)$.

If $q = 1 + O( \rho^\infty)$ on $\pi_L(\WF'_{\bfc}(F))$, then in the sum over $j$  in \eqref{comp-symbol-formula}, only the $j=0$ term is nonzero, because in all other terms, either $a = 0$ or $q=1 + O( \rho^\infty)$ (implying that any derivative of $q$ is $O( \rho^\infty)$) when evaluated at $y=y'', \sigma'' = 1, \mu = d_y \Phi, \nu =\Phi + \sigma d_\sigma \Phi$. Therefore, in this case, $QF_1 = F_1$ mod $O( \rho^\infty)$.

$\bullet$ $F_2$ term. In the notation \eqref{u2},  the composition $Q F_2$ takes the form 
\begin{multline*}
(2\pi)^{-n} \int e^{i \big((y-y'')\cdot \mu + (1-\sigma'') \nu\big)/ \rho} q(\lambda,  \rho, y, \mu, \nu) {\rho''}^{r_{\lb}-k/2} {\rho'}^{m-r_{\lb}-k'/2+n/2} \\
\times e^{i\Phi_1(y, v)/ \rho''} e^{i\Phi_2(y'',y', \sigma''/\sigma, v, w)/ \rho'} a(\lambda,  \rho', y'', y',\sigma/\sigma'',  v, w) \, dv \, dw \, d\mu \, d\nu \, \frac{dy'' \, d \rho''}{{ \rho''}^{n+1}} \,  \omegab. 
\end{multline*}
This can be written
\begin{multline*}
 \frac{\rho^{r_{\lb}-k/2-n} { \rho'}^{m-r_{\lb}-k'/2+n/2}}{(2\pi)^n}
\int e^{i\big((y-y'')\cdot \mu + (1-\sigma'') \nu + \sigma'' \Phi_1(y'', v) + \sigma \Phi_2(y'', y', \sigma/\sigma'', v, w) \big)/\rho} \\
\times q(\lambda, \rho, y, \mu, \nu) (\sigma'')^{-r_{\lb}+k/2+n-1} a(\lambda,  \rho', y'', y', \sigma/\sigma'',  v, w) \, dv \, dw \,  d\mu \, d\nu \, dy'' \, d\sigma'' \, \omegab. 
\end{multline*}
Now we perform stationary phase in the $(y'', \sigma'', \mu, \nu)$ variables. The phase has a nondegenerate stationary point where $y'' = y, \sigma'' = 1, \mu = d_y \Phi_1, \nu = \Phi_1 -  d_\sigma \Phi$,  and the rest of the argument to bound $\WF'_{\bfc}(QF)$ is the same as for $F_1$. We also see from the stationary phase expansion that $\WF'_{\lb}(QF)$ is contained in both $\WF'(Q)$ and $\WF'_{\lb}(F)$. 

$\bullet$ $F_4$ term. This works just as for the $F_2$ term.

$\bullet$ $F_3$ term. In the notation \eqref{u3},  the composition $Q F_3$ takes the form 
\begin{multline*}
(2\pi)^{-n}  \int  e^{i\big((y-y'')\cdot \mu + (1-\sigma'') \nu\big)/ \rho} q(\lambda, \rho, y, \mu, \nu) ( \rho'')^{m-(k+k')/2+2n/4} (\sigmat \sigma'')^{r_{\rb}-k/2}  \\
\times 
\int  e^{i\Phi'_1(y',v)/ \rho'} e^{i\Phi'_2(y', y'', \sigmat \sigma'', v,w)/ \rho''} a(\lambda,  \rho'', y'', y', \sigmat \sigma'',v,w) \, dv \, dw \, d\mu \, d\nu \, \frac{dy'' \, d \rho''}{( \rho'')^{n+1}} \,  \omegab.
\end{multline*}
This can be written
\begin{multline*}
(2\pi)^{-n} \int  e^{i\big((y-y'')\cdot \mu + (1-\sigma'') \nu + \sigma''\Phi'_2(y', y'', \sigmat \sigma'', v,w) \big)/ \rho} q(\lambda,  \rho, y, \mu, \nu) ( \rho/\sigma'')^{m-(k+k')/2}   \\
\times (\sigmat \sigma'')^{r_{\rb}-k/2}
 e^{i\Phi'_1(y',v)/ \rho'}  a(\lambda,  \rho / \sigma'', y'', y', \sigmat \sigma'', v, w) \, dv \, dw \, d\mu \, d\nu \, \frac{dy'' \, d\sigma''}{\sigma''} \, \omegab. 
\end{multline*}
To investigate the behaviour of this integral locally near a point $(x=0, \sigmat = 0, y, y') \in \bfc \cap \rb$, we  perform stationary phase in the $(y'', \sigma'', \mu, \nu)$ variables. The phase has a nondegenerate stationary point where $y'' = y, \sigma'' = 1, \mu = d_y \Phi'_2, \nu = \Phi'_2 + \sigmat d_{\sigmat} \Phi'_2$,  and we get an asymptotic expansion as $ \rho \to 0$ of the form
$$
 \rho^{m-(k+k')/2+2n/4} \sigmat^{r_{\rb}-k/2} 
\int e^{i\Phi'_1(y',v)/ \rho'} e^{i\Phi'_2(y, y', \sigmat, v,w)/ \rho} \tilde a(\lambda,  \rho, y, y', \sigmat, v, w) \, dv \, dw  \, \omegab, 
$$
where 
\begin{multline}
\tilde a(\lambda,  \rho, y, y', \sigmat, v, w) = 
 \sum_{j=0}^M \rho^j \Big(  \frac{(-i(\partial_{y''} \cdot \partial_\mu + \partial_{\sigma''} \partial_\nu))^j}{j!}  q(\lambda,  \rho, y, \mu, \nu)  \\ \times (\sigma'')^{-m+r_{\rb}+k'/2} a(\lambda, \rho'', y'', y', \sigmat \sigma'',  v, w) \Big) \Bigg|_{y=y'', \sigma'' = 1, \mu = d_y \Phi'_1, \nu =\Phi'_2 + \sigmat d_{\sigmat} \Phi'_2} \!\!\!\!\!\!\!\!+ O( \rho^{M+1}) .
\label{comp-symbol-formula-4}\end{multline}
This is a Legendre distribution associated to $\Lambda$ of the same order as $F$, and with the same index family. Moreover, we see from the formula  \eqref{comp-symbol-formula-4} that the microlocal support  $\WF'_{\bfc}(QF_3)$ is contained in $\WF'_{\bfc}(F)$, as well as contained in $\pi_L^{-1} \WF'(Q)$. Finally, if $q = 1 + O( \rho^\infty)$ on $\pi_L(\WF'_{\bfc}(F))$, then in the sum over $j$  in \eqref{comp-symbol-formula}, only the $j=0$ term is nonzero, because in all other terms, either $a = 0$ or $q=1 + O( \rho^\infty)$ (implying that any derivative of $q$ is $O( \rho^\infty)$) when evaluated at $y=y'', \sigma'' = 1, \mu = d_y \Phi'_2, \nu =\Phi'_2 + \sigma d_\sigma \Phi'_2$. Therefore, in this case, $QF_3 = F_3$ mod $O(x^\infty)$. 

$\bullet$ $F_5$ term. Writing $F_5$ in the form \eqref{u5}, we investigate $QF_5$ near a point $(z,  \rho', y')$ where $z \in M^\circ$. In this case, we can find a neighbourhood $W$ of $z$ with $\overline{W} \subset M^\circ$, and then the set 
$$
\{ (z,z') \in \supp Q \mid z \in W \} 
$$
is contained in $W \times W'$ for some $W'$ with $\overline{W'} \subset M^\circ$, since the support of $Q$ is contained in the set where $\sigma \in [1/2, 2]$. But in $W \times W'$, the kernel of $Q$ is smooth since $Q$ has differential order $-\infty$. Therefore, in this region the composition is given by an integral
$$
\int Q(z,z'') ( \rho')^{r_{\rb}-k/2} \int e^{i\Phi_1(y',v)/ \rho'} a(\lambda, z'', y',  \rho', v) \, dv \, dz'' \, \omegab
$$
with $Q(z,z'')$ smooth, and this has the form 
$$
( \rho')^{r_{\rb}-k/2} \int e^{i\Phi_1(y',v)/ \rho'} \tilde a(\lambda, z, y',  \rho', v) \, dv \, \omegab
$$
for some $\tilde a$ depending polyhomogeneously on $\lambda$ and  smoothly in its other arguments. Moreover, if for a fixed $(\lambda, y', v)$, $a$ is $O(( \rho')^\infty)$ in a neighbourhood of
$$
\{ (\lambda, z, y', 0, v) \mid z \in M \},
$$
then  the same is true of $\tilde a$. Therefore, $\WF'_{\rb}(QF_5)$ is contained in $\WF'_{\rb}(F_5)$ but is (in general) no smaller. 

$\bullet$ Since $\WF'(F_6) = \WF'(QF_6) = \emptyset$, the $F_6$ term makes no contribution to the wavefront set.

This completes the proof. 
\end{proof}

A similar result holds if $F$ is associated to a Legendre conic pair rather than a single Legendre submanifold. However, rather than give a full analogue of the result above, we give the following special cases which suffice for our needs. 

\begin{lem}\label{QFQ'-microsupport-conic}
(i) Suppose that $F \in I^{m, p; r_{\lb}, r_{\rb}; \mcA}(\MMkb, (\Lambda, \Lambda^\sharp); \Omegakbh)$ is a Legendre distribution on $\MMkb$ associated to a conic Legendrian pair $(\Lambda, \Lambda^\sharp)$,  and suppose that $Q \in \Psi^{-\infty}_k(M; \Omegakbh)$ is a scattering pseudodifferential operator such that $Q$ is microlocally equal to the identity operator near $\pi_L(\Lambda \cup \Lambda^\sharp)$. Then $Q F  - F \in I^{\infty, \infty; \infty, r_{\rb}; \mcA}(\Lambda, \Lambda^\sharp)$, so vanishes to infinite order at $\lb$ and $\bfc$. Similarly, if $F Q - F \in I^{\infty, \infty; r_{\lb}, \infty; \mcA}(\Lambda, \Lambda^\sharp)$ vanishes to infinite order at $\bfc$ and $\rb$.

(ii) Suppose that $F$ is as above,  and suppose that $Q$, $Q'$ are scattering pseudodifferential operators as above. If 
\begin{equation}
\pi_L^{-1} \WF'(Q) \cap \pi_R^{-1} \WF'(Q') \cap \Lambda^\sharp = \emptyset,\label{WF'QFQ'-conic}\end{equation}
 then $Q F Q' \in I^{m,r}(\MMkb, \Lambda; \Omegakbh)$; in particular, $\WF'_{\bfc}(QFQ')$ is disjoint from $(\Lambda^\sharp)'$.  \end{lem}

\begin{proof} The proof of (i) is similar to above. To prove (ii), decompose  $F = F_\Lambda + F_{\sharp}$, where $F_\Lambda \in I^{m,r}(\MMkb, \Lambda; \Omegakbh)$ is a Legendre distribution associated only to $\Lambda$ and $F_{\sharp}$ is localized sufficiently close to $\Lambda^\sharp$. Here, sufficiently close means that when we write down $QF_\sharp Q'$ as an (sum of) integral(s), using a phase function that  local parametrizes of $(\Lambda, \Lambda^\sharp)$, then \eqref{WF'QFQ'-conic} implies that the total phase is non-stationary on the support of the integrand. The usual integration-by-parts argument then shows that this kernel is  rapidly decreasing at $\bfc, \lb, \rb$ and hence trivially satisfies the conclusion of the lemma. On the other hand, Lemma~\ref{QFQ'-microsupport} applies to $F_\Lambda$ and completes the proof. 
\end{proof}


\section{Low energy estimates on the spectral measure}
\subsection{Pointwise bounds on Legendre distributions}\label{leb}

Now we give a pointwise estimate on Legendre distributions of a particular type. First we begin with a trivial estimate. 

\begin{prop}\label{triv-1} Let $\Lambda \subset \Tscstar_{\bfc}(\MMb)$  be a Legendre  submanifold  that projects diffeomorphically to\footnote{In this subsection, $\bfc$ denotes the boundary hypersurface of $\MMb$ (as opposed to $\MMkb$).} $\bfc$. Suppose that $u \in I^{-n/2 - \alpha, -\alpha, -\alpha; \mcA}(\MMkb, \Lambda)$. 
Let 
\begin{equation}
b = \min(\min \mcA_{\bfo} + n, \, \min \mcA_{\lb} + n/2, \, \min \mcA_{\rb} + n/2, \, \min \mcA_{\zf}).
\label{b}\end{equation}
Then, as a multiple of the half-density $|dg dg' d\lambda/\lambda|^{1/2}$, 
we have a pointwise estimate
$$
|u| \leq \lambda^b (\rho^{-1} + (\rho')^{-1})^{\alpha}.
$$
\end{prop}
This is trivial since in this case, $u$ may be written as an oscillatory function with no integration, and the order of vanishing/growth at the boundary may be determined by inspection from \eqref{u1} --- \eqref{u5}. (The discrepancies of $n$ and $n/2$ in \eqref{b} come about from comparing 
the nonvanishing half-density $\omegab$ on $\MMkb$ with the metric half-density $|dg dg' d\lambda/\lambda|^{1/2} = \rho_{\lbo}^{-n/2} \rho_{\rbo}^{-n/2} \rho_{\bfo}^{-n} \, \omegab$.)

Now consider a situation in which the Legendre submanifold does not project diffeomorphically to $\bfc$. Let $\ddiagb$ denote the boundary of the diagonal in $\MMb$, as in \eqref{ddiagb}. Recall that we have coordinates $(y, y', \sigma)$ on $\bfc$ near $Z$. Let $w = (y - y', \sigma - 1)$, 
 and let $\kappa$ be the corresponding scattering coordinates dual to $w$. Then $\ddiagb$ is given by $\{ w = 0  \}$ as a submanifold of $\bfc$ and the contact form on $\Tscstar_{\bfc} \MMb$ takes the form
\begin{equation}
d\nu - \mu \cdot dy - \kappa \cdot dw.
\label{cf}\end{equation}
In these coordinates, the Legendre submanifold $\Nscstar \ddiagb$ is given by $\{  w = 0, \mu = 0, \nu = 0 \}$. Let $\Lambdabf$ be a Legendre submanifold contained in $\Tscstar_{\bfc} \MMb$, denote by $\pi$ the natural projection from $\Tscstar_{\bfc} \MMb \to \bfc$, and for any $q \in \Lambdabf$ denote by $d\pi$ the induced map from $T_q \Lambdabf \to T_{\pi(q)} \bfc$. We consider the following situation in which the rank of $d\pi$ is allowed to change. 

\begin{prop}\label{Leg-bd-1} Let   $\Lambdabf$ be as above. Suppose   that $\Lambdabf$ intersects $\Nscstar \ddiagb$  at $\Gbf = \Lambdabf \cap \Nscstar \ddiagb$ which is codimension 1 in $\Lambdabf$, and suppose that $\pi |_{\Gbf}$ is a fibration, with $(n-1)$-dimensional fibres, to $\ddiagb$. 
Assume further that $d\pi$ has full rank on $\Lambdabf \setminus {\Gbf}$, while 
\begin{equation}
\text{$\det d\pi$ vanishes to  order exactly $n-1$ at ${\Gbf}$.}
\label{det-ass}\end{equation}

Suppose $u \in I^{-n/2-\alpha, -\alpha, -\alpha; \mcA}(\MMkb, \Lambdabf; \Omegakbh)$, and suppose that the (full) symbol of $u$ vanishes to order $(n-1)/2 + \alpha$ on ${\Gbf} \times [0, \lambda_0]$, where $(n-1)/2 + \alpha \in \{ 0, 1, 2, \dots \}$. Then  as a multiple of the scattering half-density $|dg dg' d\lambda/\lambda|^{1/2}$, we have a pointwise estimate 
\begin{equation} 
|u| \leq  C \lambda^b \big( 1+ \frac{ |w|}{\rho} \big)^\alpha \sim C\lambda^b (1 + \lambda d(z, z'))^\alpha
\label{Linfty-est}
\end{equation}
with $b$ as in \eqref{b}. Here $d(z,z')$ is the Riemannian distance between $z, z' \in M^\circ$. 
\end{prop}

\begin{remark} Notice that the condition on $\pi$ at ${\Gbf}$ implies that $d\pi$ has corank at least $n-1$ on ${\Gbf}$, hence that $\det d\pi$ must vanish to order at \emph{least} $n-1$ there. Condition \eqref{det-ass} is therefore  that the order of vanishing at ${\Gbf}$ is the least possible, which is a nondegeneracy assumption concerning the manner in which the rank of the projection changes at ${\Gbf}$. It implies, in particular, that $\Lambdabf$ intersects $\Nscstar \ddiagb$ cleanly. 
\end{remark}

\begin{proof}
Let $q$ be an arbitrary point in ${\Gbf}$. By rotating in the $w$ variables, we can ensure that $d\kappa_1 |_{\Gbf}$ vanishes at $q$ (since $\kappa_1, \dots, \kappa_n$ are coordinates on the fibres of $\Nscstar \ddiagb \to \ddiagb$, and since $\pi |_{\Gbf} : {\Gbf} \to \ddiagb$ has $(n-1)$-dimensional fibres). We claim that $(y, w_1, \kappa_2, \dots, \kappa_{n})$ furnish coordinates on $\Lambdabf$ locally near $q$. 
To see this, first note that  
$d\kappa_2 |_{\Gbf}, \dots, d\kappa_n |_{\Gbf}$ are linearly independent at $q$, and furnish coordinates on the fibres of ${\Gbf} \to \ddiagb$. Next, since $\ddiagb$ is $(n-1)$-dimensional, ${\Gbf}$ is $2(n-1)$-dimensional, and the fibres of ${\Gbf} \to \ddiagb$ are $(n-1)$-dimensional, it follows that ${\Gbf} \to \ddiagb$ is a submersion. Since $y_i$ are local coordinates on the base $\ddiagb$, 
we see that $(y, \kappa_2, \dots, \kappa_{n})$ furnish coordinates on ${\Gbf}$ locally near $q$. Since $w_1 = 0$ on ${\Gbf}$, to prove the claim it suffices to show that $dw_1 |_{\Lambdabf} \neq 0$ at $q$. 

To see this, we use \eqref{det-ass} which implies that $d\pi$ has corank exactly $n-1$ at $q$, and hence there is a tangent vector $V \in T_q \Lambdabf$ such that $d\pi(V)$ is not tangent to $\ddiagb$. Therefore, it has a nonzero $\partial_{w_j}$ component, which means that some $dw_j$ does not vanish at $q$ when restricted to $\Lambdabf$. But since $\Lambdabf$ is Legendrian, the form \eqref{cf} vanishes when restricted to $\Lambdabf$, which implies that its differential $\omega \equiv d\mu \cdot dy + d\kappa \cdot dw$ also vanishes on $\Lambdabf$. Hence $\omega(\partial_{\kappa_j}, V) = 0$ at $q$, $j \geq 2$, since $\partial_{\kappa_j}$ and $V$ are both tangent to $\Lambdabf$. But this implies that $dw_j(V) = 0$ for $j \geq 2$, i.e. $V$ has no $\partial_{w_j}$ component for $j \geq 2$. It follows that $dw_1(V) \neq 0$, showing that $dw_1 |_{\Lambdabf} \neq 0$ at $q$.  It follows that $(y, w_1, \kappa_2, \dots, \kappa_{n})$ indeed furnish coordinates on $\Lambdabf$ locally near $q$. We will use the notation 
$\ybbar = (w_2, \dots, w_{n})$ and $\mubbar = (\kappa_2, \dots, \kappa_n)$. Notice that $w_1 |_{\Lambdabf}$ is a boundary defining function for ${\Gbf}$, as a submanifold of $\Lambdabf$, locally near $q$.

Now we write the other coordinates on $\Lambdabf$ as functions of $(y, \ybar, \mubbar)$ as follows:
\begin{equation}
\ybbar_i = W_i(y, \ybar, \mubbar), \ \mu_i = M_i(y, \ybar, \mubbar), \ \kappa_1 = \Mbar(y, \ybar, \mubbar), \ \nu = N(y, \ybar, \mubbar) \text{ on } \Lambdabf. 
\end{equation}
Notice that the vanishing of \eqref{cf} on $\Lambdabf$ implies that 
\begin{equation}
dN = \sum_{i=1}^{n-1} M_i dy_i + K dw_1 + \sum_{j=2}^n \kappa_j d W_j \text{ on } \Lambdabf. 
\label{WKN}\end{equation}
By equating the coefficients of $d\mubbar$, $dy$ and $dw_1$ on each side of \eqref{WKN}, we obtain 
 the following  identities:
\begin{equation}\begin{gathered}
\sum_{j=2}^n v_j \frac{\partial W_j(y, \ybar, v)}{\partial v_i} = \frac{\partial N(y, \ybar, v)}{\partial v_i}, \quad i = 2 \dots n, \\
\sum_{j=2}^n v_j \frac{\partial W_j(y, \ybar, v)}{\partial y_i} + M_i(y, \ybar, v) = \frac{\partial N(y, \ybar, v)}{\partial y_i}, \quad i = 1 \dots n-1, \\
\sum_{j=2}^n v_j \frac{\partial W_j(y, \ybar, v)}{\partial w_1} + \Mbar(y, \ybar, v) = \frac{\partial N(y, \ybar, v)}{\partial w_1}. 
\end{gathered}\label{identities}\end{equation}
We claim that the function
\begin{equation}
\Phi(y, \ybar, \ybbar, v) = \sum_{j=2}^n\big( \ybbar_j - W_j(y, \ybar, v) \big) v_j + N(y, \ybar, v)
\label{Phi-defn}\end{equation}
parametrizes $\Lambdabf$ locally near $q$. Notice that $W$, $M$ and $N$ are all $O(\ybar)$ at $q$. Hence, $\Phi = \ybbar \cdot v + O(\ybar)$, 
so $d_{v_j} \Phi = {\ybbar}_j + O(\ybar)$, $2 \leq j \leq n$,  have linearly independent differentials at the point $\tilde q = (y(q), w=0, \nu=0, \mu=0, \kappa_1 = 0, \mubbar(q))$ corresponding to $q$, i.e. $\Phi$ is a nondegenerate parametrization of $\Lambdabf$ near $q$. Next, using the first equation in \eqref{identities} we find that 
\begin{equation}
d_{v_j} \Phi = {\ybbar}_j - W_j(y, \ybar, v).  
\end{equation}
So $\ybbar = W$ when $d_v \Phi = 0$. The Legendrian submanifold parametrized is then given by (using \eqref{identities})
\begin{equation}\begin{gathered}
\Big\{ \big(y, \ybar, W, -v \cdot \frac{\partial W}{\partial y} + \frac{\partial N}{\partial y}, -v \cdot \frac{\partial W}{\partial \ybar} + \frac{\partial N}{\partial \ybar}, v, N \big) \Big\} \\
= \big\{ (y, \ybar, W, M, \Mbar, v, N) \big\} = \Lambdabf.
\end{gathered}\end{equation}

Notice that the second derivative matrix $d^2_{vv}\Phi$ vanishes at $\ybar = 0$. Therefore we can write $d^2_{vv}\Phi = \ybar A + O(\ybar^2)$, where $A$ is a smooth $(n-1) \times (n-1)$ matrix function of $(\oly, v)$, where we write $\oly = (y, \ybar, \ybbar)$. We claim that $A$ is invertible at (and therefore, near) $\tilde q$. To see this, we start from the fact that the map 
$$
\{ (\oly, v)  \} \to \{ (\oly, d_\oly\Phi, \Phi, d_v \Phi) \}
$$
is locally a diffeomorphism onto its image. (This follows directly from the nondegeneracy condition on $\Phi$, that the differentials $d(\partial \Phi / \partial v_j)$ are linearly independent.) Note that the determinant of
 the differential of the map 
$$
\{ (\oly, d_\oly\Phi, \Phi, d_v \Phi) \} \to \{ (\oly, d_v\Phi) \}
$$
is equal to the determinant of the differential of the map
$$
\{ (\oly, d_\oly\Phi, \Phi, d_v \Phi) \mid d_v \Phi = 0 \} \to \oly ,
$$
and this map is $\pi |_{\Lambdabf}$ (in local coordinates). It  follows that the order of vanishing of $\det d\pi$ at $q$ is the same as the order of vanishing of the determinant of the differential of the map 
$$
\{ (\oly, v) \} \to \{ (\oly, d_v\Phi) \}
$$
at $\tilde q$. But this determinant is simply $\det d^2_{vv}\Phi$. It follows from \eqref{det-ass} that $\det d^2_{vv}\Phi$ vanishes to order exactly $n-1$ at $\tilde q$. But this implies that the matrix $A$ is invertible at $\tilde q$, as claimed.

Now we write $u$ as an oscillatory integral. It suffices to prove the proposition assuming that $u$ has symbol supported close to $q$ and that $u$  itself is supported close to $\ddiagb$, since away from $\ddiagb$ the result follows from Proposition~\ref{triv-1}. It can then be written with respect to the phase function $\Phi$: modulo a smooth term vanishing to order $O(x^\infty)$, $u$ is a multiple of the scattering half-density $|dg dg' d\lambda/\lambda|^{1/2}$ given by
\begin{equation}
\rho^{-(n-1)/2 - \alpha} \lambda^n \int e^{i\Phi(y, w,v)/\rho} a(\lambda, \rho,y,v) \, dv |dg dg' d\lambda/\lambda|^{1/2}. 
\label{toest}\end{equation}
Moreover, we may assume that $a$ is a function only of $\lambda, \rho, y, \ybar$ and $v$, polyhomogeneous conormal in $\lambda$ with index set $\mcA_{\bfo}$, smooth and compactly supported in the remaining variables, and vanishing to order $(n-1)/2 + \alpha$ at $\rho=\ybar = 0$. It can therefore be written 
\begin{equation}
a = \sum_{j=0}^{(n-1)/2 + \alpha - 1} \rho^j \ybar^{(n-1)/2 + \alpha-j} a_j(\lambda, y, \ybar, v) + \rho^{(n-1)/2 + \alpha} b(\lambda, \rho, y, \ybar, v). 
\label{a-sum}\end{equation}
Note that the estimate is trivial  if  $|w_1| \leq \rho$, since then the integrand is uniformly bounded, and hence the integral is uniformly bounded in agreement with the estimate \eqref{Linfty-est} (since $|\ybar|$ is locally comparable to $|w|$). From now on, then, we will assume that $|w_1| \geq \rho$. We begin by estimating the $a_0$ term. 

Now, for fixed $\ybar \neq 0$, let us change variable from $v_1, \dots, v_{n-1}$ to $\theta_1, \dots, \theta_{n-1}$, where
\begin{equation}
\theta_i = \ybar^{-1/2} d_{v_i} \Phi.
\label{theta-var}\end{equation}
Then 
\begin{equation}
\frac{\partial \theta_i}{\partial v_j} = \ybar^{-1/2} d^2_{v_i v_j} \Phi = \ybar^{1/2} A_{ij},
\label{dwdv}\end{equation}
where $A_{ij}$ is nonsingular as we have noted above. Therefore,
\begin{equation}
\frac{\partial \Phi}{\partial \theta} = \Big( \frac{\partial \theta}{\partial v} \Big)^{-1} \frac{\partial \Phi}{\partial v} = A^{-1} \theta.
\label{A^{-1}theta}\end{equation}
This shows that the $\theta$ coordinates are suitable coordinates in which to perform stationary phase computations. We proceed with a standard argument, which can be found in Sogge's book \cite{Sogge}, for example. We use the identity
\begin{equation*}
e^{i\Phi/\rho} = \Big( \frac{\rho}{\ybar^{1/2} i \theta_j} \frac{\partial}{\partial v_j} \Big) e^{i\Phi/\rho},
\end{equation*}
which can be written
\begin{equation}
e^{i\Phi/\rho} =  \Big( \sum_k \frac{\rho}{ i \theta_j} A_{jk} \frac{\partial}{\partial \theta_k} \Big) e^{i\Phi/\rho}.
\label{ibp}\end{equation}
We also need the following observation: by applying \eqref{dwdv} repeatedly, we obtain
\begin{equation}
\big| \frac{\partial^{|\alpha|} A}{\partial^\alpha \theta} \big| \leq C |w_1|^{-|\alpha|/2} \leq C \rho^{-|\alpha|/2}.
\label{A-derivs}\end{equation}

In the $\theta$ coordinates, we are trying to prove the estimate
\begin{equation}
\Big| \rho^{-(n-1)/2 - \alpha} \int_{\RR^{n-1}} \ybar^\alpha e^{i\Phi(y, w,\theta)/\rho} \tilde a_0(\lambda, \rho, y, w_1, \theta) \, d\theta \Big| \leq C \big( \frac{\ybar}{\rho}\big)^{\alpha} \lambda^b.
\label{wts}\end{equation}
Here the $\ybar^{(n-1)/2}$ factor was absorbed as a Jacobian factor, and $\tilde a_0$  is again smooth. Clearly this is equivalent to a uniform bound on 
\begin{equation}
\Big| \rho^{-(n-1)/2} \lambda^{-b} \int_{\RR^{n-1}} e^{i\Phi(y, w,\theta)/\rho} \tilde a_0(\lambda, \rho, y, w_1, \theta) \, d\theta \Big| .
\label{uniform-bd}\end{equation}
We introduce a partition of unity in $(\rho, \theta)$-space, $1 = \chi_0 + \sum_{j=1}^{n-1} \chi_j$, where $\chi_0$ is a compactly supported function of $\theta/\sqrt{\rho}$, and $\chi_j$ is supported where $|\theta|\geq  \sqrt{\rho}$, and where $\theta_j \geq  |\theta|/(n-1)$. We can do this with derivatives estimated by 
\begin{equation}
| \nabla_\theta^{(k)} \chi_k | \leq C \rho^{-k/2}.
\label{chi-derivs}\end{equation}

The integral with $\chi_0$ inserted is trivial to estimate since it occurs on a set of measure $\rho^{(n-1)/2}$. With $\chi_j$ inserted, we use the identity 
\eqref{ibp} $M$ times, for $M$ a sufficiently large integer. Thus we consider 
$$
\rho^{-(n-1)/2} \int \chi_j \Big( \sum_k \frac{\rho}{ i \theta_j} A_{jk}(y, \theta) \frac{\partial}{\partial \theta_k} \Big)^M e^{i\Phi(y, w,\theta)/\rho} \tilde a_0(\lambda, \rho, y, w_1, \theta) \, d\theta
$$
and integrate by parts $M$ times. The result can be estimated by
\begin{equation}
C \rho^{-(n-1)/2+M} \sum_{k=0}^M  \rho^{-(M-k)/2} \int_{|\theta| \geq \sqrt{\rho}} 1_{\supp \,  \chi_j} \theta_j^{-M-k}  \, d\theta
\label{sum}\end{equation}
where $M-k$ derivatives fall on the $\chi_j$ or $A_{jk}$ terms (via \eqref{A-derivs} and \eqref{chi-derivs}), and at most $k$ fall on a $\theta_j^{-p}$ term. Note that on the support of $\chi_j$, we can estimate $\theta_j^{-1} \leq c |\theta|^{-1}$. 
The $\theta$ integral is absolutely convergent for $M > n-1$, and 
$$
\int_{|\theta| \geq \sqrt{\rho}}  |\theta|^{-M-k}  \, d\theta = C_k \rho^{-(M+k)/2 + (n-1)/2}
$$
since $\dim \theta = n-1$. Substitution of this into \eqref{sum} gives a uniform bound since $\tilde a$ is polyhomogeneous in $\lambda$ with index set $\mcA_{\bfo} + n$. Moreover, since $\Phi$ and $\tilde a$ are smooth in $\ybar$, the bound is uniform as $\ybar \to 0$. 

To treat the terms $a_i$ for $i > 0$ and $b$ in \eqref{a-sum}, we perform the same manipulations as above, and we end up with a uniform bound times $C\rho^{i} \ybar^{-i}$, which is bounded for $\rho \leq \ybar$. This completes the proof. 
\end{proof}

\subsection{Geometry of $L^{\bfc}$}
We collect here some facts concerning the geometry of the Legendre submanifold $L^{\bfc}$ (see Section~\ref{sec:MMkb}).  We begin by defining 
$$
\Gbf = \{ q = (y, y, \sigma, \mu, -\mu, \nu, -\nu) \in \Nscstar \ddiagb; \nu^2 + h^{ij} \mu_i \mu_j = 1 \}.
$$
Clearly, $\Gbf$ is an $S^{n-1}$-bundle over $\ddiagb$. 

\begin{lem}\label{L-geometry-2}
 The Legendre submanifold $\Ndiag$
intersects $L^{\bfc}$ cleanly at $\Gbf$, and the projection $\pi : L^\bfc \to \bfc$ satisfies \eqref{det-ass}. 
\end{lem}

\begin{proof} According to \cite{HV2}, the Legendre submanifold $L^{\bfc}$ is given by the flowout from $\Gbf$ by the vector field
\begin{equation}
V_l = -\nu (\sigma \dbyd{}{\sigma} + \mu \dbyd{}{\mu}) + h \dbyd{}{\nu} + \dbyd{h}{\mu_i}\dbyd{}{y_i} - \dbyd{h}{y_i}\dbyd{}{\mu_i}
, \quad h = \sum_{i,j} h^{ij}(y) \mu_i \mu_j 
\label{Vl}\end{equation}
(see \cite[Section 3.1]{GHS}). Observe that  at least one of the coefficients of $\partial_\sigma$ or $\partial_\nu$ is nonvanishing, so either $\dot \sigma \neq 0$ or $\dot \nu + \dot \nu' \neq 0$ under the flowout by $V_l$. Since $\sigma = 1$ and $\nu + \nu' = 0$ at $\Ndiag$, we see that $V_l$ is everywhere transverse to $\Ndiag$, so
$\Gbf$ has codimension $1$ in $L^{\bfc}$, and intersects $L^{\bfc}$ cleanly. 

It remains to show that the projection $\pi$ from $L^{\bfc}$ to $\bfc$ satisfies \eqref{det-ass}. First we choose coordinates on $\Lbf$. Near a point on $\Lbf$ at which $|\mu|_h^2 := h^{ij} \mu_i \mu_j < 1$, and therefore $\nu \neq 0$, we can choose coordinates $(\mu, y', \epsilon)$ where $\epsilon$ is the flowout time from $\Gbf$ along the vector field $V_l$. Coordinates on the base are $(y, y', \sigma)$. With the dot indicating derivative along the flow of $V_l$, i.e. $d/d\epsilon$, we have 
$$\begin{aligned}
\dot \sigma &= -\nu \\
\dot y^i &= 2 h^{ij} \mu_j 
\end{aligned} \quad \text{ at } \Gbf.
$$
It follows that 
$$\begin{aligned}
\sigma &= 1 - \nu \epsilon + O(\epsilon^2), \\
y^i &= (y')^i + 2h^{ij} \mu_j \epsilon + O(\epsilon^2) 
\end{aligned}$$
and we see that near $\Gbf$,
$$
\frac{\partial \sigma}{\partial \epsilon} \neq 0, \quad 
\frac{\partial y^i}{\partial \mu_j} = \epsilon h^{ij} + O(\epsilon^2),
$$
which, using the positive-definiteness of $h^{ij}$, shows that $\det d\pi$, where $\pi$ is the map
$$
\Lbf \ni \big(\mu, y', \epsilon\big) \mapsto \Big( y(\mu, y', \epsilon), y', \sigma(\mu, y', \epsilon) \Big),
$$
vanishes to order exactly $n-1$ as $\epsilon \to 0$. 

On the other hand, near a point on $\Lbf$ at which $|\mu| = 1$, we can choose a coordinate $\mu_i$ which is nonzero. Without loss of generality we suppose that $i = 1$. Then write $\yybar = (y_2, \dots, y_{n-1})$ and $\mmubar = (\mu_2, \dots,  \mu_{n-1})$. We can take $(\nu, \mmubar, y', \epsilon)$ as coordinates on $\Lbf$. Calculating as above, we find that 
$$\begin{aligned}
y^1 &= y_1' + 2h^{1j} \mu_j \epsilon + O(\epsilon^2), \\
y^i &= (y')^i + 2h^{ij} \mu_j \epsilon + O(\epsilon^2),\quad i \geq 2, \\
\sigma &= 1 - \nu \epsilon + O(\epsilon^2)
\end{aligned}$$
which shows that
$$
\frac{\partial y_1}{\partial \epsilon} > 0, \quad 
\frac{\partial \yybar^i}{\partial \mmubar_j} = \epsilon h^{ij} + O(\epsilon^2),  
\quad \frac{\partial \sigma}{\partial \nu} = -\epsilon + O(\epsilon^2).
$$
Again we find that $\det d\pi$, where $\pi$ is the map
$$
\Lbf \ni \big(\nu, \mmubar, y', \epsilon \big) \mapsto \Big( y(\nu, \mmubar, y', \epsilon), y', \sigma(\nu, \mmubar, y', \epsilon) \Big),
$$
vanishes to order exactly $n-1$ as $\epsilon \to 0$. 
\end{proof}

\begin{lem}\label{L-geometry-1}
There exists $\delta > 0$ such that, if 
$$
q = (y, y', \sigma, \mu, \mu', \nu, \nu') \in L^{\bfc} \text{ and } |\nu + \nu'| < \delta,
$$
then either $q \in \Gbf$, or $d\pi : T_q \Lbf \to T_{\pi(q)} \bfc$ is invertible, and hence $\pi : L \to \bfc$ is a diffeomorphism locally near $q$. 
\end{lem}

\begin{proof} We use the explicit description of $\Lbf$ given in \cite[Section 5]{HV2}:
\begin{equation}\begin{split}\label{eq:sp-1c}
\Lbf 
= \, &\{(\sigma,y,y',\nu,\nu',\mu,\mu'):\ \exists(y_0,\muh)\in S^* (\partial M),
\ s,s'\in(0,\pi),\ \text{s.t.}\\
&\quad \sigma  = \frac{\sin s}{\sin s'},
\ \nu=-\cos s,
\ \nu'=\cos s',\\
\quad(y,\mu)&=\sin s\exp(sH_{\half h})(y_0,\muh),
(y',\mu')=-\sin s' \exp(s' H_{\half h})(y_0,\muh)\}\\
&\cup T_+ \cup T_- \cup F_+ \cup F_-, \quad T_\pm  = 
\{(\sigma,y,y,\pm 1,\mp 1,0,0):\ \sigma\in(0,\infty),\ y\in\partial M\}, \\\
F_\pm = \, &\{ (\sigma, y, y', \pm 1, \pm 1, 0, 0) \mid \sigma\in(0,\infty),\ \exists \ \text{geodesic of length $\pi$ connecting $y, y'$.} \}.
\end{split}\end{equation}
We see that $\nu = -\nu'$ on $\Lbf$ only on $\Gbf \cup T_+ \cup T_-$. 
A compactness argument shows that that for any neighbourhood $U$ of $\Gbf \cup T_+ \cup T_-$, the set
$$
\{ (y, y', \sigma, \mu, \mu', \nu, \nu') \in L^{\bfc} \mid |\nu + \nu'| < \delta \}
$$
is contained in $U$ if $\delta$ is sufficiently small. So 
 it is enough to show that $\Lbf$ projects diffeomorphically to $\bfc$ in some neighbourhood of $\Gbf \cup T_+ \cup T_-$,  except at $\Gbf$ itself. 
Lemma~\ref{L-geometry-2} shows that $L^{\bfc} \subset \Tscstar_{\bfc} \MMb$ projects diffeomorphically to the base $\bfc$ in a sufficiently small deleted neighbourhood of $\Gbf$. 
Now consider a neighbourhood of $T_+ \cap \{ \sigma \leq 1 - \epsilon \}$ for some small $\epsilon$. As shown in \cite{HV2}, near this set, $(y', \mu', \sigma)$ are smooth coordinates. Also, we have
from \eqref{eq:sp-1c} 
$$
(y, \mu) = \sigma \exp\Big( \frac{s'-s}{\sin s'} H_{\half h} \Big)(y',  \mu'). 
$$
Using the expression \eqref{Vl} for the Hamilton vector field, we find that near $T_+$, we have
$$
y^i = {y'}^i + \frac{s'-s}{\sin s'} h^{ij}\mu'_j + O(|\mu'|^2) 
= (1 - \sigma) h^{ij}\mu'_j + O((\sin s)^2 + (\sin s')^2+ |\mu'|^2),
$$ 
which shows that at $T_+$, where $\sin s = \sin s' = \mu' = 0$, we have
$$
\frac{\partial y^i}{\partial \mu_j'} \Big|_{y', \sigma}= (1 - \sigma) h^{ij}.
$$
Since $(y', \mu', \sigma)$ furnish smooth coordinates near $T_+$, this equation and the positive-definiteness of $h^{ij}$ show that also $(y, y', \sigma)$ furnish smooth coordinates in a neighbourhood of $T_+$ when $\sigma < 1 - \epsilon$. (Of course, we know from Lemma~\ref{L-geometry-2} that this cannot hold uniformly up to $\sigma = 1$).  A similar argument holds for $\sigma > 1 + \epsilon$ and for $T_-$. 
\end{proof}

\begin{remark}\label{sigma-behaviour} These lemmas will be applied to distributions of the form 
\begin{equation}
Q(\lambda) dE_{\sqrt{L}}(\lambda) Q(\lambda),
\label{QEQ}\end{equation}
 where $Q$ is a  pseudodifferential operator with small microsupport. Notice that by taking the microsupport sufficiently small, we can localize the microsupport of \eqref{QEQ} to points $(y, y', \sigma, \mu, \mu', \nu, \nu')$ such that $y$ is close to $y'$, $\mu$ is close to $\mu'$ and $\nu$ is close to $\nu'$. However, we cannot localize so that $\sigma$ is close to $1$, simply because if $x, x' \in (0, \epsilon)$, then $\sigma = x/x'$ can take any value in $(0, \infty)$. Therefore, it is important to understand the properties of $\pi$ on $L$ near the whole of the sets $T_\pm$, not just close to $\Nscstar \ddiagb$. 
\end{remark}

\subsection{Proof of Theorem~\ref{main3}, part (A)}\label{(A)}
By Proposition~\ref{QQ1}, to prove part (A) of Theorem~\ref{main3} it is sufficient to prove 
Theorem~\ref{pointwisesmestimates} 
for  for  $\bL = \bH$ and for $\lambda \leq \lambda_0$, that is, to prove the estimates
\begin{equation} 
\Big| \big( Q_i(\lambda)   dE_{\sqrt{\bH}}^{(j)}(\lambda) Q_i(\lambda) \big) (z,z') \Big| \leq C \lambda^{n-1-j} \big( 1 + \lambda d(z,z') \big)^{-(n-1)/2 + j} , \quad j \geq 0. 
\label{spec-meas-j}\end{equation}

Our starting point is Theorem~\ref{lesmLd}. As an immediate consequence of this theorem, the $j$th $\lambda$-derivative $dE^{(j)}_{\sqrt{\bH}}(\lambda)$ is a Legendre distribution in the space $$I^{m-j,p-j; r_{\lb}-j, r_{\rb}-j; \mcA^{(j)}}(\MMkb, (L^{\bfc}, L^{\sharp, \bfc}); \Omegakbh),$$ where $\mcA^{(j)}$ is an index family with index sets at the faces $\bfo, \lbo, \rbo, \zf$ starting at order $-1-j$, $n/2 - 1-j$, $n/2 - 1-j$, $n-1-j$ respectively.

Next we choose a partition of unity. We choose $Q_0$ to be multiplication by the function $1-\chi(\rho)$, where $\chi(\rho) = 1$ for $\rho \leq \epsilon$ and $\chi(\rho) = 0$ for $\rho \geq 2\epsilon$, for some sufficiently small $\epsilon$. Then, $Q_0  dE^{(j)}_{\sqrt{\bH}}(\lambda) Q_0$ is polyhomogeneous on $\MMkb$, with index sets as above at  $\bfo, \lbo, \rbo, \zf$ and supported away from the remaining boundary hypersurfaces. Now recall that $|dg dg' d\lambda/\lambda |^{1/2}$ is equal to $\rho_{\bfo}^{-n} \rho_{\lbo}^{-n/2} \rho_{\rbo}^{-n/2}$ times a smooth nonvanishing section of the half-density bundle $\Omegakbh$. 
It is then immediate that $Q_0  dE^{(j)}_{\sqrt{\bH}}(\lambda) Q_0$ is bounded, as a multiple of $|dg dg' d\lambda/\lambda|^{1/2}$ by $\lambda^{n-1-j}$,
which yields \eqref{spec-meas-j} for $i=0$ since in this region we have $\lambda d(z,z') \leq C$. 

Next, we choose $Q'_1$ such that $\Id - Q'_1$ is microlocally equal to the identity for $|\mu|_h^2 + \nu^2 \leq 3/2$, and microsupported in $|\mu|_h^2 + \nu^2 \leq 2$. Let $Q_1 = \chi(\rho) Q_1'$. Then, we claim that $Q_1 dE^{(j)}_{\sqrt{\bH}}(\lambda)  Q_1$ has empty wavefront set, and is therefore polyhomogeneous with index sets at the faces $\bfo, \lbo, \rbo, \zf$ starting at order $-1$, $n/2 - 1$, $n/2 - 1$, $n-1$ respectively. To see this, we write 
\begin{multline}
Q_1  dE^{(j)}_{\sqrt{\bH}}(\lambda)  Q_1 =  dE^{(j)}_{\sqrt{\bH}}(\lambda) - (\Id - Q_1) dE^{(j)}_{\sqrt{\bH}}(\lambda) \\  -  dE^{(j)}_{\sqrt{\bH}}(\lambda)(\Id - Q_1)  + (\Id - Q_1) dE^{(j)}_{\sqrt{\bH}}(\lambda)    (\Id - Q_1).
\label{energy-cutoff}\end{multline}
Since $\Id - Q_1$ is microlocally equal to the identity on $\pi_L(\WF_{\bfc}' dE^{(j)}_{\sqrt{\bH}}(\lambda))$ and on $\WF_{\lb}'(dE^{(j)}_{\sqrt{\bH}}(\lambda))$, Lemma~\ref{QFQ'-microsupport} shows that  the sum of the first two terms on the right hand side above vanishes to infinite order at $\lb$ and $\bfc$, and similarly the sum of the third and fourth terms vanishes to infinite order at $\lb$ and $\bfc$. Now consider the multiplication of $\Id - Q_1$ on the right, and group together the first and third terms, and the second and fourth terms on the RHS. We see, using the adjoint of Lemma~\ref{QFQ'-microsupport} (since also $\Id - Q_1$ is microlocally equal to the identity on $\WF_{\rb}'(dE^{(j)}_{\sqrt{\bH}}(\lambda))$) 
that the sum of the first and third terms vanishes to infinite order at $\rb$, and similarly the sum of the second and fourth terms vanishes at $\rb$. Hence $Q_1  dE^{(j)}_{\sqrt{\bH}}(\lambda) Q_1$ vanishes to all orders at $\bfc, \lb, \rb$ and has empty wavefront set as claimed.  This piece therefore also satisfies \eqref{spec-meas-j}.

We now further decompose $\Id - Q_0 - Q_1 = \chi Q_1'$, which has compact microsupport, into a sum of terms. Choosing $\delta$ as in Lemma~\ref{L-geometry-1}, we partition the interval $[-2, 2]$ into $N-1$ intervals $B_i$ each of length $ \delta/2$, and choose a decomposition $\Id - Q_1 = \sum_{i=2}^N Q_i$ where $Q_i$, and hence also $Q_i^*$, is microsupported in the set $\{|\mu|_h^2 + \nu^2 \leq 2,  \nu \in 2B_i \}$ (where $2B_i$ is the interval with the same centre as $B_i$ and twice the length).  It follows that if $q' = (y, y', \sigma, \mu, \mu', \nu, \nu') \in (L^\bfc)'$ is such that $\pi_L(q') \in \WF'(Q_i)$ and $\pi_R(q') \in \WF'(Q_i^*)$, then $|\nu - \nu'| \leq  \delta$. Together with Lemma~\ref{QFQ'-microsupport-conic}, this means that $Q_i dE^{(j)}_{\sqrt{\bH}}(\lambda) Q_i^*$ is associated only to the Legendrian $L^\bfc$ and not to $L^{\sharp, \bfc}$, since on $(L^{\sharp, \bfc})'$ we have $|\nu - \nu'| = 2 > \delta$. 

Next, by Lemma~\ref{L-geometry-1}, if $q' = (y, y', \sigma, \mu, \mu', \nu, \nu')  \in (L^\bfc)'$ is such that $\pi_L(q') \in \WF'(Q_i)$ and $\pi_R(q') \in \WF'(Q_i^*)$, then due to our choice of $\delta$,  either $q \in \Gbf$, or locally near $q$, $L^\bfc$ projects diffeomorphically to $\bfc$. Therefore, the microsupport of $Q_i dE^{(j)}_{\sqrt{\bH}}(\lambda) Q_i^*$, $i \geq 2$, is a subset of $(L^\bfc)'$ which satisfies the conditions of either Proposition~\ref{triv-1} or  Proposition~\ref{Leg-bd-1}. 

In the case of Proposition~\ref{triv-1}, we have $b = n-1 - j$, $\alpha = -(n-1)/2 + j$ and estimate \eqref{spec-meas-j} follows directly.
Next consider the case of  Proposition~\ref{Leg-bd-1}. In this case, we have to determine the order of vanishing of the symbol of $Q_i dE^{(j)}_{\sqrt{\bH}}(\lambda) Q_i^*$ at $\Gbf$. 
Locally near $q \in \Gbf \cap L^\bfc$, $L^\bfc$ can be parametrized  by a phase function $\Phi$ that \emph{vanishes} at $\Gbf$ when $d_v \Phi = 0$ --- see \eqref{Phi-defn}. The kernel  $Q_i dE_{\sqrt{\bH}}(\lambda) Q_i^*$ is a Legendrian of order $-1/2$. Each time we apply a $\lambda$ derivative to $dE_{\sqrt{\bH}}(\lambda)$, it hits either the phase function or the symbol. If it hits the phase, then the order of the Legendrian is reduced by $1$, but it brings down a factor of $\Phi$ which vanishes at $\Gbf \times [0, \lambda_0]$. If it hits the symbol, then the order of the Legendrian is not reduced. Therefore, as a Legendrian of order $-1/2 - j$, the full symbol of $Q_i dE^{(j)}_{\sqrt{\bH}}(\lambda) Q_i^*$ vanishes to order $j$ at $\Gbf\times [0, \lambda_0]$. Therefore, we can apply Proposition~\ref{Leg-bd-1} with $b = n-1-j$ and $\alpha = -(n-1)/2 + j$, and we deduce \eqref{spec-meas-j} in this case. This concludes the proof of \eqref{spec-meas-j} and hence establishes Theorem~\ref{pointwisesmestimates} for low energies $\lambda \leq \lambda_0$.


\section{High energy estimates (in the nontrapping case)}

In the previous section we proved estimates on the spectral measure $dE_{\sqrt{\bH}}(\lambda)$ for $\lambda \in (0, \lambda_0]$. We now prove high energy estimates, i.e. estimates for $\lambda \in [\lambda_0, \infty)$. For convenience, we introduce the semiclassical parameter $h = \lambda^{-1}$, so that we are interested in estimates for $h \in (0, h_0]$, where $h_0 = \lambda_0^{-1}$. To do this, we use the description of the high-energy asymptotics of the spectral measure from \cite{HW}. The structure of the argument will be the same as in the previous
section, and our main task is to adapt each of the intermediate results --- Lemmas~\ref{QFQ'-microsupport} and \ref{QFQ'-microsupport-conic}, Propositions~\ref{triv-1} and \ref{Leg-bd-1},  Lemma~\ref{L-geometry-2} and Lemma~\ref{L-geometry-1} --- to the high-energy setting. \emph{Throughout this section we assume that the manifold $(M, g)$ is nontrapping.}

\subsection{Microlocal support} 
We begin by defining, by analogy with the discussion in Section~\ref{ows}, the notion of microlocal support  of a Legendre distribution on $\bX$. 

Let $\Lambda \subset \Tscstar_{\mf} \bX$ be the Legendre submanifold associated to $F$. We assume that $\Lambda$ is compact. 
Recall from \cite[Section 3]{HW} that $\Lambda$ determines associated Legendre submanifolds $\Lambda_{\bfc}$, $\Lambda_{\lb}$ and $\Lambda_{\rb}$ which are the bases of the fibrations on  $\partial_{\bfc} \Lambda$, $\partial_{\lb} \Lambda$ and $\partial_{\rb} \Lambda$, respectively. The Legendre submanifold $\Lambda_{\bfc}$ can be canonically identified with a Legendre submanifold of $\Tscstar_{\bfc} \MMb$, while $\partial_{\lb} \Lambda$ and $\partial_{\rb} \Lambda$ may be canonically identified with Legendre submanifolds of $\Tscstar_{\partial M} M$. We define $\Lambda'$ by negating the fibre coordinates corresponding to the right copy of $M$, i.e. 
$$
q' = (z,z', \zeta, \zeta') \in \Lambda' \iff q = (z, z', \zeta, -\zeta') \in \Lambda.
$$
Similarly we define $\Lambda_{\bfc}'$ and $\Lambda_{\rb}'$ as in the previous section. 

Then we define the microlocal support $\WF'(F)$ of $F \in I^m(\Lambda)$ to be a closed subset of 
$$
\Lambda'  \cup \Big( \Lambda_{\bfc}'\times [0, h_0] \Big) \cup \Big(\Lambda_{\lb} \times [0, h_0] \Big) \cup \Big(\Lambda'_{\rb}\times [0, h_0] \Big)
$$ 
in the same way as before: we say that $q' \in \Lambda'$ is not in $\WF'(F)$ iff there is a neighbourhood of $q \in \Lambda$ in which $F$ has order $-\infty$, in the sense of Section~\ref{ows}. That is, in a local oscillatory representation for $F$ of the form (for simplicity, where $q$ lies over the interior of $\MMb$),
$$
h^{m - k/2-n} \int_{\RR^k} e^{i\psi(z, v)/h} a(z, v, h) \, dv |dg dg' dh/h^2|^{1/2},
$$
where $q = (z_*, d_z \psi(z_*,v_*))$ and $d_v\psi(z_*, v_*) = 0$ (these conditions determining $(z_*, v_*)$ locally uniquely provided that $\psi$ is a nondegenerate parametrization of $\Lambda$), the condition that $F$ has order $-\infty$ in a neighbourhood of $q$ is equivalent to $a$ being $O(h^\infty)$ in a neighbourhood of the point $(z_*, v_*, 0)$. Similarly, $q' \in \Lambda_{\bfc}' \times [0, h_0]$ is not in $\WF'(F)$ iff there is a neighbourhood of $q \in \Lambda_{\bfc} \times [0, h_0]$ in which $F$ has order $-\infty$.

Similarly, $(\tilde q, h) \in \Lambda_{\lb} \times [0, h_0]$ is not in $\WF'(F)$ iff 
$F$ can be written modulo $(hxx')^\infty \CI(\MMb)$ using local oscillatory integral representations
with symbols that vanish in a neighbourhood of the \emph{fibre} in their domain corresponding to $(\tilde q, h)$, and 
$(\tilde q', h) \in \Lambda_{\rb}' \times [0, h_0]$ is not in $\WF'(F)$ iff 
$F$ can be written modulo $(hxx')^\infty \CI(\MMb)$ using local oscillatory integral representations
with symbols that vanish in a neighbourhood of the fibre in their domain corresponding to $(\tilde q, h)$. These components of $\WF'(F)$ will be denoted $\WF'_{\mf}(F)$, $\WF'_{\lb}(F)$, $\WF'_{\bfc}(F)$  and $\WF'_{\rb}(F)$, respectively. 

If $F \in I^m(\Lambda)$, then  $F \in (h x x')^\infty \CI(M^2)$ iff $\WF'(F)$ is empty. 
Also note that if $\WF'_{*}(F)$ is empty, then $\partial_{*} \Lambda'$ is disjoint from $\WF'_{\bfc}(F)$, but the converse need not hold: if the kernel of $F$ is supported away from $\mf$ then certainly $\WF'_{\mf}(F)$ will be empty, but $\WF'_{*}(F)$ need not be. 

Particular examples of Legendre distributions on $\bX$ are the kernels of semiclassical scattering pseudodifferential operators $Q$ of differential order $-\infty$ with compact operator wavefront set. In the case of such a pseudodifferential\footnote{Throughout this section we deal with semiclassical scattering pseudodifferential operators. The words `semiclassical scattering' will usually be omitted.} 
 operator, the Legendre submanifold $\Lambda$ is a compact  subset of $\Nsfstar \diagb$, defined in \eqref{Nsfstar-diagb},  and the components $\Lambda_{\lb} \cup \Lambda'_{\rb}$ are empty. Thus in this case we may (and will) identify the microlocal support $\WF'_{\mf}(Q)$ with a compact subset of  $\Tscstar M$, and $\WF'_{\bfc}(Q)$ may be identified with a compact subset of $\Tscstar_{\partial M} M \times [0, h_0)$. 

In the next lemma, $\pi_L$ and $\pi_R$ denote the maps defined in either \eqref{piLpiR} or \eqref{piLpiRmfbfc}, as the case may be. 

\begin{lem}\label{QFQ'-microsupport-he}
Suppose that $F$ is a Legendre distribution on $\bX$ and $Q$ is a  semiclassical scattering pseudodifferential operator. Assume that $F \in I^{m; r_{\bfc}, r_{\lb}, r_{\rb}}(\bX, \Lambda; \sfOh)$ is associated to a compact Legendre submanifold $\Lambda$ and that $Q$ is of differential order $-\infty$ and semiclassical order $0$, with compact operator wavefront set. Then $Q F$ is also a Legendre distribution in the space $I^{m; r_{\bfc}, r_{\lb}, r_{\rb}}(\bX, \Lambda; \sfOh)$ and we have
\begin{equation}\begin{gathered}
 \WF'_{\mf}(QF) \subset  \pi_L^{-1} \WF'_{\mf}(Q)  \cap \WF'_{\mf}(F) \\ 
\WF'_{\bfc}(QF) \subset  \pi_L^{-1} \WF'_{\bfc}(Q)  \cap \WF'_{\bfc}(F)   \\ 
\WF'_{\lb}(QF) \subset \WF'_{\bfc}(Q) \cap \WF'_{\lb}(F) \\
 \WF'_{\rb}(QF) \subset \WF'_{\rb}(F) .
\label{WFQF-he}\end{gathered}\end{equation}
Moreover, if $Q$ is microlocally equal to the identity on $\pi_L(\WF'_{\mf}(F))$, $\pi_L(\WF'_{\bfc}(F))$ and $\WF'_{\lb}(F)$, then $Q F  - F \in I^{\infty, \infty, \infty, r_{\rb}}(\bX, \Lambda; \sfOh)$, i.e. vanishes to infinite order at $\mf$, $\lb$ and $\bfc$. 
\end{lem}

We omit the proof, as it is essentially identical to that of Lemma~\ref{QFQ'-microsupport}. 
There is of course a corresponding theorem for composition in the other order, which is obtained by taking the adjoint of the lemma above. Combining the two we obtain

\begin{cor}\label{QFQ'-cor-he}
Suppose that $F$ and $Q, Q'$ are as above. Then 
\begin{equation}\begin{gathered}
 \WF'_{\mf}(QFQ') \subset  \pi_L^{-1} \WF'_{\mf}(Q)  \cap \pi_R^{-1} \WF'_{\mf}(Q') \cap \WF'_{\mf}(F)   \\ 
 \WF'_{\bfc}(QFQ') \subset  \pi_L^{-1} \WF'_{\bfc}(Q)  \cap \pi_R^{-1} \WF'_{\bfc}(Q') \cap \WF'_{\bfc}(F)   \\ 
 \WF'_{\lb}(QFQ') \subset \WF'_{\bfc}(Q) \cap \WF'_{\lb}(F) \\
 \WF'_{\rb}(QFQ') \subset \WF'_{\bfc}(Q') \cap  \WF'_{\rb}(F) .
\label{WFQFQ'-he}\end{gathered}\end{equation}
\end{cor}

A similar result holds if $F$ is associated to a Legendre conic pair rather than a single Legendre submanifold. 

\begin{lem}\label{QFQ'-microsupport-conic-he}
(i) Suppose that $F \in I^{m, p; r_{\bfc}, r_{\lb}, r_{\rb}}(\bX,(\Lambda, \Lambda^\sharp); \sfOh)$ is a Legendre distribution on $\bX$ associated to a conic Legendrian pair $(\Lambda, \Lambda^\sharp)$, 
and suppose that $Q$ is a  pseudodifferential operator such that $Q$ is microlocally equal to the identity operator near $\pi_L(\Lambda \cup \Lambda^\sharp)$. Then $Q F  - F \in I^{\infty, \infty; \infty, \infty,  r_{\rb}}(\bX, (\Lambda, \Lambda^\sharp), \sfOh)$, so vanishes to infinite order at $\mf$, $\lb$ and $\bfc$. Similarly, if $Q'$ is microlocally equal to the identity operator near $\pi_R(\Lambda \cup \Lambda^\sharp)$, then 
$F Q' - F \in I^{\infty, \infty; \infty,  r_{\lb}, \infty}(\bX, (\Lambda, \Lambda^\sharp), \sfOh)$ vanishes to infinite order at $\mf$, $\bfc$ and $\rb$.

(ii) Suppose that $F$ is as above,  a Legendre distribution on $\MMb$ associated to a conic Legendrian pair $(\Lambda, \Lambda^\sharp)$ of order $(m,p; r_{\bfc}, r_{\lb}, r_{\rb})$,  and suppose that $Q$, $Q'$ are  pseudodifferential operators. If 
\begin{equation}
\pi_L^{-1} \WF'_{\bfc}(Q) \cap \pi_R^{-1} \WF'_{\bfc}(Q') \cap \Lambda^\sharp = \emptyset,
\label{WF'QFQ'-conic-mf}\end{equation}
 then $Q F Q' \in I^{m; r_{\bfc}, r_{\lb}, r_{\rb}}(\MMb, \Lambda; \Omega)$; in particular, $\WF'_{\bfc}(QFQ')$ is disjoint from $(\Lambda^\sharp)'$.  \end{lem}
 
 We omit the proof, which is a straightforward modification of the arguments in Section~\ref{ows}. 

\subsection{Pointwise estimates on Legendre distributions}

Now we give a pointwise estimate on Legendre distributions of a particular type. First we begin with the trivial case.

\begin{prop}\label{triv-1-he} Let $\Lambda \subset \Tscstar_{\mf}(\bX)$ be a Legendre distribution that projects diffeomorphically to $\mf$. Suppose that $u \in I^{m,r_{\bfc}, r_{\lb},r_{\rb}}(\bX, \Lambda)$ with $$m=n/2  - l, \quad r_{\bfc} = -n/2  - \alpha, \quad r_{\lb} = r_{\rb} =  - \alpha.$$ Then, as a multiple of the half-density $|dg dg' d\lambda|^{1/2}$, 
we have a pointwise estimate
$$
|u| \leq C\lambda^l (x^{-1} + (x')^{-1})^{\alpha}.
$$
\end{prop}

Generalizing Proposition~\ref{Leg-bd-1} to the case of 
$\bX = \MMb \times [0, h_0]$  is straightforward.

\begin{prop}\label{Leg-bd-he} Let $\Lambda$ be a Legendrian submanifold of $\Tsfstar[\mf]\bX$. Assume that $\Lambda$ intersects $\Nsfstar \diagb$, defined in \eqref{Nsfstar-diagb}, at $G = \Lambda \cap \Nsfstar \diagb$ which is codimension 1 in $\Lambda$ and transversal to the boundary at $\bfc$, and that $d\pi$ has full rank on $\Lambda \setminus G$, while $\pi |_G$ is a fibration $G \to \diagb$ with $(n-1)$-dimensional fibres, with condition  \eqref{det-ass} holding at $G$. 

Assume that $u \in I^{m,r_{\bfc}, r_{\lb},r_{\rb}}(\bX, \Lambda; \sfOh)$, with $m,r_{\bfc}, r_{\lb},r_{\rb}$ as in Lemma~\ref{triv-1-he} and that the full symbol of $u$ vanishes to order $(n-1)/2 + \alpha$ both at $G \subset \Lambda$ and at
$\partial_{\bfc} G \times [0, h_0] \subset \partial_{\bfc} \Lambda \times [0, h_0]$. Then, as a multiple of the half-density $|dg dg' d\lambda|^{1/2}$, we have a pointwise estimate
\begin{equation}
|u| \leq 
C\lambda^{l-\alpha} (1 + \lambda d(z,z'))^{\alpha}.
\end{equation}
\end{prop}

\begin{proof} 
First consider $u$ on a neighbourhood of $\bX$ disjoint from $\diagb$. In that case, the result follows from Proposition~\ref{triv-1-he}. 

Next consider $u$ near $\diagb$, but away from $\bfc$. Then if $u$ is microlocally trivial at $\Nsfstar \diagb$, the result  follows from Proposition~\ref{triv-1-he}. If not, then the geometry is the same as that considered in Proposition~\ref{Leg-bd-1} (with $\rho$ replaced by $h$; also note that the estimate in Proposition~\ref{Leg-bd-1} is respect to the half-density $\lambda^n |dg dg' d\lambda|^{1/2}$), and the result follows from that Proposition. 

So we are reduced to the case where we are microlocally close to $\Lambda \cap \partial_{\bfc} \Nsfstar \diagb = \partial_{\bfc} G$. Let 
 $q \in \partial_{\bfc} G$. In a neighbourhood of $\partial_{\bfc}  \diagb$, we have coordinates $(x, y, w)$, where $w = (y-y', \sigma - 1)$ as before. In terms of these we can write points in $\Tsfstar_{\mf} \bX$ in the form 
 $$
 \kappa \cdot \frac{dw}{xh} + \mu \cdot \frac{dy}{xh} + \tau \cdot \frac{dx}{xh} + \nu d\big(\frac1{xh}\big),
 $$
 and this defines local coordinates $(x, y, w; \tau, \mu, \kappa, \nu)$ on $\Tsfstar_{\mf} \bX$. Then, contracting the canonical one-form with $x h^2 \partial_h$ and restricting to $\Tsfstar[\mf] \bX$ gives the contact form on $\Tsfstar[\mf] \bX$, which in these coordinates takes the form 
\begin{equation}
d\nu - \tau dx - \mu \cdot dy - \kappa \cdot dw.
\label{cf-mf}\end{equation} 
 
 Using the transversality of $\Lambda$ to $\Tsfstar[\bfc \cap \mf] \bX$ we see, as in the proof of Proposition~\ref{Leg-bd-1} that $(x, y, w_1, \mubbar)$ form coordinates on $\Lambda$. Then as in the proof of Proposition~\ref{Leg-bd-1}, we can write the remaining coordinates as functions of $(x, y, w_1, \mubbar)$ on $\Lambda$:
$$\begin{aligned}
\wbar_i &= W_i(x, y, w_1, \mubbar), \quad i = 2 \dots n, \\
\mu_i &= M_i(x, y, w_1, \mubbar), \\
\kappa_1 &= K(x, y, w_1, \mubbar), \qquad \qquad \qquad \text{ on } \Lambda. \\
\nu &= N(x, y, w_1, \mubbar), \\
\tau &= T(x, y, w_1, \mubbar)
\end{aligned}$$
In the same way as before, we find that 
$$
\tilde \Phi(x, y, w, v) = \sum_{j=2}^n \big(\wbar_j - W_j(x, y, w_1, v) \big)v_j + N(x, y, w, v), \quad v = (v_2, \dots, v_n),
$$
parametrizes $\Lambda$ locally, and has the properties that  $\tilde\Phi = O(w_1)$ when $d_v \tilde\Phi = 0$, and $\tilde\Phi = \Phi + O(x)$ where $\Phi$ is precisely as in the proof of Proposition~\ref{Leg-bd-1}. We can then follow the proof given there, where \eqref{toest} is replaced by 
\begin{equation}
x^{-(n-1)/2 - \alpha} \lambda^{(n-1)/2 + k} \int e^{i\tilde\Phi(x, y, w, v)/xh} \tilde a(x, y, w_1, v, h) \, dv
\end{equation}
in which the function $\tilde a$ vanishes to order $(n-1)/2 + \alpha$ at $x = 0$ and at $w_1 = 0$. In effect we have replaced the large parameter $1/x$ in the phase of \eqref{toest} by $1/xh$, while $x$  plays the role of a smooth parameter. 

The rest of the argument is parallel to the proof of Proposition~\ref{Leg-bd-1}. We first note that the estimate is trivial when $|w_1| \leq x h$. Assuming then that $|w_1| \geq x h$, we make the change of variables \eqref{theta-var}. By continuity, the matrix $A$ in \eqref{A^{-1}theta} remains nonsingular, and \eqref{A-derivs} remains valid, for small $x$. Hence, we can integrate by parts using the identity 
$$
e^{i\tilde\Phi/x} = \Big( \sum_k \frac{xh}{i\theta_j} A_{jk} \frac{\partial}{\partial \theta_k} \Big) e^{i\tilde\Phi/x}
$$
analogous to \eqref{ibp}.  

In the $\theta$ coordinates, we are trying to prove the estimate
\begin{equation*}
\Big| x^{-(n-1)/2 - \alpha} h^{-(n-1)/2 - l} \int_{\RR^{n-1}} \ybar^\alpha e^{i\tilde\Phi(x,y, w,\theta)/xh} \tilde a_0(x, y, w_1, \theta) \, d\theta \Big| \leq C h^{-l} \big( \frac{\ybar}{x}\big)^{\alpha},
\end{equation*}
since when $|w| \geq xh$, 
$$
\frac{|w|}{xh} \sim \lambda d(z,z') \sim 1 + \lambda d(z,z')
$$
(and recall that $|w| \sim |w_1|$ locally). As before, 
 the $\ybar^{(n-1)/2}$ factor was absorbed as a Jacobian factor, and $\tilde a$ is again smooth. This estimate is equivalent to a uniform bound on \begin{equation}
\Big| (xh)^{-(n-1)/2} \int_{\RR^{n-1}} e^{i\tilde\Phi(x,y, w,\theta)/x} \tilde a_0(x, y, w_1, \theta) \, d\theta \Big| .
\label{uniform-bd-he}\end{equation}
We introduce a modified partition of unity in $(x, \theta)$-space, $1 = \chi_0 + \sum_{j=1}^{n-1} \chi_j$, where $\chi_0$ is a compactly supported function of $\theta/\sqrt{xh}$, and $\chi_j$ is supported where $|\theta|\geq  \sqrt{xh}$, and where $\theta_j \geq  |\theta|/(n-1)$, with derivatives estimated by 
\begin{equation}
| \nabla_\theta^{(k)} \chi_k | \leq C (xh)^{-k/2}.
\label{chi-derivs-he}\end{equation}
Then the rest of the argument proceeds just as before, leading to \eqref{uniform-bd-he}. 
\end{proof}

\subsection{Geometry of the Legendre submanifold $L$}
We prove results analogous to Lemma~\ref{L-geometry-2} and Lemma~\ref{L-geometry-1}. First, we define 
$$
G = \{ q \in \Nsfstar \diagb \mid \sigma(h^2 \Delta_g)(q) = 1 \},
$$
where $\sigma$ is the semiclassical principal symbol. This is an $S^{n-1}$-bundle over $\diagb$. 

\begin{lem}\label{L-geometry-he-1} The Legendre submanifold $L$ introduced in Section~\ref{sec:bX} intersects $\Nsfstar \diagb$  
 cleanly at $G$, and the projection $\pi : L \to \mf$ satisfies \eqref{det-ass}. 
 \end{lem}
 
 \begin{proof} This is proved just as for Lemma~\ref{L-geometry-2}. 
As shown in \cite{HW},  $L$ can be obtained as the flowout from $G$ by a vector field $V_l$, which is obtained from the Hamilton vector field of $\Delta_g - \lambda^2$ by dividing by boundary defining function factors (see \cite[Section 11]{HW}), so that it becomes smooth up to the boundary of $\Tsfstar \bX$. This vector field takes the form  \eqref{Vl} up to $O(x)$ near $\bfc$, and repeating the argument
below \eqref{Vl} with $x$ as a smooth parameter  establishes the lemma in a neighbourhood of $\partial_{\bfc} G$, i.e. for $x + x' \leq \epsilon$ for some small $\epsilon > 0$. 

Away from $\bfc$, we can use coordinates $(z,z')$ on $\mf$, and 
writing points in $\Tsfstar[\mf] \bX$ in the form
$$
\zeta \cdot \frac{dz}{h} + \zeta' \cdot \frac{dz'}{h} + \tau d\big( \frac1{h} \big)
$$
defines fibre coordinates $(\zeta, \zeta', \tau)$ on $\Tsfstar[\mf] \bX$. 
In terms of these coordinates, we have 
\begin{equation}
V_l = g^{ij}(z) \zeta_i \frac{\partial}{\partial z^j} - \frac1{2} \frac{\partial g^{ij}(z)}{\partial z_k} \zeta_i \zeta_j \frac{\partial}{\partial \zeta_k} + g^{ij}(z) \zeta_i \zeta_j \frac{\partial}{\partial \tau}.
\label{Vll}\end{equation}
We recognize the equations for $(z, \zeta)$ as equations for geodesic flow. Moreover, letting $|\zeta|_g = g^{ij}(z) \zeta_i \zeta_j$, we find that $ ({|\zeta|}^2_g)\dot {} = 0$ and $|\zeta|_g = 1$ on $G$, hence $|\zeta|_g = 1$ on $L$; similarly $|\zeta'|_g = 1$ on $L$. Finally, $\dot \tau = 1$ and $\tau = 0$ on $G$. It follows that near a point on $G$ where (say) $\zeta_1 \neq 0$, we can use coordinates $(\zetabar, z', \tau)$ as coordinates on $L$, where $\zetabar = (\zeta_2, \dots, \zeta_n)$, $\zbar = (z_2, \dots, z_n)$. We then find, from \eqref{Vll}, that
$$\begin{aligned}
z^1 &= (z')^1 + g^{ij} \zeta_j \tau + O(\tau^2), \\
z^i &= (z')^i + g^{ij} \zeta_j \tau + O(\tau^2), \quad i \geq 2, \\
\end{aligned}$$
and we see that near $G$, 
$$
\frac{\partial z^1}{\partial \tau} \neq 0, \quad 
\frac{\partial \zbar^i}{\partial \zetabar_j} = \tau g^{ij}, 
$$
which shows that $\det d\pi$, where $\pi$ is the map
$$
L \ni (\zetabar, z', \tau) \mapsto \Big(z^1(\zetabar, z', \tau), \zetabar(\zetabar, z', \tau), z' \Big),
$$
vanishes to order exactly $n-1$ at $G$. 
\end{proof}


\begin{lem}\label{L-geometry-he} 
(i) There exists $0 < \delta < 1$ and $\epsilon > 0$ such that the Legendre submanifold $L \subset \Tsfstar_{\mf} \bX$ projects diffeomorphically to the base $\mf$ locally near all points $(x,y, x',y',  \mu, \mu', \nu, \nu', \tau) \in L \setminus G$ such that $x + x' < 2\epsilon$ and $  |\nu + \nu'| < \delta$. 

(ii) For any $\epsilon > 0$ there exists $\iota > 0$ such that $L$ projects diffeomorphically to the base near all points $(z, z', \zeta, \zeta', \tau) \in L \setminus G$ such that $x + x' > \epsilon$ and $|\tau| < \iota$. 
\end{lem}

\begin{proof} (i) A topological argument shows that for sufficiently small $\epsilon$, depending on $\delta$, the subset of $L$ where 
$x + x' < 2\epsilon$ and $ |\nu + \nu'| < \delta$ is contained in a small neighbourhood of the set $G \cup T_+ \cup T_-$, where $T_\pm \subset \partial_{\bfc} L = \Lbf$ are as in \eqref{eq:sp-1c}. Lemma~\ref{L-geometry-he-1} shows that $L$ projects diffeomorphically to $\mf$ in a deleted neighbourhood of $G$. Near the sets $T_\pm$, we use Lemma~\ref{L-geometry-1} and the fact, proved in \cite{HW}, that $L$ is transverse to the boundary at $\bfc$ to show that $(y, y', \sigma, \rho_{\bfc})$ form coordinates locally near $T_\pm$ away from $G$. Here $\rho_{\bfc}$ is a boundary defining function for $\bfc$ and can be taken to be $x$ for $\sigma > 1$ or $x'$ for $\sigma < 1$. Therefore, $L$ projects diffeomorphically to $\mf$ locally near $T_\pm$ and away from $G$. 

(ii) The calculation above shows that if $\tau$ is small, then $d(z,z')$ is small and $|\zeta + \zeta'|$ is small, i.e. $(z, \zeta, z', \zeta', \tau)$ is close to $G$. So by taking $\iota$ sufficiently small, we restrict attention to a small neighbourhood of $G \cap \{ x + x' \geq \epsilon \}$. The result then follows directly from Lemma~\ref{L-geometry-he-1}.
\end{proof}

\begin{remark} In fact, we can take $\iota$ to be the injectivity radius of $M$. 
\end{remark}

Let $M'$ be the compact subset of $M^\circ$ given by $\{ x \geq \epsilon \}$, where $\epsilon$ is as in Lemma~\ref{L-geometry-he}, and let $\iota$ be the injectivity radius of $M$. For any $z_0 \in M'$, let $z$ denote the Riemannian normal coordinates centred at $z_0$, and $\zeta$ the corresponding dual coordinates. Define the quantity
\begin{multline*}
\eta = \inf_{z_0 \in M'} \min \big\{ |z-z'| + |\zeta - \zeta'| :  |z-z_0| \leq \iota/4, |z'-z_0| \leq \iota/4, \\ \gamma(0) = (z, \zeta),\  \gamma(t) = (z', \zeta'),\  t \geq \iota \big\}
\end{multline*}
where the minimum is taken over all geodesics $\gamma : \RR \to M^\circ$ which are arc-length parametrized. 

\begin{lem}\label{eta} The quantity $\eta$ is strictly positive.
\end{lem}

\begin{proof} We use the nontrapping assumption. This means that there is no geodesic $\gamma$ with $\gamma(0) = (z, \zeta) = \gamma(t)$, if $t > \iota$. Therefore, by compactness, the minimum for a fixed $z_0$ in the expression above is strictly positive. This minimum varies continuously with $z_0$ and therefore the inf over all $z_0$ in the compact set $M'$ is also strictly positive.
\end{proof}

\subsection{Proof of Theorem~\ref{main3}, part (B)}\label{(B)}

We now assemble our results to prove \eqref{spec-meas-j-1} for $\lambda \geq \lambda_0$, i.e. $h \leq h_0$, which by Proposition~\ref{QQ1} and Section~\ref{(A)} is sufficient to prove part (B) of Theorem~\ref{main3}.  

We now choose a partition of unity consisting of pseudodifferential operators. This is done similarly to the previous section. In particular, we will choose $Q_1$ to have microsupport disjoint from the characteristic variety of $h^2 \bH - 1$, while the others will have compact microsupport, that is, they  will be pseudodifferential operators of differential order $-\infty$.  In detail, we  choose $Q_1$ such that $\Id - Q_1$ is microlocally equal to the identity where $\sigma(h^2 \Delta_g)  \leq 3/2$, and microsupported where $\sigma(h^2\Delta_g) \leq 2$ (here $\sigma$ denotes the semiclassical principal symbol). Then, we claim that $dE_{\sqrt{\bH}}^{(j)}(\lambda)$ is in $ (hxx')^\infty \CI(M^2)$. To see this, we write 
\begin{multline*}
Q_1 dE_{\sqrt{\bH}}^{(j)}(\lambda) Q_1 = dE_{\sqrt{\bH}}^{(j)}(\lambda) - (\Id - Q_1)dE_{\sqrt{\bH}}^{(j)}(\lambda) \\  - dE_{\sqrt{\bH}}^{(j)}(\lambda)(\Id - Q_1)  + (\Id - Q_1) dE_{\sqrt{\bH}}^{(j)}(\lambda)   (\Id - Q_1)
\end{multline*}
and use Theorem~\ref{hesmLd} and the microlocal support estimates as in the discussion below \eqref{energy-cutoff} to show that $\WF'(dE_{\sqrt{\bH}}^{(j)}(\lambda))$ is empty.   This piece therefore is in $ (hxx')^\infty \CI(M^2)$, and trivially satisfies \eqref{spec-meas-j}.

We now further decompose $\Id - Q_1$, which has compact microsupport, into a sum of terms. We first choose a function $m \in \CI(\MMb)$ that is equal to $1$ in a neighbourhood of $\partial \MMb$ and supported where $x + x' < 2\epsilon$, where $\epsilon$ is as in Lemma~\ref{L-geometry-he}. 
Choosing $\delta$ as in Lemma~\ref{L-geometry-he}, we divide up the interval $[-2, 2]$ into $N-1$ intervals $B_i$ each of width $\leq  \delta/4$, and choose a decomposition $(\Id - Q_1) m = \sum_{i=2}^N Q_i$ where $Q_i$, and hence also $Q_i^*$, are supported on the set $x + x' < 2\epsilon$ and microsupported in the set $\{ \sigma(h^2 \Delta_g) \leq 2, \nu \in 2B_i \}$.  It follows that if $q' = (x,y, x',y',  \mu, \mu', \nu, \nu', \tau) \in L'$ is such that $\pi_L(q') \in \WF'_{\mf}(Q_i)$ and $\pi_R(q') \in \WF'_{\mf}(Q_i^*)$, then $|\nu - \nu'| \leq \delta/2$. Together with Theorem~\ref{hesmLd} and  Lemma~\ref{QFQ'-microsupport-conic-he}, this means that $Q_i dE_{\sqrt{\bH}}^{(j)}(\lambda)Q_i^*$ is a Legendrian distribution associated only to  $L$ and not to $L^\sharp$, since on $(L^\sharp)'$ we have $|\nu - \nu'| = 2 > \delta/2$. Then  Lemma~\ref{L-geometry-he-1} guarantees that on the microsupport of $Q_i dE_{\sqrt{\bH}}^{(j)}(\lambda)Q_i^*$, the projection $\pi$ to $\mf$ is either a diffeomorphism or satisfies the conditions of Proposition~\ref{Leg-bd-he}.

We finally decompose $(\Id - Q_1) (1-m) = \sum_{i=N+1}^{N+N'} Q_i$, where $Q_i$ is microsupported in a sufficiently small set so that
$\WF_{\mf}(Q_i)$ is a subset of 
\begin{equation}
\{ (z, \zeta) \mid |z-z_0| + |\zeta - \zeta_0| < \eta/2 \}
\label{WF-mf-condition}\end{equation}
for some $z_0 \in M' = \{ x \geq \epsilon \} \subset M^\circ$ and some 
$\zeta_0$ (where we use Riemannian normal coordinates as in  Lemma~\ref{eta}). By construction, then, if $q' = (z, z', \zeta, \zeta', \tau) \in \WF'_{\mf} (Q_i dE_{\sqrt{\bH}}^{(j)}(\lambda) Q_i^*)$, then we must have
$|z-z'| + |\zeta - \zeta'| < \eta$ from \eqref{WF-mf-condition}, and also
$\gamma(0) = (z, \zeta),  \gamma(t) = (z', \zeta')$ for some geodesic $\gamma$. From Lemma~\ref{eta} we conclude $t < \iota$, thus $\gamma$ is the short geodesic between $z$ and $z'$. Consequently, $\tau < \iota$ and by Lemma~\ref{L-geometry-he} either $L$ locally projects diffeomorphically to $\mf$, or $q' \in \Nscstar \diagb$. 

We next consider the symbol of $Q_i dE_{\sqrt{\bH}}^{(j)}(\lambda) Q_i^*$. As in the previous section, this symbol vanishes to order $j$ both at $G \subset \mf$ and at $\partial_{\bf} G \times [0, h_0] \subset \bfc$, due to the vanishing of the phase function $\tilde\Phi$ at $G$ when $d_v \tilde \Phi = 0$. 
Therefore, in all cases, $Q_i dE_{\sqrt{\bH}}^{(j)}(\lambda) Q_i^*$ satisfies the conditions of Proposition~\ref{Leg-bd-he} with $l = j$, and the required estimate \eqref{spec-meas-j} follows from this proposition. This completes the proof of \eqref{restr} for $\lambda_0 \leq \lambda < \infty$.


\section{Trapping results}

\subsection{Spectral projection estimates}\label{sec:spe}
In this section we study the Laplacian on a manifold $N$ with $C^\infty$ bounded geometry, in the sense that  the local injectivity radius $\iota(p)$, $p \in  N$ has a positive lower bound, say $\epsilon$; the metric $g_{ij}$,   expressed in normal coordinates in the ball of radius $\epsilon/2$ around any point $p$ is uniformly bounded in $C^\infty(B_r(0))$, as $p$ ranges over $N$; and the inverse metric $g^{ij}$ is uniformly bounded in sup norm.  (In fact, we only need this to be bounded in $C^k$ for some $k$ depending on dimension $n$, but $k$ tends to infinity as $n \to \infty$.) This implies that the distance function $d(q, q')$ satisfies the $n \times n$ Carleson-Sj\"olin condition (see \cite[Section 2.2]{Sogge}) uniformly over all $p \in N$ and $q, q' \in B(p, \epsilon/2)$ with $d(q, q') \geq \epsilon/4$. 

Then the following Sogge-type restriction theorem holds:

\begin{prop}\label{sp-proj-est} Let $N$ be a complete Riemannian manifold of dimension $n$ with $\CI$ bounded geometry. Then the Laplacian $\Delta_N$ on $N$ satisfies 
\begin{equation}
\big\| \indic_{[\lambda, \lambda + 1]}(\sqrt{\Delta_N}) \big\|_{L^q(N) \to L^{q'}(N)} \leq C \lambda^{n(1/q - 1/q') - 1}, \quad \lambda \geq 1. 
\label{Sogge}\end{equation}
\end{prop}

This result is quite likely well-known to experts, but to our knowledge such a result has not appeared in the literature, so we provide a sketch proof.

\begin{proof} We adapt Sogge's argument. Let $\epsilon$ be as above. We then  choose an nonzero even Schwartz function $\chi$ such that its Fourier transform $\hat \chi$ is nonnegative and supported in $[\epsilon/4, \epsilon/2]$. It follows that $\chi(0) > 0$, and by taking $\epsilon$ sufficiently small, we can arrange that $\Re \chi \geq c > 0$ on $[0,1]$. 

Now let $\chiev(\sigma) = \chi(\sigma - \lambda) + \chi(-\sigma - \lambda)$. This is an even function, and since $\chi$ is rapidly decreasing, for sufficiently large $\lambda$ we have
$$
\Re \chiev \geq \frac{c}{2} \text{ on } [\lambda, \lambda + 1].
$$
That is,
$$
(\Re \chiev)^2 - \frac{c^2}{8} = F_\lambda, \text{ where } F_\lambda \geq 0 \text{ on } [\lambda, \lambda + 1].
$$
Then for $f \in L^p$, 
\begin{equation*}\begin{gathered}
\frac{c^2}{8} \big\| \indic_{[\lambda, \lambda + 1]}(\sqrt{\Delta_N}) f \big\|_{L^2}^2 
= \big\| \indic_{[\lambda, \lambda + 1]}(\sqrt{\Delta_N}) \Big( 
(\Re \chiev(\sqrt{\Delta_N}))^2 -  F_\lambda(\sqrt{\Delta_N})
\Big) f \big\|_{L^2}^2 \\
= \Big\langle \indic_{[\lambda, \lambda + 1]}(\sqrt{\Delta_N})  
\Re \chiev(\sqrt{\Delta_N}) f, \indic_{[\lambda, \lambda + 1]}(\sqrt{\Delta_N}) \Re \chiev(\sqrt{\Delta_N}) f \Big\rangle \\ - 
\Big\langle F_\lambda(\sqrt{\Delta_N}) \indic_{[\lambda, \lambda + 1]}(\sqrt{\Delta_N}) f, \indic_{[\lambda, \lambda + 1]}(\sqrt{\Delta_N}) f \Big\rangle \\
\leq \big\| \Re \chiev(\sqrt{\Delta_N}) f \big\|_{L^2}^2  
\leq \big\| \chiev(\sqrt{\Delta_N}) f \big\|_{L^2}^2.
\end{gathered}\end{equation*}
So it is enough to estimate the operator norm of the operator $\chievl$  from $L^q$ to $L^2$. To do this  we express $\chievl$ in terms of the half-wave group $e^{it\sqrt{\Delta_N}}$:
\begin{equation}
\chievl = \frac1{\pi} \int e^{it \sqrt{\Delta_N}}  \widehat{\chiev}(t) \, dt .
\end{equation}
Since $\widehat{\chiev} = e^{-it\lambda} \hat \chi(t) + e^{it\lambda} \hat \chi(-t)$ is even in $t$, we can write this as
\begin{equation}
\chievl = \frac1{\pi} \int \cos t \sqrt{\Delta_N} \big(e^{-it\lambda} \hat \chi(t) + e^{it\lambda} \hat \chi(-t) \big)  \, dt .
\label{chilambda}\end{equation}
Using the fact that the kernel of $\cos t\sqrt{\Delta_N}$ is supported in $\mathcal{D}_t$ for any complete Riemannian manifold, we see that $\chievl$ is supported in $\mathcal{D}_{\epsilon/2}$. Moreover, the argument of Sogge shows that $\chievl$ maps any $f \in L^p(M)$ and supported in a ball of radius $\epsilon/2$ to $L^2(M)$ with a bound 
$$
\big\| \chievl f \big\|_2 \leq C \| f \|_p,
$$
where $C$  is uniform over $M$ due to the bounded geometry. 
We then choose a sequence of balls $B(x_i, \epsilon/2)$ that cover $M$, such that $B(x_i, \epsilon)$ have uniformly bounded overlap, i.e. such that $\sum_i \indic_{B(x_i, \epsilon)}$ is uniformly bounded. Then for any $f \in L^p(M)$, and using the continuous embedding from $l^p \to l^2$ for $1 \leq p < 2$, 
\begin{equation}\begin{gathered}
\big\| \chievl f \big\|_{2}^{2} \leq \sum_i \big\| \chievl f \big\|_{L^{2}(B(x_i, \epsilon/2))}^{2} \\
\leq C \lambda^{n(1/p - 1/2) - 1/2} \sum_i \big\|  f \big\|_{L^{p}(B(x_i, \epsilon))}^{2} \\
\leq C \lambda^{n(1/p - 1/2) - 1/2} \Big( \sum_i \big\|  f \big\|_{L^{p}(B(x_i, \epsilon))}^{p} \Big)^{2/p} \\
\leq C \lambda^{n(1/p - 1/2) - 1/2}  \| f \|_{L^p}^{2},
\end{gathered}\end{equation}
showing that $\chievl$ maps from $L^p(M)$ to $L^2(M)$ with a bound
$C\lambda^{n(1/p - 1/2) - 1/2}$. 
\end{proof}


\subsection{Spatially localized results for trapping  manifolds}\label{sec:slr}

Let us assume now that $M^\circ$ is asymptotically \emph{Euclidean} and has several ends $\mcE_1,\dots,\mcE_k$.
By an \emph{end} here we mean a  connected component $\mcE_i$  of  $\{x<2\eps\}$ where $x$ is a boundary defining function and $\eps>0$ is a small fixed number, so that $\mcE_i$ is diffeomorphic to $(r_i, \infty) \times S^{n-1}$ with a metric of the form $dr^2 + r^2 h(y, dy, 1/r)$, with $h$ smooth,  and such that the projection of the trapped set  to $M^\circ$ is disjoint from $\mcE_i$. 

\begin{prop}\label{restrictionchi} Assume $M^\circ$ is asymptotically Euclidean, possibly with several ends. 
Let $\chi\in C^\infty(M)$ be supported in $\{x<\eps\}$ and let $\bH$ be as in Theorem~\ref{main3}. Then one has 
\begin{equation}
||\chi dE_{\sqrt{\bH}}(\la)\chi ||_{L^p\to L^{p'}}\leq C \la^{n(1/p-1/p')-1} \text{ for } 1 < p \leq \frac{2(n+1)}{n+3}.
\label{lre}\end{equation}
\end{prop}
\begin{proof}
As in \cite{HV1}, we can write $dE_{\sqrt{\bH}}(\la)=(2\pi)^{-1} P(\la)P(\la)^*$, where $P(\la)$ is the Poisson operator associated to $\bH$. Hence one needs to get $L^p(M)\to L^2(\pl M)$ bounds
for $P(\la)^*\chi$. The Schwartz kernel of $P(\la)^*$ is given by 
\begin{equation}\label{Poissonkernel}
P^*(\la;y, z')= [{x}^{-\frac{n-1}{2}}e^{i\la/x}R(\la; x,y;z')]|_{x=0}.
\end{equation}
Let $\chi_1,\chi_2,\chi_3\in C^\infty(M)$  be supported in $\{x<2\eps\}$ and equal to $1$ in $\{x<\eps\}$, and 
$\chi_i\chi_j=\chi_j$ if $j<i$. Let $(M_i,g_i)$ be a non-trapping asymptotically Euclidean manifold with one unique end isometric to $\mcE_i$. The existence of such a manifold can be easily proved if one takes 
$\eps$ small enough. There is a natural identification $\iota_j:M_j\cap\{x<2\eps\}\to M\{x<2\eps\}$, and so functions supported in $\{x<2\eps\}$ can be considered as functions on $M$ or $\cup_j M_j$. To simplify notations, we shall implicitly use this identification in what follows, instead of writing $\iota_j^*,{\iota_j}_*$ .  Let $\bH_j = \Delta_{M_j} + V_j$, where $V_j$ is equal to $V$ in the identified region, such that $\bH_j$ satisfies the conditions of Theorem~\ref{main3} (which can always be achieved by making $V_j$ sufficiently positive in a compact set away from the identified region). 
For $\la\in\{z\in \cc; {\rm Im}(\la)>0\}$, we define $R_j(\la):=(\bH_j-\la^2)^{-1}$ the resolvent, and by \cite{HV2} the Schwartz kernel of this operator extends continuously to $\la\in\rr$ as a Legendre distribution. For $\la>0$ it corresponds to the outgoing resolvent while for $\la<0$ it is the incoming resolvent.  For what follows, we consider ${\rm Re}(\la)>0$ to deal with the outgoing case.  We have the following identities for ${\rm Im}(\la)>0$ 
\[\begin{gathered} 
(\bH_j-\la^2)\sum_j\chi_2R_j(\la)\chi_1=\chi_1+\sum_{j}[\bH_j,\chi_2]R_j(\la)\chi_1,\\
\sum_j\chi_2R_j(\la)\chi_3(\bH_j-\la^2)=\chi_2+\sum_j\chi_2R_j(\la)[\chi_3,\bH_j],
\end{gathered}\]
which can be also written as 
\[\begin{gathered} 
\sum_j\chi_2R_j(\la)\chi_1=R(\la)\chi_1+\sum_{j}R(\la)[\bH_j,\chi_2]R_j(\la)\chi_1,\\
\sum_j\chi_2R_j(\la)\chi_3=\chi_2R(\la)+\sum_j\chi_2R_j(\la)[\chi_3,\bH_j]R(\la).
\end{gathered}\]
Multiplying the second identity by $\chi_1$ on the right and combining with the first one, we deduce that
\begin{equation}\label{identityresolv}
\chi_2R(\la)\chi_1=\sum_j\chi_2R_j(\la)\chi_1+\sum_{i,j}\chi_2R_i(\la)[\chi_3,\bH]R(\la)[\bH,\chi_2]R_j(\la)\chi_1.
\end{equation} 
Since $R_j(\la),R(\la)$ extend to $\la\in\rr$ as operators mapping $C_0^\infty(M)$ to $C^\infty(M)$, \eqref{identityresolv}  also extends to $\la\in \rr$ as a map from $C_0^\infty(M)$ to $C^\infty$
(since $[\bH,\chi_i]$ is a compactly supported differential operator).
Now to obtain the Poisson operator $P(\la)^*$, we use \ref{Poissonkernel} and deduce from \eqref{identityresolv} that 
\begin{equation}\label{identitypoisson}
P(\la)^*\chi_1=\sum_jP_j(\la)^*\chi_1+\sum_{i,j}P_i^*(\la)[\chi_3,\bH]R(\la)[\bH,\chi_2]R_j(\la)\chi_1
\end{equation}
where $P_j(\la)^*$ is the adjoint of the Poisson operator for $\bH_j$ on $(M_j,g_j)$ (mapping to $\pl M$  by the natural identification of  $\pl M_i$ with $\pl M$). 
Since $\nabla \chi_2$ and $\nabla\chi_3$ are compactly supported, we can choose $\eta \in C_0^\infty(M^\circ)$, supported in $\{ x < 2\epsilon \}$, such that $\eta = 1$ on $\supp \nabla \chi_2 \cup \supp \nabla \chi_3$, and write \eqref{identitypoisson} in the form 
\begin{equation}\label{identitypoisson2}
P(\la)^*\chi_1=\sum_jP_j(\la)^*\chi_1+\sum_{i,j}P_i^*(\la) \eta [\chi_3,\bH] \eta R(\la) \eta [\bH,\chi_2] \eta R_j(\la)\chi_1. 
\end{equation}

In \cite[equation (1.5)]{CV}, Cardoso and Vodev prove the following $L^2$ estimate: if $\eta\in C_0^\infty(M)$ (resp. $\eta_j\in C_0^\infty(M_j)$) is supported in $\{x<2\eps\}$, then for $\eps$ small enough, there is $C>0$ such that for all $\la>1$
\begin{equation}\label{cardosovodev}
||\eta R(\la) \eta||_{H^{-1}\to H^{1}}\leq C\la, \quad(\textrm{resp. } ||\eta_j R_j(\la) \eta_j||_{H^{-1}\to H^{1}}\leq C\la).
\end{equation}
Since the spectral measure $dE_j(\la)$ for $\sqrt{\bH_j}$
on $(M_j,g_j)$ satisfies 
\[dE_j(\la)=\frac{\la}{\pi i}(R_j(\la)-R_j(-\la))=\frac1{2\pi}P_j(\la)P_j(\la)^*,\]
we deduce by the $TT^*$ argument and \eqref{cardosovodev} that
\begin{equation}\label{etajpj}
||\eta_j P_j(\la)||_{L^2(\pl M_j)\to L^2(M_j)}\leq C 
\end{equation}
if $\eta_j$ is as above. Now since $M_j$ is non-trapping, we also know from Theorem \ref{main3} and the $TT^*$ argument that for $p\in [1,2(n+1)/(n+3)]$
\begin{equation}\label{Pj^*}
||P_j(\la)^*\chi_1||_{L^p(M_j)\to L^2(\pl M_j)}\leq C\la^{n(\frac{1}{p}-\demi)-\demi}. 
\end{equation}
We now use the following 
\begin{lem}\label{resolvent-spmeasure} Assume that $M_j$ is asymptotically Euclidean and nontrapping. 
Let $\chi\in C^\infty(M_j)$ be equal to $1$ in $\{x<\eps\}$ and supported in $\{x<2\eps\}$ and let  $\eta\in C_0^\infty(M_j)$ be supported in $\{x<2\eps\}$ such that 
\begin{equation}
\inf \{ x \mid \exists \, (x, y) \in \supp \eta \} \geq \gamma \sup \{ x \mid \exists \, (x, y) \in \supp \chi \}
\label{supp-condition}\end{equation}
for some $\gamma > 1$; 
in particular, the distance between the support of $\eta$ and $\chi$ is positive. Then the following estimate holds for $1 < p \leq 2(n+1)/(n+3)$ and $\lambda \geq 1$:
\[ ||\eta R_j(\la) \chi||_{L^p(M_j)\to L^{p'}(M_j)}\leq \frac{C}{\la}  ||\eta \, dE_j(\la)\chi||_{L^p(M_j)\to L^{p'}(M_j)} + O(\la^{-\infty}). \]  
\end{lem} 
\begin{proof}
Recall that $R_j(\pm \lambda)$ is the sum of a pseudodifferential operator and of  Legendre distributions associated to the Legendre submanifolds $(\Nsfstar \diagb, L_\pm)$ and to $(L_\pm, L_\pm^\sharp)$. Since the distance between the support of $\eta$ and $\chi$ is positive, we see that $\eta R_j(\pm \lambda) \chi$
are, like $dE(\lambda)$, both Legendre distributions (conic pairs) associated to $(L, L^\sharp)$ with disjoint microlocal support; indeed, the nontrapping assumption implies that $L_+$ and $L_-$ intersect only at $G$ which is contained in $\Nsfstar \diagb$, while $L_+^\sharp$ and $L_-^\sharp$ are disjoint. We claim that we can choose a microlocal partition of unity, $\sum_{i=1}^N Q_i = \Id$, where $Q_i$ are semiclassical scattering pseudodifferential operators, such that for each pair $(i, k)$, either $Q_i \eta R_j( \lambda) \chi Q_k$ or $Q_i \eta R_j(- \lambda) \chi Q_k$ is microlocally trivial. This does not quite follow from the disjointness of the microlocal supports of $\eta R_j(\pm \lambda) \chi$; we must also check that at $T_\pm$, there are no points $(y, y', \sigma, \mu, \mu', \nu, \nu'), (y, y', \sigma^*, \mu, \mu', \nu, \nu') \in \Tsfstar_{\bfc} \bX$, differing only in the $\sigma$ coordinate, such that the first point is in $\WF'_{\bfc}(\eta R_j( \lambda) \chi)$ and the second point is in $\WF'_{\bfc}(\eta R_j(- \lambda) \chi)$ (cf. Remark~\ref{sigma-behaviour}). This follows from \eqref{eq:sp-1c}; in fact, the coordinates $(\nu, \nu')$ determine $\sigma$ except on the sets $T_\pm$. However, on $T_\pm$, we find that $(y, y', \sigma, \mu=0, \mu'=0, \nu = \pm 1, \nu' = \mp 1)$ is in $L_+$ iff $\sigma  \leq 1 $ and $\nu = 1$, or $\sigma \geq 1$ and $\nu = -1$, while it is in $L_-$ iff $\sigma  \leq 1 $ and $\nu = -1$, or $\sigma \geq 1$ and $\nu = 1$. But condition \eqref{supp-condition} implies that $\sigma \geq \gamma > 1$ on the support of the kernel of $\eta R_j(\pm \lambda) \chi$, so we see that indeed it is not possible to have $(y, y', \sigma, \mu, \mu', \nu, \nu') \in \WF'_{\bfc}(\eta R_j( \lambda) \chi)$ and $(y, y', \sigma^*, \mu, \mu', \nu, \nu') \in \WF'_{\bfc}(\eta R_j(- \lambda) \chi)$. 

Now let $\mathcal{N}$ be the set of pairs $(i,k)$, with $1 \leq i,k \leq N$, such that $Q_i \eta R_j( \lambda) \chi Q_k$ is not microlocally trivial. This means that if $(i,k) \in \mathcal{N}$, then $Q_i \eta R_j( -\lambda) \chi Q_k$ \emph{is} microlocally trivial. Let us also observe that as the $Q_i$ are uniformly bounded as operators $L^2 \to L^2$, and as they are Calder\'on-Zygmund operators in a uniform sense as $h \to 0$, then they are uniformly bounded as operators $L^p \to L^p$ for $1 < p < \infty$. Therefore we can compute: 
\begin{equation}\begin{gathered}
\| \eta R_j(\la) \chi \|_{L^p(M_j) \to L^2(M_j)} \leq
\sum_{i,k = 1}^N \| Q_i \eta R_j(\la) \chi Q_k \|_{L^p(M_j) \to L^2(M_j)} \\
= \sum_{(i,k) \in \mathcal{N}} \| Q_i \eta R_j(\la) \chi Q_k \|_{L^p(M_j) \to L^2(M_j)} + O(\la^{-\infty}) \\
= \sum_{(i,k) \in \mathcal{N}} \| Q_i \eta (R_j(\la) - R_j(-\la)) \chi Q_k \|_{L^p(M_j) \to L^2(M_j)} + O(\la^{-\infty}) \\
= \frac1{2\pi  \lambda} \sum_{(i,k) \in \mathcal{N}} \| Q_i \eta \,  dE_j(\la) \chi Q_k\|_{L^p(M_j) \to L^2(M_j)} + O(\la^{-\infty}) \\
\leq \frac{CN^2}{\lambda}  \|  \eta \, dE_j(\la) \chi  \|_{L^p(M_j) \to L^2(M_j)} + O(\la^{-\infty}) ,
\end{gathered}\end{equation}
proving the lemma. 
\end{proof}

Since  $\eta dE_j(\la)\chi=\eta P_j(\la)P_j(\la)^*\chi$, we deduce from Lemma~\ref{resolvent-spmeasure} and equations \eqref{etajpj} and \eqref{Pj^*} that 
\begin{equation}
||\eta R_j(\la) \chi||_{L^p(M_j)\to L^2(M_j)}\leq C\la^{n(\frac{1}{p}-\demi)-\demi - 1}, \quad \la \geq 1. 
\label{etarjchi}\end{equation}
Now we can analyze the boundedness of the right-hand term of \eqref{identitypoisson2} as follows: $\eta R_j(\la) \chi$ maps $L^p(M_j)\to L^2(M_j)$ with norm $C\la^{n(\frac{1}{p}-\demi)-\demi - 1}$ by \eqref{etarjchi}; 
$[\bH, \chi_2]$ maps $L^2(M_j)$ to $H^{-1}(M)$ with norm independent of $\la$; $\eta R(\la) \eta$ maps $H^{-1}(M)$ to $H^1(M)$ with norm $C\la$ by \eqref{cardosovodev}; $[\chi_3, \bH]$ maps $H^1(M_j)$ to $L^2(M)$ with norm independent of $\la$; and $P_i^*(\la) \eta$ maps $L^2(M)$ to $L^2(M)$ with uniformly bounded norm by \eqref{Pj^*}. This  concludes the proof of Proposition~\ref{restrictionchi}. 
\end{proof}

\begin{remark} Observe that we missed the endpoint $p=1$ due to our use of Calder\'on-Zygmund theory. In the case that $M$ is exactly Euclidean for $x < 2\epsilon$ we can take $M_j$ to be flat Euclidean space and then it is straightforward to check that $\eta R_j(\la) \chi$ is bounded $L^1(M_j) \to L^2(M_j)$ with norm $O(\la^{(n-3)/2})$, which gives us Proposition~\ref{restrictionchi} for $p=1$  in this case. 
\end{remark}

In the paper \cite{SeeS} by Seeger-Sogge, spectral multiplier estimates are proved for compact manifolds for the same exponents as in Theorem~\ref{bori}. This was done using Sogge's discrete $L^2$ restriction theorem, i.e. Proposition~\ref{sp-proj-est}. One may suspect that, since spectral multiplier estimates can be proved in the compact case, and since we have localized restriction estimates outside the trapped sets, that one should be able to prove spectral multiplier estimates on asymptotically conic manifolds unconditionally, i.e. without any nontrapping assumption. We have not been able to prove this, however, but have the following localized results:  

\begin{prop}\label{nontrappingpart}
Let $M^\circ$ be a manifold with Euclidean ends, and let $p \in [1, 2(n+1)/(n+3)]$. 
Let $\bH$ be as in Proposition~\ref{main3}, let $\chi$ be a cutoff function as in Proposition \ref{restrictionchi}, let $F$ be a multiplier satisfying the assumption of Theorem \ref{bori}, i.e.
$F\in H^s$ for some $s>\max(n(\frac{1}{p}-\demi),\demi)$. Then we have
\[\sup_{\alpha>0}|| F(\alpha\sqrt{\bH})\chi ||_{p\to p}\leq C ||F||_{H^s}.\]
\end{prop}

This is proved by following the proof of Theorem~\ref{bori}, using \eqref{lre} in place of \eqref{re}.

\begin{prop}\label{compactpart} 
Let $\omega\in C_c^\infty(M^\circ)$ be compactly supported  and let  $\bH$ and $F$ be as above. Then the following estimate holds:
\[ \sup_{\alpha>0}|| \omega F(\alpha\sqrt{\bH})||_{L^p\to L^p}\leq  ||F||_{H^s}.\]
\end{prop}

This is proved by following the method of Seeger-Sogge \cite{SeeS}, using the compact support of $\omega$ to obtain the embedding from $L^2$ to $L^p$ as in \cite[Equation (3.11)]{SeeS}. 

\subsection{Examples with elliptic trapping}\label{trappingex}

Here we show that the restriction estimate at high frequency generically fails for asymptotically conic manifolds with  
elliptic closed geodesics.   Indeed, it has been proved by Babich-Lazutkin \cite{BaLa} and Ralston \cite{Ra}
that if there exists a closed geodesic $\gamma$ in $M$ such that  
the eigenvalues of the linearized Poincar\'e map of $\gamma$ are of modulus $1$ and are not roots of unity, then
there exists  a sequence of quasimodes $u_j\in C_0^\infty(K)$ with  
$K$ a fixed compact set containing the geodesic, a sequence of positive real numbers $\la_j\to \infty$ 
such that for all $N>0$ there is $C_N>0$ such that  
\begin{equation}\label{quasimode}
\| u_j \|_{L^2} = 1, 
\quad ||(\Delta_g-\la_j^2)u_j||_{L^2}\leq C_N\la_j^{-N}.
\end{equation}
We show
\begin{prop}\label{applquasimode}
Assume that $(M,g)$ is an asymptotically conic manifold with an elliptic closed geodesic  
such that   the eigenvalues of the linearized Poincar\'e map of $\gamma$ are of modulus $1$ and are not roots of unity.
Then for all $p\in [1,2)$ and $M\geq 0$  
the spectral measure $dE_{\sqrt{\Delta_g}}(\lambda)$ does {\bf not} satisfy the following restriction estimate  
\[\exists C>0, \ \exists \lambda_0>0, \ \forall \la\geq \lambda_0,\ ||dE_{\sqrt{\Delta_g}}(\la)||_{L^p\to L^{p'}}\leq C\la^{M}.\] 
\end{prop}
\begin{proof}: Let $u_j$ be the quasimodes above. Then the inequality
$$
||(\Delta_g-\la_j^2)u_j||_{L^2}\leq C_N\la_j^{-N}
$$
implies that 
\[|| \indic_{\rr\setminus [\la_j^2-2C_N\la_j^{-N},\la_j^2+2C_N\la_j^{-N}]}(\Delta_{g})u_j||_{L^2} \leq 1/2 \]
since $\| (\Delta_g - \lambda_j^2) v \| \geq c \| v \|$ if $v$ is in the range of the spectral  projector $\indic_{\RR \setminus [\lambda_j^2 - c, \lambda_j^2 + c]}(\Delta_g)$. Therefore 
\begin{equation}
|| \indic_{ [\la_j^2-2C_N\la_j^{-N},\la_j^2+2C_N\la_j^{-N}]}(\Delta_g)u_j||_{L^2}\geq \frac{\sqrt{3}}{2},
\end{equation}
and using the fact that $\indic_{ [\la_j^2-2C_N\la_j^{-N},\la_j^2+2C_N\la_j^{-N}]}(\Delta_g)$ is a projection, 
\begin{equation}\label{spprojector}
\big\langle u_j,  \indic_{ [\la_j^2-2C_N\la_j^{-N},\la_j^2+2C_N\la_j^{-N}]}(\Delta_g)u_j \big\rangle \geq \frac{3}{4}.
\end{equation}
This implies that for large enough $\lambda$ we have 
\begin{equation}\label{spprojector1}
\big\langle u_j,  \indic_{ [\la_j-2C_N\la_j^{-N-1},\la_j+2C_N\la_j^{-N-1}]}(\sqrt{\Delta_g})u_j \big\rangle \geq \frac{3}{4}.
\end{equation}

Now assume that there exists $C$ such that 
$||dE_{\sqrt{\Delta_g}}(\la)||_{L^p\to L^{p'}}\leq C\la^{M}$. Then using the continuous embeddings from $L^2(K) \to L^p(K)$ and $L^{p'}(K)$ to $L^2(K)$, we see that there is $C'>0$ such that
\[ \big\langle u_j,  dE_{\sqrt{\Delta_g}}(\la)u_j \big\rangle \leq C'\la^M||u_j||_{L^2}\leq 2C'\la^M.\]  
By integrating this on the interval  $[\la_j-2C_N\la_j^{-N-1},\la_j+2C_N\la_j^{-N-1}]$, we contradict  \eqref{spprojector1} if $N+1$ is chosen larger than $M$ and $j$ is large enough.  
\end{proof} 
 
\begin{remark}\label{metricbottle}
In fact, one can construct examples where the spectral measure blows up exponentially with respect to the 
frequency $\la$. Consider a Riemannian manifold $(M,g)$ which is a connected sum of flat $\RR^n$ and a sphere $S^n$, so that it 
contains  an open set $S$ isometric to part of a round sphere $S^n$, namely
\[ S = \{ x=(x_1,x_2, \dots, x_{n+1})\in \rr^{n+1}; |x|=1, x_1^2 + x_2^2 > 1/4 \}\]
Consider the functions $u_N(x):=(x_1+ix_2)^N$ (as functions on $\RR^{n+1}$). These restrict to eigenfunctions on $S^{n}$  with corresponding eigenvalue $N(N+n-1)$ and with norm $||u_N||_{L^2}\sim cN^{-1/4}$ for some $c>0$ as $N\to \infty$.
Let $\chi\in C_0^\infty(S)$ be equal to $1$ on $S\cap \{ x_1^2 + x_2^2 \geq 1/2\}$ and extend it by $0$ on $M\setminus S$.  
The modified function $v_N=\chi u_N/||\chi u_N||_{L^2}$ satisfies
\[(\Delta_g-N(N+n-1))v_N= [\Delta_g,\chi]u_N/||\chi u_N||_{L^2}.\]
But since $|x_1+ix_2|<1/2$ on the support of $[\Delta_g,\chi]$ and since $||\chi u_N||>CN^{-1/4}$ for some $C>0$ when 
$N$ is large, we deduce that $(\Delta_g-N(N+n-1))v_N=O_{L^2}(e^{-\alpha N})$ for some $\alpha>0$. Applying the argument 
of Proposition \ref{applquasimode}, we deduce that  there exists 
$C>0,\beta>0$ and a sequence $\la_N\sim \sqrt{N(N+n-1)}$ such that
\[ || dE(\la_N)||_{L^p\to L^{p'}}\geq Ce^{\beta \la_N}.\]
\end{remark}


\section{Conclusion}

We conclude by mentioning several ways in which the investigations of this paper could be extended. 

Theorem~\ref{main3} is only stated for dimensions $n \geq 3$. This is because the proof relies on the analysis of \cite{GH1} and \cite{GHS}, which is only done for $n \geq 3$. It would be interesting to treat also the case $n=2$. The main difficulty in doing this is to write down a suitable inverse for the model operator at the $\zf$ face in the construction of \cite[Section 3]{GH1}, which is not invertible as an operator on $L^2(M)$ in two dimensions as it is in all higher dimensions. 

One could also extend Theorem~\ref{main3} by allowing potential functions which are $O(x^2)$ instead of only $O(x^3)$ at infinity, i.e. inverse-square decay near infinity. This should be relatively straightforward, because all the analysis has been done in the two papers cited above. For potentials of the form $V = V_0 x^2$ with $V_0$ strictly negative at $\partial M$, this would have the effect of changing the `numerology', i.e. the range of $p$ and the power of $\lambda$ in \eqref{restr}, for example. Here we preferred not to treat this case, in order not to complicate the statement of Theorem~\ref{main3}, but rather to keep the numerology as it is in the familiar setting of the classical Stein-Tomas theorem, and in Sogge's discrete $L^2$ restriction theorem. 

Another way to extend Theorem~\ref{main3} would be to allow operators $\bH$ with eigenvalues. In this case, we would consider the positive part $\indic_{(0, \infty)}(\bH)$ of the operator $\bH$. We expect such a generalization to be straightforward, as the analysis has been carried out in \cite{GH1}, \cite{GHS}, with the only complication being that $\indic_{(0, \infty)}(\bH)$ does not satisfy the finite speed propagation property \eqref{fsp}. 

We close by posing, as open problems, several possible generalizations that seem to be a little less straightforward:

\begin{itemize} 
\item Prove (or disprove) the restriction theorem for high energies in the presence of trapping, in the case that the trapped set is hyperbolic and the topological pressure assumption of \cite{NZ} and \cite{BGH} is satisfied. 

\item The Hardy-Littlewood-Sobolev theorem tells us that the resolvent of the Laplacian at zero energy on $\RR^n$ is bounded from $L^p(\RR^n)$ to $L^{p'}(\RR^n)$ when $n \geq 3$ and $p = 2n/(n+2)$; this holds true on any asymptotically conic manifold. 
Since this value of $p$ is in the range $[1, 2(n+1)/(n+3)]$, this suggests that the resolvent kernel $(\Delta - (\lambda \pm i0)^2)^{-1}$ on an asymptotically conic manifold should be bounded from $L^p(\RR^n)$ to $L^{p'}(\RR^n)$ when $p = 2n/(n+2)$. Prove (or disprove) this. 

\item Prove (or disprove) the spectral multiplier result for high energies in the trapping case, i.e. Propositions~\ref{nontrappingpart} and \ref{compactpart} without the cutoff functions. 
\end{itemize}

\end{document}